\documentclass[11pt]{amsart}

\usepackage{tikz}

\usepackage{mathrsfs}
\usepackage[colorlinks=true,urlcolor=blue, citecolor=red,linkcolor=blue,linktocpage,pdfpagelabels, bookmarksnumbered,bookmarksopen]{hyperref}
\usepackage[hyperpageref]{backref}
\usepackage{cleveref}
\usepackage{xcolor}
\usepackage{amsthm} 
\usepackage{latexsym,amsmath,amssymb}
\usepackage{accents}
\usepackage[colorinlistoftodos,prependcaption,textsize=tiny]{todonotes}
\usepackage{a4wide}
\usepackage{soul}
\usepackage{mathtools} 
\usepackage{xparse} 
\usepackage{enumitem}

\definecolor{indigo}{rgb}{0.29, 0.0, 0.51}
\definecolor{p1}{gray}{0.4}
\definecolor{p2}{gray}{0.6}
\definecolor{p3}{gray}{0.98}
\definecolor{p4}{gray}{0.8}
\definecolor{p5}{gray}{0.9}

plus 6pt minus 12pt
\def\eps{\varepsilon}

\def\Om{\Omega}

\def\vp{\varphi}
\def\id{{\rm Id\, }}
\def \a{\alpha}

\newcommand{\cO}{\mathcal{O}}

\def\B{{B}}
\def\C{{\mathcal C}}

\def\N{{\mathbb N}}

\def\S{{\mathbb S}}
\def\I{{\mathcal I}}

\def\p{{\partial}}
\renewcommand{\O}{{\Omega}}

\def\e{\varepsilon}

\newtheorem{theorem}{Theorem}
\newtheorem{claim}{Claim}
\newtheorem{lemma}[theorem]{Lemma}

\newtheorem{proposition}[theorem]{Proposition}

\newtheorem{remark}[theorem]{Remark}
\newtheorem{definition}[theorem]{Definition}

\newcommand{\loc}{\mathrm{loc}}

\def\dist{{\rm dist\,}}
\def\divv{{\rm div\,}}

\def\loc{{\rm loc}}
\def\tr{{\rm tr}}

\newcommand{\dif}{\,\mathrm{d}}
\newcommand{\dx}{\dif x}
\newcommand{\dy}{\dif y}

\newcommand{\dd}{\dif}

\newcommand{\R}{\mathbb{R}}
\newcommand{\bB}{\mathbb{B}}

\newcommand{\s}{\mathbb{S}}

\newcommand{\pl}{\partial}

\newcommand{\brac}[1]{\left (#1 \right )}

\newcommand{\Ep}{\bigwedge\nolimits}

\newcommand{\Jac}{{\rm Jac}}

\newcommand{\barint}{
\rule[.036in]{.12in}{.009in}\kern-.16in \displaystyle\int }

\newcommand{\barcal}{\mbox{$ \rule[.036in]{.11in}{.007in}\kern-.128in\int $}}

\def\mvint_#1{\mathchoice
          {\mathop{\vrule width 6pt height 3 pt depth -2.5pt
                  \kern -8pt \intop}\nolimits_{\kern -3pt #1}}%
          {\mathop{\vrule width 5pt height 3 pt depth -2.6pt
                  \kern -6pt \intop}\nolimits_{#1}}%
          {\mathop{\vrule width 5pt height 3 pt depth -2.6pt
                  \kern -6pt \intop}\nolimits_{#1}}%
          {\mathop{\vrule width 5pt height 3 pt depth -2.6pt
                  \kern -6pt \intop}\nolimits_{#1}}}

\numberwithin{theorem}{section} \numberwithin{equation}{section}

\newcommand{\lap}{\Delta }
\newcommand{\g}{\nabla }
\newcommand{\aleq}{\precsim}

\newcommand{\ust}{\underset}
\newcommand{\scal}[2]{\left\langle #1,#2 \right\rangle}
\newcommand{\pr}{\mathcal{P}_r}
\newcommand{\pra}{\mathcal{P}_{r,\alpha}}

\newcommand*\oline[1]{%
  \kern0.1em            
  \vbox{%
    \hrule height 0.5pt 
    \kern0.4ex          
    \hbox{%
      \kern-0.1em       
      $#1$
      \kern-0.1em       
    }
  }
  \kern0.1em            
}

\def\XXint#1#2#3{{\setbox0=\hbox{$#1{#2#3}{\int}$}
     \vcenter{\hbox{$#2#3$}}\kern-.5\wd0}}

\let\latexchi\chi
\makeatletter
\renewcommand\chi{\@ifnextchar_\sub@chi\latexchi}
\newcommand{\sub@chi}[2]{
  \@ifnextchar^{\subsup@chi{#2}}{\latexchi^{}_{#2}}%
}
\newcommand{\subsup@chi}[3]{
  \latexchi_{#1}^{#3}%
}
\makeatother

\usepackage{graphicx}
\newcommand\smallO{
  \mathchoice
    {{\scriptstyle\mathcal{O}}}
    {{\scriptstyle\mathcal{O}}}
    {{\scriptscriptstyle\mathcal{O}}}
    {\scalebox{.7}{$\scriptscriptstyle\mathcal{O}$}}
  }
  \DeclareMathOperator \VMO{\text{VMO}}

\title[Min-max free boundary \(n\)-harmonic maps]{Min-max \(n\)-harmonic maps of degree 1 with free-boundary into \(\mathbb{S}^{n-1}\) in almost round balls }

 \author{Dorian Martino}
\address[Dorian Martino]{
Department of Mathematics, ETH Zurich, Rämistrasse 101, 8092 Zurich, Switzerland}
\email{dorian.martino@math.ethz.ch}

 \author{Katarzyna Mazowiecka}
 \address[Katarzyna Mazowiecka]{
 Institute of Mathematics,%
 University of Warsaw,
 Banacha 2,
 02-097 Warszawa, Poland}
 \email{k.mazowiecka@mimuw.edu.pl}

 \author{R\'emy Rodiac}
\address[R. Rodiac]{Laboratoire J.A. Dieudonn\'e, Universit\'e C\^ote d'Azur, CNRS UNMR 7351,06108, Nice, France.}
\email{remy.rodiac@univ-cotedazur.fr}

\begin{document}

\begin{abstract}
Let \(n\geq 3\) and let \(\Omega \subset \R^n\) be a \(\C^1\) bounded domain which is diffeomorphic to a ball. We investigate here the problem of finding critical points of the \(n\)-energy in the space $\I=\{v\in W^{1,n}(\Omega,\R^n) ; \ |\tr_{|\partial \Omega}v|=1\}$. Maps in $\I$ have a well-defined topological degree on $\partial \Omega$ but this degree is not continuous for the weak convergence in \(W^{1,n}\). Hence finding critical points with prescribed degrees results in a problem of lack of compactness. We first prove that minimizers of the $n$-energy exist only when \(\Omega\) is a round ball and when the prescribed degree is \(-1,0\) or \(1\). We then develop a mountain pass approach for the $(n+\alpha)$-energies and study the convergence, when \(\alpha\) goes to zero, of the resulting critical points via a bubbling analysis. We exclude the existence of bubbles in the case where $\Omega$ is close to a ball by proving an energy gap result for free boundary $n$-harmonic maps from $\mathbb{B}^n$ to $\mathbb{B}^n$. We thus obtain the existence of critical points of the \(n\)-energy with prescribed degree \(1\) when \(\Omega\) is close to a ball.
\end{abstract}

\sloppy

\maketitle

\tableofcontents
\sloppy

\section{Introduction}

\subsection{Motivation}

The Riemann mapping theorem states that any simply connected open set of $\R^2$ which is not \(\R^2\) is conformally equivalent to the unit ball $\bB^2$. In \(\R^2 \simeq \mathbb{C}\), a domain is conformally equivalent to another domain if and only if there exists a biholomorphism (or a bi-antiholomorphism) that maps one into the other.  In higher dimension, the notions of conformal map and holomorphic map differ. The last notion is related to the theory of several complex variables whereas the former one means that the map  infinitesimally preserves angles. The Riemann mapping theorem does not hold in higher dimensions, neither with  the  notion of biholomorphic maps nor with the notion of conformal maps. Indeed, one can 
see that \(\mathbb{B}^4\) is not biholomorphic to \(\mathbb{B}^2\times \mathbb{B}^2\), \cite[Exercise 1.1.16]{H2005}, and thus there exist simply connected domains which are not biholomorphic. Moreover, if we are interested in conformality, Liouville's rigidity theorem shows that conformal maps in higher dimensions are precisely products of translations, dilations and inversions, see Theorem \ref{th:Liouville} below, hence the conformal version of the Riemann mapping theorem does not hold either in higher dimensions. 
The search for a suitable generalization of the Riemann mapping theorem to higher dimensions stimulated research in different areas of mathematics \cite{Chern_Ji_1996} and in particular in geometric function theory \cite{Iwaniec_Martin_2001}.
 In dimension 2, the theory of conformal mappings is related to the theory of harmonic functions and to the study of the Dirichlet energy. Riemann himself used the solvability of the Dirichlet problem for harmonic functions in the original proof of his uniformization theorem, see e.g.\ \cite[Chapter 8, p.217]{Stein_Shakarchi_2003}. The role of the Dirichlet energy, which is conformally invariant in dimension 2, is played by the \(n\)-energy, see Equation \eqref{eq:n-energy}, in dimension \(n\). Indeed the \(n\)-energy is also conformally invariant in dimension \(n\). In this paper we first prove that orientation-preserving conformal maps form \(\mathbb{B}^n\) to \(\mathbb{B}^n\) can be characterized as minimizers of the \(n\)-energy among maps with traces taking values in \(\mathbb{S}^{n-1}\) and with prescribed degree equal to \(1\) on \(\p \mathbb{B}^n=\mathbb{S}^{n-1}\). 
  We also prove that the \(n\)-energy does not admit any minimizer with prescribed degrees different from zero if the source domain \(\Omega\) is not a round ball. It is then natural to look for non-minimizing critical points of the \(n\)-energy with prescribed degree one when \(\Omega\) is not a ball. Our main result roughly states  that if \(\Omega\) is diffeomorphic to \(\mathbb{B}^n\) and is close enough to \(\mathbb{B}^n\), in a sense to be specified later,  then there exists a critical point of the \(n\)-energy with prescribed degree one on \(\p \Omega\). It is then tempting to consider this critical point as a generalization of a conformal map for domains which are not balls. However we do not prove that the critical point we have obtained is injective nor show any property of this critical point, except for its regularity \(\C^{1,\gamma}(\overline{\Omega},\R^n)\) for some \(0<\gamma<1\) obtained by an application of \cite[Proposition 3.1]{MazowieckaRodiacSchikorra} and \cite[Proposition 2.4]{MartinoMazowieckaRodiac2025}.

\subsection{Statement of the results}\label{sec:statement}
Let us now introduce the necessary definitions in order to state our main theorem. Let \(n \geq 3\) and let \(\Omega \subset \R^n\) be  a \(\C^1\) bounded domain. Throughout the paper we will always assume that \(\Omega\) is \(\C^1\)-diffeomorphic to the unit ball \(\mathbb{B}^n\). We consider the function space
\begin{equation}\label{eq:defI}
\I \coloneqq \big\{v\in W^{1,n}(\O,\R^n) :  |\tr_{|\p \Om} u|=1 \text{ a.e.}\big\},
\end{equation}
and for \(u\) in \(\I\) we define the \(n\)-energy by
\begin{equation}\label{eq:n-energy}
	E _n(u) = \int_{\Omega} |\dd u|^n.
\end{equation}
Minimizing \(E_n\) in the entire space \(\I\) produces only constant map. In order to produce non-trivial critical point we can try to look for local minimizers or non-minimizing critical points. To do so, we can think of using the topology of \(\I\) and the topological degree. It is now well-established that  traces of maps in \(\I\) have a well-defined topological degree since the trace space \(W^{1-\frac{1}{n},n}(\p \Omega,\mathbb{S}^{n-1})\) embeds into \(\VMO(\p \Omega, \mathbb{S}^{n-1})\)\, cf.\ \cite{Brezis_Nirenberg_1995_I}.
Hence we can define the classes
\begin{equation}\label{eq:Idclass}
	\I_k\coloneqq \big\{v\in W^{1,n}(\Omega,\R^n)\colon |\tr_{| \p \Om }v|=1  \text{ a.e.\ and } \deg(v,\partial \Omega)=k \big\}.
\end{equation}
The images of the sets \(\I_k\) by the trace operator are the homotopy classes in \(\VMO \cap L^1\) of \(W^{1-\frac{1}{n},n}(\p \Omega,\mathbb{S}^{n-1})\). We can also show that the sets \(\I_k\) are the connected components of \(\I\) and they are open and close for the (strong) topology of \(W^{1,n}(\Omega,\R^n)\). From the openness condition we see that a minimizer of \(E_n\) in any of \(\I_k\) would be a local minimizer of \(E_n\) in \(\I\).  However, the direct method of calculus of variations does not lead to the existence of minimizers of \(E_n\) in every \(\I_k\) because these classes are not closed for the weak topology of \(W^{1,n}(\Omega,\R^n)\). Indeed, the degree is not continuous for the weak convergence and thus we face a problem with lack of compactness see  e.g.\ \cite[Example 2]{Berlyand_Mironescu_2006}.  Our first result is the following.

\begin{theorem}\label{th:main_1}
For \(k\in \mathbb{Z}\) and \(\Omega \subset \R^n\) we define 
\begin{equation}\label{eq:mkOmega}
m(k,\O)\coloneqq \inf \{ E_n(v) : v \in \I_k \}. 
\end{equation}
For \(k=0\), we have \(m(0,\Omega)=0\) and any constant in \(\mathbb{S}^{n-1}\) minimizes \(E_n\) in \(\I_0\). \\
For \(k \in \mathbb{Z}\setminus \{0\}\), if  we assume that \(\Omega \subset \R^n\) is a \(\C^1\) bounded domain so that \(\p \Omega\) is  a connected compact smooth submanifold of dimension \((n-1)\) of \(\R^n\), then
\begin{equation}
m(k,\Omega) = |k|\, n^{\frac{n}{2}}\, |\bB^n|,
\end{equation}
and \(m(k,\Omega)\) is attained if and only if \(\Omega\) is a Euclidean ball and \(|k|=1\). In this case, minimizers are exactly the conformal transformations between this ball and \(\mathbb{B}^n\).
\end{theorem}

Thus a natural question is to look for non-minimizing critical points of \(E_n\) in \(\I_1\) when \(\Omega\) is not a ball. We first explain what we mean exactly by critical points.

\begin{definition}\label{def:critical}
We say that \(u\in \I\) is a critical point of \(E_n\) in \(\I\) if, for all \(\varphi \in W^{1,n}(\Omega,\R^n)\) such that \(\tr_{|\p \O} \varphi(x) \in T_{u(x)} \mathbb{S}^{n-1}\) for \(\mathcal{H}^{n-1}\)-a.e.\ \(x\in \p \O\), we have
\begin{equation*}
\int_{\Om} |\dd u|^{n-2}\dd u:\dd \varphi =0.
\end{equation*} 
\end{definition}

 We have denoted by \( A:B:=\tr (A^TB)\) the inner product between two matrices \(A\) and \(B\) in \(M_n(\R)\). The preceding definition is motivated by the fact that, for \(\varphi\) as in Definition \ref{def:critical}, \(\tr_{| \p \Omega} (u+t\varphi)\) belongs to \(\mathbb{S}^{n-1}\) at first order, i.e., \(|\tr_{| \p \Omega} (u+t\varphi)|=1+O(t^2)\) a.e.\  Our main result in this paper shows the existence of a critical point of \(E_n\) in \(\I\) with prescribed degree \(1\) for domains which are close to a ball.
\begin{theorem}\label{th:main_2}
	Let $n\geq 3$. There exists \(L>0\) such that, if \(\Omega\) is a \(\C^1\)-bounded domain verifying that there exists a \(\C^1\)-diffeomorphism \(\Phi\colon \overline{\Om}\rightarrow \overline{\mathbb{B}^n}\) with \(\Phi(\p \Om)=\mathbb{S}^{n-1}\) and with \(\|\Phi-\id\|_{\C^1(\Omega)}<L\), then there exists a critical point of \(E_n\) in \(\mathcal{I}_1\).
\end{theorem}
We note that in dimension \(2\) critical points of the Dirichlet energy with prescribed degrees can be described exactly. They are given by Blaschke products of Möbius maps and they minimize the Dirichlet energy in their homotopy classes \(\I_k\), for \(k\in \mathbb{Z}\); see for instance \cite[Lemma 3.5]{Berlyand_Mironescu_Rybalko_Sandier_2014} or \cite[Theorem 4.25]{Millot_Sire_2015}.

As a side remark we claim that, with the same techniques, we can also obtain an analogous result for the \(n\)-Ginzburg--Landau (GL) energy. Since this is not the main focus of this paper we will not provide a proof of the following statement.
\begin{theorem}\label{th:main_3}
There exist \( L>0\) and \(\kappa_0>0\) such that, if \(\Omega\) is a \(\C^1\)-bounded domain verifying that there exists a \(\C^1\)-diffeomorphism \(\Phi\colon \overline{\Om}\rightarrow \overline{\mathbb{B}^n}\) with \(\Phi(\p \Om)=\mathbb{S}^{n-1}\) and with \(\|\Phi-\id\|_{\C^1(\Omega)}<L\) and if \(\kappa<\kappa_0\) then there exists a critical point of \(E_{n,\kappa}\) in \(\mathcal{I}_1\) where 
 \begin{equation}\label{eq:GL_energy}
 E_{n,\kappa}(u)=\frac{1}{n}\int_{\Omega}| \dd u|^n +\kappa^n \int_{\Omega} (1-|u|^2)^2.
 \end{equation}
 \end{theorem}
Theorem \ref{th:main_2} can be seen as a perturbation result with respect to perturbation of the domain and Theorem \ref{th:main_3} can be seen as a perturbation result for perturbations of the domain and of the energy.

\subsection{Earlier works on variational problems with prescribed degrees}

Before explaining the strategy of the proof of Theorem \ref{th:main_2} we would like to mention some earlier results on variational problems with prescribed topological degrees which are a particular case of problems with free boundaries. Such problems were introduced for instance in the context of the GL energy \eqref{eq:GL_energy} in dimension \(n=2\), see \cite{Berlyand_Voss_2001,Berlyand_Mironescu_2003}, as an intermediate problem between the minimization of GL energy without magnetic field but with a prescribed boundary data with non-zero degree, see \cite{Bethuel_Brezis_Helein_1994}, and the minimization of the full GL energy with magnetic field (and with Neumann boundary conditions), see \cite{Sandier_Serfaty_2007}. Whereas the main focus in the study of the GL energy with a prescribed Dirichlet condition or in the study of the GL energy with magnetic field is the asymptotic behaviour as \(\e \to 0\) of a family of minimizers, the difficulty in the model with prescribed degrees is to prove the mere existence of minimizers because of the lack of compactness of the problem. The literature on prescribed degree problems for the GL energy is now abundant and we refer to \cite{Golovaty_Berlyand_2002,Berlyand_Mironescu_2006,Berlyand_Golovaty_Rybalko_2006,Berlyand_Mironescu_2008,Dos_Santos_2009,Berlyand_Rybalko_2010,Berlyand_Misiats_Rybalko_2010,Berlyand_Misiats_2011,Farina_Mironescu_2013,Berlyand_Mironescu_Rybalko_Sandier_2014,Misiats_2014,Lamy_Mironescu_2014,Mironescu_2014,Rodiac_Sandier_2014,DosSantos_Rodiac_2016,Rodiac_2019} and references therein. Since the difficulty is the lack of compactness of the problem and it is already present when studying only the Dirichlet energy with prescribed degrees, the issue of the construction of harmonic maps with prescribed degrees on the boundary was also considered in \cite{Berlyand_Mironescu_Rybalko_Sandier_2014}, for simply connected domains, and in \cite{Hauswirth_Rodiac_2016}, for annular domains. We note that harmonic maps with prescribed degrees are directly related to the notion of half-harmonic maps introduced by Da-Lio--Rivi\`ere, see \cite{DaLio_Riviere_2011,DaLioRiviere_2011b,Moser_2011}. Besides its possible interest to find a generalization of conformal maps between domains of \(\R^n\), the study of \(n\)-harmonic maps with prescribed degrees or with free boundary into \(\mathbb{S}^{n-1}\) can help gaining a new insight into problems which are critical for Sobolev injections and problems with lack of compactness and developing tools which are more robust than the ones used in dimension 2. We observe that variational problems with lack of compactness appear often in geometry and physics and include for example, semi-linear elliptic problems with critical exponents (e.g.\ the Yamabe problem), harmonic maps in dimension 2, minimal surfaces, surfaces with constant mean curvature, Willmore surfaces, Yang--Mills connections..., see e.g.\  \cite[Chapter II]{Struwe_2008} or the surveys \cite{Brezis_1988,Struwe_2023,R2020,LMR2026}.

\subsection{Strategy of the proof}
We now explain the strategy of the proof of our main result Theorem \ref{th:main_2}. The idea is inspired by \cite{Berlyand_Mironescu_Rybalko_Sandier_2014} and is the following. We want to apply the mountain pass theorem of Ambrosetti--Rabinowitz in the form recalled in Theorem \ref{th:Mountain Pass} below. In order to apply Theorem \ref{th:Mountain Pass}, with the notation of this theorem, we would like to take
\[ K_0=\p \Omega, \quad K=\overline{\Omega}, \quad X=\I_1, \quad J=E_n\]
and we need to find a continuous path \(\chi\colon \p \Omega \rightarrow \I_1\) verifying some properties. Let us assume that we have a family of continuous paths indexed by \( 0\leq r <1\) and denoted by \(\chi_r\colon \p \Omega \rightarrow \I_1\) that satisfy:
\begin{equation}\label{eq:minimal_energy}
\forall a\in \p\O,\qquad E_n( \chi_r(a)) \xrightarrow[r \to 0]{} n^{\frac{n}{2}} |\mathbb{B}^n|,
\end{equation}
\begin{equation}\label{eq:topo_condition}
\forall a\in \p\O,\qquad \barint_{\Omega} |\chi_r(a) -a|  \dd x \xrightarrow[r \to 0]{} 0.
\end{equation}
Condition \eqref{eq:topo_condition} means that the map \(a\in \p \O \mapsto \barint_{\Omega} \chi_r(a) \dd x\) is close to the identity in \(L^\infty\) norm for \(r\) sufficiently small. Thus, thanks to a topological argument similar to Brouwer fixed point theorem, for each \(F\in \C^0(\overline{\Omega}, \mathcal{I}_1)\) such that \( F=\chi_r\) on \(\p \Omega\), we can find \(a_F\in \Omega\) such that \(\barint_\Omega F(a_F)=0\). This property along with Condition \eqref{eq:minimal_energy} and the fact that \(\Omega\) is not a round ball can be used to proved that we are in presence of a mountain pass geometry, that is
\begin{equation*}
c_1(r)\coloneqq\max_{\p \Omega} E_n(\chi_r(a)) < \inf_{\{F\in \C^0(\overline{\Omega}, \mathcal{I}_1); F=\chi_r \text{ on } \p \Omega\}} \max_{a\in \overline{\Omega}} E_n(F(a)) \eqqcolon c(r).
\end{equation*}
The paths \(\chi_r\) are taken to be M\"obius maps centred at \( (1-r)a\) in \cite{Berlyand_Mironescu_Rybalko_Sandier_2014} where the authors prove the existence of critical points  in \(\I_1\) of a Ginzburg--Landau energy like \eqref{eq:GL_energy} for \(\kappa\) small with \(\Omega\) a smooth bounded simply connected domain. Actually, they use the Riemann mapping theorem and the conformal invariance of the Dirichlet energy to work on a disk with a modified Ginzburg--Landau energy. Since the Riemann mapping theorem does not hold in higher dimension, the construction of paths \(\chi_r\) satisfying \eqref{eq:minimal_energy} and \eqref{eq:topo_condition} is one of the main novelty of this paper to which Section \ref{sec:AlmostMobius} is devoted. 

Proving the presence of a mountain pass geometry is not the only difficulty here. Indeed, first we cannot apply directly Theorem \ref{th:Mountain Pass} because \(\I\) and \(\I_1\) are not known to be \(\C^1\) Banach manifolds. Second, and more importantly, since the problem is critical for the Sobolev injections, Theorem \ref{th:Mountain Pass} provides only a Palais--Smale sequence and nothing guarantees that it converges to an actual critical point of \(E_n\) in \(\I_1\). The strategy to overcome these two difficulties is well-known by now. We consider subcritical problems by introducing a parameter \(\a>0\) and by looking for critical points of 
\begin{equation}\label{eq:SacksUhlenbeckfunctional}
	E_{n+\alpha}(u) \coloneqq \int_{\Omega} |\dd u|^{n+\alpha} \dx,
\end{equation}
in the class
\begin{equation}\label{eq:Idalphaclass}
	\I_{1,\alpha}\coloneqq \{v\in W^{1,n+\alpha}(\Omega,\R^n)\colon |v|=1 \text{ on } \partial \Omega \text{ and } \deg(v,\partial \Omega)=1 \}.
\end{equation}
This is reminiscent of the pioneer work of Sacks--Uhlenbeck \cite{sacks1981}. In this situation we can prove that \(\I_{1,\alpha}\) is a \(\C^1\) Banach manifold, that \(E_{n+\alpha}\) is a \(\C^1\) functional which satisfies the Palais--Smale condition and that the mountain pass geometry persists for \(\alpha>0\) small enough. Hence for each \(\alpha>0\) we obtain, thanks to Theorem \ref{th:Mountain Pass}, a critical point \(u_{\alpha}\) of \(E_{n+\alpha}\) in \(\I_{1,\alpha}\). We then study the convergence of \( (u_{\alpha})_{\alpha>0}\), the main ingredient for that being the bubbling theorem obtained in \cite[Theorem 1.3]{MartinoMazowieckaRodiac2025}. However, contrarily to the case when \(n=2\), see \cite[Theorem 8.13]{Berlyand_Mironescu_Rybalko_Sandier_2014}, the precise energy of the bubbles is not known in higher dimension. Thus, in order to conclude from a contradiction argument using the bubbling decomposition that \(u_\alpha\) converges weakly in \(W^{1,n}\) towards a map in \(\I_1\), we need to show that no bubble of degree one has an energy arbitrarily close to the energy level \(n^{\frac{n}{2}}|\bB^n|\). This energy gap result is the second main novelty of this work and is shown in Section \ref{sec:gapMobius} by using the stability of the identity map for the free-boundary problem.

\subsection{Organization of the paper}The paper is organized as follows. In Section \ref{sec:Preliminaries} we recall some properties of the topological degree and give a result estimating the energetic cost for a sequence of maps to have a different degree in the weak limit. This is a first study of the lack of compactness of the problem. We also prove two preliminary technical lemmas, the first one estimating the difference in the \(L^1\) norm of two \(n\)-harmonic extensions in term of the \(L^1\) norm of their traces  and the second one being a gap theorem showing that there exist no free boundary \(n\)-harmonic maps with arbitrarily small energy. Section \ref{sec:Minimization} is devoted to the proof of Theorem \ref{th:main_1}, the main ingredients are the area-energy inequality recalled in Section \ref{sec:Preliminaries} and a construction of test functions that are almost conformal. Section \ref{sec:AlmostMobius} is devoted to the construction of continuous paths \(\chi_r\colon \p \Omega \rightarrow \I_1\) for \(r\) small satisfying \eqref{eq:minimal_energy} and \eqref{eq:topo_condition}. The maps \(\chi_r(a)\) play a similar role as M\"obius maps in the unit ball \(\mathbb{B}^n\) and are thus called ``almost M\" obius'' maps. With the help of these paths, we can define our min-max scheme, for the approximate functional \eqref{eq:SacksUhlenbeckfunctional}, in Section \ref{sec:MountainPass}. In this section we first explain that \(\I_{1,\alpha}\) in \eqref{eq:Idalphaclass} is a \(\C^1\) Banach manifold, that \(E_{n+\alpha}\) is \(\C^1\) on that space and that it satisfies the Palais--Smale condition. Then we prove that our min-max scheme leads to a genuine min-max geometry, hence obtaining \((n+\alpha)\)-harmonic maps with free boundary into \(\mathbb{S}^{n-1}\). Before analysing the asymptotic behaviour of these critical points when \(\alpha\) goes to zero, we prove in Section \ref{sec:gapMobius} a second gap theorem showing that there is no \(n\)-harmonc maps with free boundary into \(\mathbb{S}^{n-1}\) with \(n\)-energy arbitrarily close to \( E_n(\id)=n^{\frac{n}{2}}|\mathbb{B}^n|\). This result is instrumental in the proof by contradiction of Theorem \ref{th:main_2} that we complete in Section \ref{sec:ProofMain}.

\subsection*{Open problems} We conclude this introduction by mentioning some open problems.

\begin{enumerate}

	\item In the general case where $\Omega$ is not close to a ball, the min-max procedure produce a bubbling tree of free-boundary $n$-harmonic maps. Can we estimate the total energy? Is it possible to exclude the existence of bubbles? If not, is it possible to prove that there is no free-boundary $n$-harmonic maps with degree 1?

	\item Can we show that the $n$-harmonic map from Theorem \ref{th:main_2} is close to a Möbius map in some sense? If yes, is it a diffeomorphism?

	\item Can we obtain a sharp estimate of the energy gap for Möbius maps (i.e. the number $\eps_M$ in \Cref{th:Energy_gap_Mobius})? What would be the first free boundary $n$-harmonic maps with energy strictly larger than $n^{\frac{n}{2}}\, |\bB^n|$? Is there any?

\end{enumerate}

\subsection*{Notation} 

For $p>1$ we use the symbol $\Delta_p$ to denote the $p$-Laplacian, i.e., for a function $u$ we write $\Delta_p u\coloneqq \text{div} (|\dd u|^{p-2}\dd u)$.

\subsection*{Acknowledgement}

The project is co-financed by:
\begin{itemize}
 \item (D.M.) Swiss National Science Foundation, project SNF $200020\text{\textunderscore}219429$;
 \item (K.M.) the Polish National Agency for Academic Exchange within Polish Returns Programme - BPN/PPO/2021/1/00019/U/00001;
 \item (K.M.) the National Science Centre, Poland grant No. 2023/51/D/ST1/02907;
 \item  (R.R.) has been supported by the French government, through the \(\textbf{UCA}^{JEDI}\) Investments in the Future project managed by the National Research Agency (ANR) with the reference number ANR-15-IDEX-01.
\end{itemize}
Part of this work was realized when D.M. and R.R. were visiting the University of Warsaw, they would like to thank the Mathematics Department for its hospitality and excellent working conditions.

\section{Preliminaries}\label{sec:Preliminaries}

\subsection{Properties of the degree}
We start by recalling the definition and some properties of the topological degree. We refer to \cite{Outerelo_Ruiz_2009} for the classical theory of the topological degree for \(\C^1\) and \(\C^0\) maps and to \cite{Brezis_Nirenberg_1995_I,Brezis_Nirenberg_1995_II,Brezis_1997_degree,Brezis_2006_degree,Mironescu_2007,Mironescu_2011} for the more advanced theory of the topological degree in Sobolev and \(\VMO\) spaces.

Let \(X\) be an $(n-1)$-dimensional \(\C^1\) compact connected manifold without boundary. Let \(\sigma\) be an \((n-1)\)-differential form on \(\mathbb{S}^{n-1}\). For maps \(u\in \C^1(X,\mathbb{S}^1)\) we can define their topological degree by the formula
\begin{equation}\label{degree1}
\deg(u,X)= \frac{1}{\int_Y\sigma}\int_X u^*\sigma,
\end{equation}
where $u^*\sigma$ denotes the pull-back of $\sigma$ by $u$. It can be proved that the degree is an integer, that this quantity is independent of the chosen differential form and that it counts the number of times $\mathbb{S}^{n-1}$ is covered by $u(X)$ taking into account algebraic multiplicity. Moreover, an important property is that the degree is invariant by homotopy. Hence we can extend its definition to maps in \(\C^0(X,\mathbb{S}^{n-1})\). As explained in the introduction, the definition of the topological degree can be furthermore extended to maps in \(\VMO(X,\mathbb{S}^{n-1})\), cf.\ \cite{Brezis_Nirenberg_1995_I} and thus to Sobolev maps in the critical space \(W^{1-\frac{1}{n},n}(X,\mathbb{S}^{n-1})\). The definition of topological degree in some Sobolev spaces was observed earlier in \cite{Brezis_Coron_1983} and \cite[Appendix by L. Boutet de Monvel and O. Gabber]{Boutet_Georgescu_Purice_1991}. 

We now give a formula that will be useful to find a lower-bound on the \(n\)-energy in terms of the degree. In order to write this formula, we take \(\omega\) the volume form on \(\R^n\) and \(\sigma\) the classical volume form on \(\mathbb{S}^{n-1}\) given by

\begin{equation*}
\omega=\dd x_1\wedge \ldots \wedge \dd x_n, \quad \sigma =\sum_{j=1}^n (-1)^{j-1}x_j\, \dd x_1\wedge\ldots \wedge \widehat{\dd x_j} \wedge \ldots \wedge \dd x_n,
\end{equation*}
where $\widehat{\dd x_j}$ means that we omit the term $\dd x_j$ in the wedge product. Now, suppose that $X= \partial \Omega$, with $\Omega$ a \(\C^1\) bounded domain of $\R^n$. Let $\tilde{u}$ be any $\mathcal{C}^1$ extension to $\Omega$ of $u$. Using Stokes' Theorem we obtain
\begin{align*}
\deg(u,X) = \frac{1}{|\mathbb{S}^{n-1}|}\int_{\p \Omega} u^*\sigma = \frac{1}{|\mathbb{S}^{n-1}|}\int_\Omega \dd[\tilde{u}^*\sigma] =\frac{1}{|\mathbb{S}^{n-1}|}\int_\Omega\tilde{u}^*\dd\sigma = \frac{n}{|\mathbb{S}^{n-1}|}\int_\Omega\tilde{u}^*\omega. 
\end{align*}
Since $|\bB^n|=\frac{|\mathbb{S}^{n-1}|}{n}$, we obtain that
\begin{equation}\label{eq:formula_degree_ext}
\deg(u, \p \Omega)= \frac{1}{|\mathbb{B}^n|}\int_{\Omega}\Jac\, \tilde{u}(x)\, \dd x=\frac{1}{|\bB^n|}\int_\Omega \pl_{x_1} \tilde{u} \wedge \ldots \wedge \pl_{x_n}\tilde{u},
\end{equation}
where $\Jac\, \tilde{u}(x)= \det \left(\dd \tilde{u}(x)\right)$ and \(\tilde{u}\) denotes any \(\C^1\) extension of \(u\) to \(\Omega\). This formula remains true for maps \(u\in W^{1-\frac{1}{n},n}(\p \Omega, \mathbb{S}^{n-1})\) and their extensions \(\tilde{u}\) in \(W^{1,n}(\Omega,\R^n)\) by using that the right-hand side \eqref{eq:formula_degree_ext} is continuous with respect to the strong convergence in \(W^{1,n}\) and by the density of smooth maps in \(u\in W^{1-\frac{1}{n},n}(\p \Omega, \mathbb{S}^{n-1})\) see e.g.\ \cite[Theorem 2.1]{Mironescu_2007}. 
With the help of the Hadamard inequality and the arithmetico-geometric inequality, Formula \eqref{eq:formula_degree_ext} leads to the following result.

\begin{lemma}\label{lowbound}
Let $\O\subset \R^n$ be a \(\C^1\) bounded open set such that \(\p \Omega\) is a \(\C^1\) compact connected \((n-1)\)-manifold. Let $u\in W^{1,n}(\O,\R^n)$ then
\begin{equation}\label{lowerbound1}
\int_\O |\dd u|^n \geq n^{\frac{n}{2}} \left| \int_\O \Jac\, u \right|.
\end{equation}
Hence, if $u\in \I_k$, we see from Formula \eqref{eq:formula_degree_ext} that
\begin{equation}\label{lowerbound2}
E_n (u)= \int_\O |\dd u|^n \geq  n^{\frac{n}{2}}\, |\bB^n|\,| k|.
\end{equation}
Moreover, equality holds in the previous inequality if and only if $u$ is conformal, i.e.\ \( (\p_{x_1}u,\dots,\p_{x_n}u)\) is an orthogonal system and  \(|\p_{x_1} u|=\cdots=|\p_{x_n} u|\) for a.e.\ \(x\) in \(\Omega\) and \(\Jac (u) \geq 0\) a.e.\ or \(\Jac (u) \leq 0\) a.e.\ depending on the sign of the degree.
\end{lemma}

The details of the proof of this lemma are omitted, see e.g.\ \cite[Lemma 5.8]{Rodiac}. We note that a similar result was obtained for maps from \(\mathbb{S}^n\) to \(\mathbb{S}^n\) in \cite{Xu_Yang_1992}. With the help of Lemma \ref{lowbound}, to prove Theorem \ref{th:main_1} it suffices to construct a matching upper-bound and this is what we will do in Section \ref{sec:Minimization}.

\subsection{Bubbling analysis}

We pursue these preliminaries with a first study of the lack of compactness of the problem. The following result is the analogue of \cite[Lemma 1]{Berlyand_Mironescu_2006} where it is called ``Price Lemma'' since it quantifies the energetic cost for the limit of a weakly convergent sequence in \(W^{1,n}\) to have a different degree from the sequence.

\begin{proposition}\label{prop:Price_lemma}
	Let \(\Om\subset \R^n\) be a \(\C^1\) bounded domain such that \(\p \Omega\) is a \(\C^1\) compact connected \( (n-1)\)-manifold. Let \(\ell\in \mathbb{Z}\) and let \( (u_k)_k\subset \mathcal{I}_{\ell}\) be such that \(u_k \rightharpoonup u \) in \(W^{1,n}(\Om,\R^n)\) and \(\dd u_k\rightarrow \dd u\) a.e.\ in \(\Om\). Then
	\begin{equation}\label{eq:Price_lemma}
		\liminf_{k\to +\infty} \int_\Om |\dd u_k|^n \geq \int_\Om |\dd u|^n +n^{\frac{n}{2}}\, |\mathbb{B}^n|\, |\deg (u,\partial \Om)-\ell|.
	\end{equation}
\end{proposition}

\begin{proof}
	By applying Brezis--Lieb's lemma \cite{BrezisLieb83}
	we find that, if \(v_k\coloneqq u_k-u\) then
	\begin{align*}
		\int_\Om |\dd u_k|^n=\int_\Om |\dd u|^n+\int_\Om |\dd v_k|^n+\smallO_k(1).
	\end{align*}
	We can then use \Cref{lowbound} to obtain \(\int_\Om |\dd v_k|^n \geq n^{\frac{n}{2}} \left| \int_{\Omega} \Jac (v_k) \right|\). But from \Cref{le:jacobianprice} below we have that
	\(\int_\Om \Jac(v_k) =\int_\Om \Jac (u_k)-\int_\Om \Jac (u)+\smallO_k(1).\)
	Hence, with the help of Lemma \ref{lowbound} we find that 
	\begin{align*}
		 \int_\Om |\dd u_k|^n
		 &\ge \int_\Om |\dd u|^n +n^{\frac{n}{2}} \left| \int_\Om \Jac (u_k)-\int_\Om \Jac (u)\right|+\smallO_k(1) \\
		& =\int_\Om |\dd u|^n +n^{\frac{n}{2}}\, | \mathbb{B}^n|\, |\deg(u_k,\partial \Om)-\deg(u,\partial \Om)|+\smallO_k(1).
	\end{align*}
\end{proof}
In the proof of Proposition \ref{prop:Price_lemma} we have used the following lemma.

\begin{lemma}\label{le:jacobianprice}
	Let \((u_k)_k\subset W^{1,n}(\Om,\R^n)\) be such that \(u_k \rightharpoonup u\) in \(W^{1,n}(\Om,\R^n)\), then
	\begin{equation*}
		\int_\Om \Jac (u_k-u) =\int_\Om \Jac (u_k)-\int_\Om \Jac(u) +\smallO_k(1).
	\end{equation*}
\end{lemma}

\begin{proof}
	We start with the following algebraic identity which can be proved by using the multi-linearity of the determinant and an induction argument:
	\begin{multline}\label{eq:algebraic_jacobian}
		\Jac( u_k-u)= \Jac (u_k)+(-1)^n \Jac(u) -\sum_{j=1}^n \dd u_k^1\wedge \dd u_k^2 \wedge \widehat{ \dd u^j} \wedge \ldots \wedge \dd u_k^n \\
		+\sum_{i,j=1}^n \dd u_k^1\wedge \dd u_k^2 \wedge \widehat{\dd u^j} \wedge \ldots \wedge \widehat{\dd u^i}\wedge \ldots \dd u_k^n +\ldots \\
		+\sum_{i_1,\dots,i_{n-1}=1}^n \dd u_k^1 \wedge \widehat{\dd u^{i_1}}\wedge \widehat{\dd u^{i_2}}\wedge \ldots \wedge \widehat{\dd u^{i_{n-1}}}.
	\end{multline}
	Here the notation with the sum and \(\widehat{\dd u^j}\) means that  we sum over the possible positions of \( \dd u^j\). 
	Next we prove by induction on \(\ell\) that for every \(1\leq \ell \leq n\) 
	\begin{equation}\label{eq:weak_convergence_det}
		\int_\Om \dd u_k^1 \wedge \ldots \wedge \dd u_k^\ell \wedge \varphi \xrightarrow[k \to +\infty]{} \int_\Om \dd u^1 \wedge \ldots \wedge \dd u^\ell \wedge \varphi
	\end{equation}
	for every \(\varphi \in L^{\frac{n}{n-\ell}}(\Om, \Ep^{n-\ell}(\R^n))\). The proof follows \cite[Th\'eor\`eme 3.1]{Mironescu_2011}. The case \(\ell=1\) follows from the weak convergence of \(\dd u_k\) in \(L^n(\Om,\R^{n\times n})\). Let us assume that the result holds for some \( 1\leq \ell < n\). An integration by parts shows that, for \(\varphi \in \C^\infty_c(\Om, \Ep^{n-\ell-1}(\R^n))\),
	\begin{equation}\label{eq:ell+1}
		\int_\Om \dd u_k^1 \wedge \ldots \dd u_k^{\ell+1} \wedge \varphi =-\int_\Om \dd u_k^1 \wedge \ldots \dd u_k^{\ell} \wedge (u_k^{\ell+1} \dd\varphi).
	\end{equation}
	By the Rellich--Kondrachov embedding we know that \( u_k^{\ell+1} \rightarrow u^{\ell+1} \) strongly in \(L^q(\Om)\) for any \(1\leq q<+\infty\). By the induction hypothesis, \(\dd u_k^1 \wedge \ldots\wedge \dd u_k^{\ell} \) converges weakly to \( \dd u^1 \wedge \ldots\wedge \dd u^{\ell}\) in \(L^{\frac{n}{\ell}}(\Om, \Ep^{n-\ell}(\R^n))\). Hence we obtain that \eqref{eq:weak_convergence_det} holds for any \(\varphi\in \C^\infty_c(\Om, \Ep^{n-\ell-1}(\R^n))\). Since \(\C^\infty_c(\Om, \Ep^{n-\ell-1}(\R^n))\) is dense in \( L^{\frac{n}{n-\ell}}(\Om, \Ep^{n-\ell}(\R^n))\), Equation \eqref{eq:weak_convergence_det} also holds for \(\varphi\) in this space  and this concludes the induction argument.
	
	Now the weak convergence result \eqref{eq:weak_convergence_det} and the identity \eqref{eq:algebraic_jacobian} show that 
	\begin{align*}
		\int_\Om \Jac (u_k-u) &=\int_\Om \Jac (u_k) +\left[(-1)^n+\sum_{k=1}^{n-1}(-1)^k \binom{n}{k}\right] \int_\Om \Jac (u)+\smallO_k(1) \\
		&=\int_\Om \Jac (u_k)-\int_\Om \Jac (u) +\smallO_k(1),
	\end{align*}
	where we used the Newton binomial formula to obtain \((-1)^n+\sum_{k=1}^{n-1}(-1)^k \binom{n}{k}=-1\).
\end{proof}

\subsection{Regularity of $n$-harmonic extension}

In this last part of the preliminaries section we give two results related to the regularity of \(n\)-harmonic maps. The first result can be seen as a result of uniform continuity of the \(n\)-harmonic operator for the \(L^1\) norm. The second result is an energy gap theorem stating that \(n\)-harmonic maps with free boundary into \(\mathbb{S}^{n-1}\) cannot have their energy arbitrarily close to zero. Such a result is already contained in \cite[Lemma 3.1]{MartinoMazowieckaRodiac2025} but we give a reformulation that makes clear that the gap is independent of the domain inside an appropriate class of the domains since it will be important later.
\begin{lemma}[Uniform continuity of the $n$-harmonic extension in $L^1$]\label{lm:uniform_nharm_extension}
Let $\Omega\subset\R^n$ be a bounded Lipschitz domain and let $N\in\N$. For any $\Lambda>0$ and $\eps>0$, there exists $\eta=\eta(\eps,\Lambda,n)>0$ such that the following holds.
Let $\gamma_1,\gamma_2\in W^{1-\frac{1}{n},n}(\partial\Omega,\R^N)$ such that 
\[
\|\gamma_1\|_{W^{1-\frac{1}{n},n}(\partial\Omega)}
+
\|\gamma_2\|_{W^{1-\frac{1}{n},n}(\partial\Omega)}
\le \Lambda,
\qquad \text{ and } \qquad
\|\gamma_1-\gamma_2\|_{L^1(\partial\Omega)}
\le \eta.
\]
Let $u_1,u_2\in W^{1,n}(\Omega,\R^N)$ denote their $n$-harmonic extensions. Then
\[
\|u_1-u_2\|_{L^1(\Omega)} \le \eps.
\]
\end{lemma}

\begin{proof}
	Assume on the contrary that there exist $\Lambda>0$ and $\eps>0$ such that the following holds.
	
	There exists $\gamma\in W^{1-\frac{1}{n},n}(\pl \Omega,\R^N)$ and a sequence $(\gamma_k)_{k\in\N}\subset W^{1-\frac{1}{n},n}(\pl \Omega,\R^N)$ such that
	\begin{equation*}
	 \begin{split}
	  & \gamma_k \rightharpoonup {\gamma} \quad \text{ weakly in }W^{1-\frac{1}{n},n}(\pl \Omega),\\
	 & \gamma_k \to {\gamma} \quad \text{ strongly in } L^1(\pl\Omega),\\
			&\|\gamma_k\|_{W^{1-\frac{1}{n},n}(\pl \Omega)} + \|\gamma\|_{W^{1-\frac{1}{n},n}(\pl \Omega)} \leq \Lambda,
	 \end{split}
	\end{equation*}
	and their $n$-harmonic extensions satisfy
	\begin{align}\label{eq:contradiction_hypothesis}
		\|u_k-u\|_{L^1(\Omega)} \geq \eps.
	\end{align}
	Since each $u_k$ is energy minimizing, by Gagliargo's trace theorem, there exists $C=C(\Omega)>0$ such that
	\begin{align*}
		\forall k\in\N,\ \ \ \|\dd u_k\|_{L^n(\Omega)} \leq C \| \gamma_k \|_{W^{1-\frac{1}{n},n}(\pl\Omega)} \le C \Lambda.
	\end{align*}
	Thus $(u_k)_{k\in\N}$ is bounded in $W^{1,n}(\Omega,\R^N)$ and we can extract a weakly converging subsequence. We denote its  limit by $u_*\in W^{1,n}(\Omega,\R^N)$. From Rellich--Kondrachov theorem, we can assume that the limit is also strong in $L^1(\Omega,\R^N)$.

	By regularity theory, every solution to $\Delta_{n} u_k =0$ is of class $\C^{1,\beta}(\Omega,\R^n)$ for some $0<\beta<1$, see \cite{Uhlenbeck_1977}. Hence, $u_k$ is also bounded in $\C^{1,\beta}_{\loc}$. Thus, up to a subsequence, $u_k \to u_*$ strongly in $\C^1_{\loc}$ and we can pass to the limit in the system $\lap_n u_k = 0$. Hence, $u_*$ satisfy $\lap_n u_*=0$.

	By continuity of the trace, we also have in $L^1(\pl\Omega,\R^N)$ :
	\begin{align*}
		\tr_{|\pl \Omega}(u_*) = \lim_{k\to \infty} \tr_{| \pl\Omega}(u_k) = \lim_{k\to\infty}\gamma_k = \gamma.
	\end{align*}
	By uniqueness of the solution to the system
	\begin{align*}
		\left\{ 
		\begin{array}{rcll}
			\lap_n u & =& 0 & \text{in }\Omega,\\
			u &=&  \gamma & \text{on }\pl \Omega,
		\end{array}
		\right.
	\end{align*}
	we conclude that $u=u_*$ in $W^{1,n}$. This is a contradiction to \eqref{eq:contradiction_hypothesis}.
\end{proof}

The following lemma is an immediate consequence of the proof of \cite[Lemma 3.1]{MartinoMazowieckaRodiac2025}. We record it separately since we need the explicit dependence of $\eps_c$ on the domain.
\begin{lemma}[cf. {\cite[Lemma 3.1]{MartinoMazowieckaRodiac2025}}]\label{le:3.1inMMR}
	Let $L>0$. There exists $\eps_c=\eps_c(n,L)>0$ satisfying the following. Assume that $\O\subset \R^n$ is an open set with a diffeormorphism $\Phi\colon \overline{\Omega}\to \overline{\bB^n}$ such that
	\[
		\|\Phi\|_{\C^1(\Omega)} + \|\Phi^{-1}\|_{\C^1(\bB^n)} \leq L.
	\]
	Let $u\in W^{1,n}(\Omega,\mathbb{B}^n)$ be a free boundary $n$-harmonic map in the sense of \Cref{def:critical}. If $\|\dd u\|_{L^n(\Omega)} < \eps_c$,
	then $u$ is constant.
\end{lemma}

\section{Minimization with prescribed degrees}\label{sec:Minimization}

In this section, we prove Theorem \ref{th:main_1}. The two main ingredients for the proof of this result are the lower bound on the \(n\)-energy in terms of the degree provided in Lemma \ref{lowbound} and a matching upper bound given in Lemma 3.1 below. We use also the rigidity theorem of Liouville which says that the only conformal transformations of the Euclidean space in dimension \(n \geq 3\) are a composite of isometries, dilations and inversions. More precisely we will use the version of the Liouville's theorem in \cite[Theorem 5.1.1]{Iwaniec_Martin_2001} that we recall here for the comfort of the reader

\begin{theorem}\label{th:Liouville}(Liouville)
Let \(\Omega\) be a domain in \(\R^n\) with \(n \geq 3\). Let \(u\in W^{1,n}_{\text{loc}}(\Omega,\R^n)\) be such that 
\begin{equation}
\Jac \, u \geq 0 \text{ a.e.\ in } \Omega \text{ or } \Jac \, u \leq 0 \text{ a.e.\ in } \Omega,
\end{equation}
\begin{equation}
\dd u^T \dd u=|\Jac\,  u|\, \id \text{ a.e.\ in } \Omega.
\end{equation}
Then \(u\) is either constant or the restriction to \(\Omega\) of a M\"obius transformation of \(\R^n\). More precisely 
\begin{equation}
u(x) = b+\frac{\alpha A (x-a)}{|x-a|^\epsilon}
\end{equation}
for some \( a \in \R^n\setminus \Omega\), some \(b\in \R^n\) some \(\alpha \in \R\) some \(A\in O_n(\R)\) and some \(\epsilon \in \{0,2\}\).
\end{theorem}
We recall that the quantity \(m(k,\Omega)\) is defined in \eqref{eq:mkOmega}. We observe that since $\I_0$ contains the constant maps, we have $m_\kappa(0,\O)=0$. 
\begin{lemma}\label{testfunc}
Let $\delta >0$, let \(k,k'\in \mathbb{Z}\) and let $u\in \I_k$. Then there exists $v\in \I_{k'}$ such that 
\begin{equation*}
E_n(v)\leq E_n(u) +n^{\frac{n}{2}-1}\, |\bB^n| \, |k-k'|+\delta.
\end{equation*}
In particular, for any $(k,k')\in \mathbb{Z}$, it holds that 
\begin{equation*}
m(k',\Omega) \leq m(k,\Omega) +n^{\frac{n}{2}-1}\, |\bB^n|\, |k-k'|.
\end{equation*}
\end{lemma}

\begin{proof}[Proof of Lemma \ref{testfunc}.]
 Let $x_0 \in \partial \Omega$. The idea is to replace $u$ in a small neighbourhood of $x_0$ by a map which is almost conformal. This map must have a very concentrated energy and its image must be almost all $\mathbb{B}^n$. By using such a replacement we first show that there exists $v\in \I_{k+1}$ such that $E_n(v) \leq E_n(u) + n^{\frac{n}{2}-1}|\bB^n|+ \delta$, for $\delta$ small. We then explain how to iterate the construction and how to change the degree negatively.
  Without loss of generality we can assume that $x_0=0$ and $\nu(x_0)=e_n\eqqcolon (0,\dots,0,1)$ where $\nu$ denotes the unit outer normal to $\partial \O$. This is always possible, up to translating and rotating $\O$. By density\footnote{This follows from the density of \(\C^\infty(\p \Omega,\mathbb{S}^{n-1})\) into \( W^{1-\frac{1}{n},n}(\p \Omega,\mathbb{S}^{n-1})\), cf.\ \cite[Theorem 2.1]{Mironescu_2007} }of \(\C^\infty(\overline{\Omega},\R^n)\cap \I\) we can also suppose that $u$ is smooth and, up to multiply $u$ by a suitable rotation, we can suppose that $u(x_0)=e_n$.
 
\noindent \textbf{Step1: The case where $\partial \Omega$ is flat near $x_0$} 

\noindent We first assume that there exists $\sigma$ small enough such that $B(x_0,2\sigma) \cap \partial \Omega =\{x\in \R^n ; x_n =0\} \cap B(x_0,2\sigma)$ and that $\Omega \subset \{x \in \R^n ; x_n >0\}$. In the sequel we use the notation $B(x_0,r)=B_r$ for $r>0$. We consider the following conformal maps
\begin{equation*}
\begin{array}{rcll}
I_\sigma :\R^n \setminus \{0\} & \rightarrow& \R^n \\
x & \mapsto & \sigma^2\, \frac{x}{ |x|^2},
\end{array}
\end{equation*}
and
\begin{equation*}
\begin{array}{rcll}
P:\R^n \setminus \{-e_n\} & \rightarrow & \R^n \\
x & \mapsto & \left( \frac{-2x_1}{x_1^2+\dots+x_{n-1}^2+(1+x_n)^2},\dots, \frac{-2x_{n-1}}{x_1^2+\dots+x_{n-1}^2+(1+x_n)^2},\frac{1-x_1^2-\dots-x_{n-1}^2}{x_1^2+\dots+x_{n-1}^2+(1+x_n)^2}  \right).
\end{array}
\end{equation*}
We can see that \(P\) and \(I_\sigma\) are conformal maps, \( P\) sends the upper-half plane into \(\mathbb{B}^n\) and \(P(0)=e_n\). Furthermore, since \(I_\sigma\) and \(P\) reverse the orientation, the map 
\[\pi_\sigma \coloneqq P\circ I_\sigma\colon \R^n\setminus \{0\}\to \bB^n\] preserves the orientation.

Now we let \(\tau=\sigma +\sigma^{\left( 1+\frac{1}{2(n-1)}\right)}\) and we define $L_{\sigma}$ using the variables $r=|x|$ and $\theta=\frac{x}{|x|}$ by 
\begin{equation}\label{eq:defLsigma}
L_\sigma(x)= \left( \frac{r-\sigma}{\tau-\sigma}|u(\tau\theta)|+\frac{\tau-r}{\tau-\sigma}|\pi_\sigma (\sigma \theta)| \right) \frac{\frac{r-\sigma}{\tau-\sigma}\, u(\tau\theta)+\frac{\tau-r}{\tau-\sigma}\, \pi_\sigma (\sigma \theta)}{|\frac{r-\sigma}{\tau-\sigma}\, u(\tau\theta)+\frac{\tau-r}{\tau-\sigma}\, \pi_\sigma (\sigma \theta)|}.
\end{equation}
This map is well-defined for $\sigma$ small enough. Indeed, on one hand, by using Taylor's expansion Theorem, for all \(\theta\in \mathbb{S}^{n-1}\) we have
\begin{equation}\label{eq:Taylor_exp}
u(\tau \theta) = u(x_0)+\cO(\tau)=e_n +\cO(\sigma).
\end{equation}
On the other hand, from the definition of \(I_\sigma\) we see that \(I_\sigma \left(\p B_\sigma \cap \Omega\right)=\p B_\sigma \cap \Omega\). Since \(P\) is smooth near \(\{x\in \R^n : x_n\geq 0\}\) we see that 
\begin{equation}\label{eq:image_of_bubble}
 P\left(I_\sigma \left(\p B_\sigma \cap \Omega\right) \right)\subset \mathbb{B}^n\setminus B_{C\sigma}
 \end{equation}
  for some constant \(C>0\). Hence, for any \(\theta\in \mathbb{S}^{n-1}\)
\begin{equation*}
\pi_\sigma(\sigma\theta) = \pi_\sigma(x_0)+ \cO(\sigma)= e_n+\cO(\sigma).
\end{equation*}
This proves that, for $\sigma$ small enough
\begin{equation}\label{eq:denominator}
\frac{r-\sigma}{\tau-\sigma}\, u(\tau\theta)+\frac{\tau-r}{\tau-\sigma}\, \pi_\sigma (\sigma \theta)=u(0)+\cO(\sigma)=e_n+\cO(\sigma) \neq 0 .
\end{equation}
Consequently, Definition \eqref{eq:defLsigma} is well-posed for $\sigma$ small enough. Moreover, we have $L_\sigma (x)= \pi_\sigma (x)$, for all $x \in \partial B_\sigma \cap \Omega$, $L_\sigma (x)=u(x)$ for all $x\in \partial B_\tau \cap \Omega$ and $|L_\sigma (x)|=1$ for $x \in \partial \Omega \cap \left( B_\tau \setminus B_\sigma \right)$. Note that this last property would have been lost if we would have used a linear interpolation between \(u\) and \(\pi_\sigma\).

\noindent  We define
\begin{equation}
v_\sigma =\begin{cases} 
\pi_\sigma (x)  & \text{ if } x \in B_\sigma \cap \Omega, \\[2mm]
L_\sigma (x) & \text{ if } x \in \left( B_\tau \setminus B_\sigma \right) \cap \Omega, \\[2mm]
u(x) & \text{ if } x \in \Omega \setminus B_\tau.
\end{cases}
\end{equation} 
By construction, it holds $|v_{\sigma}|=1$ on $\pl\Omega$ which means that \(v_\sigma \in \I\).

\textbf{Computation of the energy of $v_\sigma$.} We have that 
\begin{equation}
E_n(v_\sigma)=E_n(u)-E_n(u,B_\tau\cap \Omega)+E_n\left(L_\sigma, \left( B_\tau \setminus B_\sigma \right)\cap \Omega \right)+ E_n(\pi_\sigma, B_\sigma \cap \Omega).
\end{equation}

\textit{Estimate of $E_{n}(\pi_\sigma, B_\sigma \cap \Omega)$}\\
Since $\pi_\sigma$ is conformal, up to a dimensional constant, its $n$-energy is equal to the area of its image counted with multiplicity, cf.\ Lemma \ref{lowbound}. Thus, by using \eqref{eq:image_of_bubble} we obtain that for some \(C>0\),

\begin{equation}\label{eq:est_pi}
E_n(\pi_\sigma,B_\sigma \cap \Omega)= n^{\frac{n}{2}-1}|\bB^n|-C\, \sigma^n+\cO(\sigma^{n+1}).
\end{equation}

\textit{Estimate of $E_n\left(L_\sigma, \left( B_\tau \setminus B_\sigma \right)\cap\Omega \right)$}\\
In order to shorten the notations, we denote
\[ V_\sigma \coloneqq \frac{r-\sigma}{\tau-\sigma}\, u(\tau\theta)+\frac{\tau-r}{\tau-\sigma}\, \pi_\sigma (\sigma \theta), \text{ and } \quad 
R_\sigma \coloneqq \frac{\frac{r-\sigma}{\tau-\sigma}\, u(\tau\theta)+\frac{\tau-r}{\tau-\sigma}\, \pi_\sigma (\sigma \theta)}{|\frac{r-\sigma}{\tau-\sigma}\, u(\tau\theta)+\frac{\tau-r}{\tau-\sigma}\, \pi_\sigma (\sigma \theta)|}.
\]
We compute the radial and the tangential parts of the differential of $L_\sigma$. It holds that 
\[ 
\partial_r L_\sigma =\frac{|u(\tau \theta)|-|\pi_\sigma(\sigma \theta)| }{\tau-\sigma}\, R_\sigma + \left( \frac{r-\sigma}{\tau-\sigma}|u(\tau\theta)|+\frac{\tau-r}{\tau-\sigma}|\pi_\sigma (\sigma \theta)| \right)\partial_r R_\sigma .\]
Using the Taylor's expansion \eqref{eq:Taylor_exp} and \eqref{eq:image_of_bubble} we obtain that $|u(\tau \theta)|-|\pi_\sigma(\sigma \theta)|=\cO(\sigma)$. We also have 
\[ 
\left| \partial_r R_{\sigma} \right|\leq 2\, \frac{\left|\partial_r \left( \frac{r-\sigma}{\tau-\sigma}\, u(\tau\theta)+\frac{\tau-r}{\tau-\sigma}\, \pi_\sigma (\sigma \theta) \right)  \right|}{\left| \frac{r-\sigma}{\tau-\sigma}\, u(\tau\theta)+\frac{\tau-r}{\tau-\sigma}\, \pi_\sigma (\sigma \theta) \right|}.
\]
By \eqref{eq:denominator}, we obtain that for $\sigma$ small enough, 
we have $|\frac{r-\sigma}{\tau-\sigma}u(\tau\theta)+\frac{\tau-r}{\tau-\sigma}\pi_\sigma (\sigma \theta)|\geq \frac{1}{2} $. Hence, 
\[
|\partial_rR_\sigma | \leq \frac{C\, \sigma}{\tau-\sigma}.
\]
We can conclude that $|\partial_rL_\sigma| \leq \frac{C\, \sigma}{\tau-\sigma}$ and 
\begin{align*}
\int_{(B_\tau \setminus B_\sigma)\cap \Omega} |\partial_r L_\sigma |^n &\leq  \frac{C\,  \sigma^n}{ \sigma^{n+\frac{n}{2(n-1)}} }(\tau^n-\sigma^n) \leq  \frac{C\, \sigma^{1+\frac{1}{2(n-1)}} \sigma^{n-1}}{\sigma^{\frac{n}{2(n-1)}}}\leq  C\, \sigma^{n-\frac{1}{2}}. \nonumber
\end{align*}

In order to estimate the tangential derivatives of \(L_\sigma\) we first observe that, if we denote by \(\nabla_{\tan}\) the tangential part of the differential, we have
\begin{equation}
\nabla_{\tan} V_\sigma =\tau \frac{r-\sigma}{\tau-\sigma}\nabla_{\tan}u (\tau \theta)+ \sigma \frac{\tau-r}{\tau-\sigma} \nabla_{\tan} \pi_{\sigma}(\sigma \theta).
\end{equation}

We notice that , for any \(x \in \R^n \setminus \{0\}\),
\[ \dd I_\sigma (x).h =\sigma^2 \left(\frac{h}{|X|^2}-2 \frac{X\cdot h}{|X|^4}X \right).\]
Hence, for all \(\theta \in \mathbb{S}^{n-1}\) we have that \(|\nabla_{\tan} I_\sigma (\sigma \theta)| \leq C\) for some \(C>0\) and, since \(P\) is smooth near \(\Omega\cap B_\tau\), we find that \(|\nabla_{\tan} \pi_\sigma| \leq C\) for some \(C>0\).
This implies that \(|\nabla_{\tan} V_\sigma| \leq C\sigma/(\tau-\sigma) \) and thus
\begin{equation}
|\nabla_{\tan} L_\sigma| \leq \frac{C \sigma}{\tau-\sigma}.
\end{equation}
Hence, as in the case of the radial derivative, we obtain  \(\int_{(B_\tau \setminus B_{\sigma})\cap \Omega}|\nabla_{\tan} L_\sigma|^n \leq C \sigma^{n-\frac{1}{2}}\) and finally
\begin{align} \label{eq:Energy_Lsigma}
	\int_{(B_\tau \setminus B_{\sigma})\cap \Omega} |\dd L_\sigma  |^n \leq C\, \sigma^{n-\frac{1}{2}}.
\end{align}

\textbf{Computation of the degree of $v_\sigma$.} Now we show that $\deg(v_\sigma, \partial \Omega)= \deg(u,\partial \Omega)+1$ for $\sigma$ small enough. Thanks to \eqref{eq:Energy_Lsigma}, we have 
\begin{eqnarray*}
\deg(v_\sigma, \partial \Omega) &=& \frac{1}{|\bB|^n}\int_\Omega \Jac (v_\sigma)  \\[2mm]
& =& \frac{1}{|\bB^n|}\left[\int_{\Omega \setminus B_\tau} \Jac (u) +\int_{B_\sigma \cap \Omega} \Jac (\pi_\sigma) + \int_{B_\tau\setminus B_\sigma} \Jac (L_\sigma)  \right]  \\[2mm]
&=& \deg(u,\partial \Omega) -\smallO_\sigma(1)+ \frac{1}{|\bB^n|}\big( |\bB^n|-C\, |\bB^n|\, \sigma^n \big) +\smallO_\sigma(1)  \\[2mm]
&=& \deg(u,\partial \Omega)+1+\smallO_\sigma(1). 
\end{eqnarray*}
To estimate the Jacobian of $\pi_{\sigma}$, we proceeded as in \eqref{eq:est_pi}.
Since the degree is an integer, for $\sigma$ small enough we obtain the result. 

\textbf{Step 2: The case where $\partial \Omega$ is not flat near $x_0$}

We still assume that $x_0=0$, $\nu(x_0)=e_n$ and that the tangent space to $\partial \Omega$ at $x_0$ is then given by $T_{x_0} \partial \Omega =\{x \in \R^n ; x_n =0 \}$. We set \(H\eqqcolon \{x\in \R^n; x_n>0\}\). Since we assumed that $\partial \Omega$ is \(\C^1\), there exist $\tau_0>0$ small and a $\mathcal{C}^1$-diffeomorphism on its image $f\colon B_{\tau_0}\cap \overline{H} \rightarrow \Omega$ such that 
\begin{align*}
	f(0)=e_n, \qquad \text{ and }\qquad \dd f(0)=\id.
\end{align*}
We let $B_\tau^-\coloneqq B_{\tau}\cap \overline{H}$ and $V_\tau \coloneqq \Phi(B_\tau^-)$, for $0<\tau <\tau_0$. There exists a constant $C>0$ such that if $\tau < \frac{\tau_0}{2}$ then 
\begin{equation}\label{estimeediffeo}
	\|f-\id\|_{\mathcal{C}^1(B_\tau^-)} \leq  C\, \tau,  \quad \text{ and } \| f^{-1} -\id\|_{\mathcal{C}^1(V_\tau)} \leq C\, \tau.
\end{equation} 
Let $v\in \C^1\left(  V_\tau,\R^n\right)$. We set $\tilde{v}=v\circ \Phi:B_\tau^- \rightarrow \R^n$. We then have, by using \eqref{estimeediffeo} and the following change of variable $y=f(x)$:
\begin{align*}
& \left| E_n (v,V_\tau) - \frac{1}{n} \int_{B^-_{\tau}} |\dd \tilde{v}|^n \right|  \\[2mm]
& =\left| \int_{V_\tau} |\dd(\tilde{v}\circ f^{-1})|^n - \int_{B^-_\tau}|\dd \tilde{v}|^n \right| \\[2mm]
 & =  \left| \int_{V_\tau}  |\dd\tilde{v}(f^{-1}(x)) \circ\dd (f^{-1})(x)|^n \dx - \int_{B^-_\tau}|\dd\tilde{v}|^n \right|  \\[2mm]
 & = \left| \int_{B^-_\tau}  |\dd\tilde{v}(y) \circ\dd (f^{-1})(f(y))|^n\, |\det \dd f (y)|\dy - \int_{B^-_\tau}|\dd\tilde{v}|^n \right|  \\[2mm]
& \leq C\, \tau\, \int_{B_{\tau}^-}|\dd \tilde{v}|^n \leq C\, \tau^{n+1}\, \|v\|_{\C^1(V_{\tau})}^n. 
\end{align*}
 We then obtain:
\begin{eqnarray}
E_n(v,V_\tau)=\int_{B_\tau^-} |\dd \tilde{v}(y)|^n\, \dd y + \cO(\tau^{n+1}). \nonumber
\end{eqnarray}
Since we assumed that \(u\in \C^\infty(\overline{\Omega},\R^n)\cap \I\) we can say that
\begin{eqnarray}
E_n(u)&=&E_n(u,V_\tau)+E_n(u,\Omega \setminus V_\tau) \nonumber \\[2mm]
&=& E_n(\tilde{u},B_\tau^-)+E_n(u,\Omega\setminus V_\tau)+ \cO_u(\tau^{n+1}), \nonumber
\end{eqnarray}
where \(\tilde{u}=u\circ f\) and \(\cO_u(\tau^{n+1})\) denotes a function bounded by \(\tau^{n+1}\) but with a constant possibly depending on \(u\).
But, using Step 1, for 
$\tau=\sigma+\sigma^{1+\frac{1}{2(n-1)}}$ and $\sigma$ small, we can find $\tilde{\chi}_\tau\colon B_\tau^- \rightarrow \R^n$ such that \(\tilde{\chi}_\tau\) agrees with \(\tilde{u}\) on \(\p B_\tau \cap \overline{H}\) and 
\[
E_n(\tilde{\chi}_\tau,B_\tau^-)= n^{\frac{n}{2}}\, |\bB^n| +\cO(\tau).
\]
 We then set $\chi_\tau =\tilde{\chi}_\tau\circ f^{-1}$ and 
\[v_\sigma =\begin{cases}
\chi_\tau & \text{if } x\in \Omega \cap V_\tau, \\
u & \text{if } x\in \Omega \setminus V_\tau.
\end{cases} \]
We thus have 
\[
	E_n(v_\sigma)\leq E_n(u)+n^{\frac{n}{2}}|\bB^n|+\cO_u(\sigma),  
\]
and
\[
	\deg(v_\sigma,\partial \Omega) \ust{\sigma\to 0}{=} \deg(u,\partial \Omega)+1 +\cO_{u}(1).
\]

Thus, for $\sigma$ small enough, $v_\sigma$ satisfies the desired properties. 

Since this is a local construction we can repeat it near several different points in order to increase the degree by several units. We can also decrease the degree by using maps which reverse the orientation.
\end{proof}

We can now conclude the proof of Theorem \ref{th:main_1}.

\begin{proof}[Proof of Theorem \ref{th:main_1}.]
Thanks to Lemma \ref{lowbound} we have 
\[
 m(k,\O) \geq |k|\, n^{\frac{n}{2}}\, |\bB^n|.
\]
However, by using the fact that $m(0,\O)=0$ and by using Lemma \ref{testfunc}, we obtain that 
\[
	m(k,\Omega)\leq  |k|\, n^{\frac{n}{2}}\, |\bB^n|. 
\] Hence it holds 
\[
m(k,\Omega)=  |k|\, n^{\frac{n}{2}}\, |\bB^n|.
\]
Now let us assume that there exists $u \in \I_k$ such that $E_n(u)=m(k,\O)$. Using Lemma \ref{lowbound} again, we obtain that $u$ is conformal, in the sense that \(\dd u^T \dd u= | \Jac \, u| \id \) a.e.\ in \(\Omega\) and \(\Jac \, u\geq 0\) a.e.\ or \(\Jac \, u \leq 0\) a.e.\ in \(\Omega\). We can thus apply the Liouville Theorem \ref{th:Liouville} to deduce that  $\O=\bB^n$. Since all the non-constant conformal maps given by the Liouville Theorem are bijective, they have a degree equal to $\pm 1$ (the sign depending if the map preserves or reverses the orientation). 
\end{proof}

\section{Almost Möbius transforms}\label{sec:AlmostMobius}

The goal of this section is to construct a family of paths \(\chi_{r}\colon \p \Omega \rightarrow \I_1\), for \(r>0\) small, which are continuous and which satisfy \eqref{eq:minimal_energy} and \eqref{eq:topo_condition}. These are the crucial tools for showing that our minmax scheme has the mountain pass geometry and does produce critical points for the subcritical energies \(E_{n+\alpha}\) for \(\alpha>0\). If \(\Omega=\mathbb{B}^n\), a natural choice for \(\chi_r\) is \(\chi_r(a)= M_{(1-r)a}\) for \(0<r<1\) and \(a\in \mathbb{B}^n\), where \(M_{(1-r)a}\) is defined in Definition \ref{def:Mobius} below. We thus want to pull-back these M\"obius maps by a diffeomorphism  from \(\overline{\Omega}\) to \(\overline{\mathbb{B}^n}\). However, since there is no conformal diffeomorphism from \(\overline{\Omega}\) to \(\overline{\mathbb{B}^n}\) if \(\Omega\) is not a round ball, we cannot expect \eqref{eq:minimal_energy} to hold for a fixed diffeomorphism and for every \(a\in \p \Omega\). Our solution will be to use a different diffeomorphism for every \(a\in \overline{\Omega}\) that we denote by \(\Phi_a\) and which is almost conformal near \(a\in \overline{\Omega}\). We point out that we need to construct these diffeomorphisms for all \(a\in \overline{\Omega}\) and not only for \(a\in \p \Omega\) because this will be useful to estimate the critical value \(c(\alpha,r)\) cf.\ Proposition \ref{prop:prochegroundstate}. Once we have these diffeomorphisms, denoted by \(\Phi_a\),  we can show that \(\chi_r(a)= M_{(1-r)\Phi_a(a)}\circ \Phi_a\) satisfies the desired properties.

\subsection{M\"obius maps from the ball to the ball}

In this section we define M\"obius maps and examine some of their properties, for which we also refer to \cite{Stoll_2016} .

\begin{definition}\label{def:Mobius}
	Given $a\in \bB^n$, we define, for any $x\in \R^n\setminus \{a\}$,
	\begin{align}\label{eq:def_psia}
		\psi_a(x)\coloneqq  a + (1-|a|^2)\, \frac{a-x}{|a-x|^2}.
	\end{align}
	The map $M_a \coloneqq \frac{\psi_a}{|\psi_a|^2}$ is called a Möbius transform. More generally, we call M\"obius transform of \(\mathbb{B}^n\) any map of the form \(R \circ M_a\) for some \(R\in O_n(\R)\) and \(a\in \mathbb{B}^n\).
\end{definition}

\begin{remark}
In dimension 2,  for \(a \in \bB^n\) and \(x\in \R^n \setminus \{a\}\) we have
\begin{align}
a+ (1-|a|^2)\frac{a-x}{|a-x|^2}&=a+(1-|a|^2)\frac{1}{\overline{a-x}} \\
& =\frac{(\overline{a-x})a+1-|a|^2 }{\overline{a-x}}= \frac{|a|^2-a\overline{x}+1-|a|^2}{\overline{a-x}}.
\end{align}
Hence, \(M_a(x)=-\overline{\left(\frac{x-a}{1-\overline{a}x} \right)}\) and we recognize, up to complex conjugation, the complex M\"obius transform.
\end{remark}

We have the following description of $M_a$.

\begin{lemma}\label{lem:conformality_Mobius}
	For any $a\in\bB^n$, the map $M_a\colon  \overline{\bB}^n \to \overline{\bB}^n$ is a bijective conformal map sending $\s^{n-1}$ onto $\s^{n-1}$ and the point $a$ to the origin. For \(x\in \mathbb{B}^n\) we have the explicit formula
	\begin{equation}\label{eq:formule_Mobius}
	 M_a(x)=\frac{|a-x|^2a+(1-|a|^2)(a-x)}{|a|^2|a-x|^2+(1-|a|^2)^2+2a\cdot (a-x)(1-|a|^2)}.
	\end{equation}
\end{lemma}

\begin{proof}
	$M_a$ is a composition of translations, dilations and inversions. Hence, $M_a$ is a conformal map. 
	Let $\iota_a$ be the inversion with respect to the sphere $S_a$ of center $-a$ and radius $\sqrt{1-|a|^2}$. The map $\iota_a$ is given by the formula
	\begin{align*}
		\forall x\in\R^n \setminus \{-a\},\ \ \ \iota_a(x) = -a - (1-|a|^2)\frac{a+x}{|a+x|^2}.
	\end{align*}
	Then $\psi_a(x) = -\iota_a(-x)$. By construction, using that \( |a|^2+1-|a|^2=1\) we can see that  the intersection of the sphere $S_a$ with $\s^{n-1}$ is orthogonal along an equator of $S_a$. Since $\iota_a$ sends hyperspheres on hyperspheres, since it preserves angles and since the sphere \(S_a\) is unchanged by \(\iota_a\) we obtain that $\iota_a(\s^{n-1})=\s^{n-1}$. Thus, $\iota_a(\bB^n) = (\R^n\cup\{\infty\})\setminus \bB^n$. Hence, $\psi_a(\bB^n) = (\R^n\cup\{\infty\})\setminus \bB^n$ with $\psi_a(\s^{n-1}) = \s^{n-1}$. Thus, $M_a(\bB^n) = \bB^n$ with $M_a(\s^{n-1}) = \s^{n-1}$. Furthermore, $\psi_a$ sends $a$ to $\infty$. Thus, $M_a(a) = 0$.
\end{proof}

We can observe that \(M_a\) concentrates when \(a\in \mathbb{B}^n\) goes to the boundary \(\p \mathbb{B}^n=\mathbb{S}^{n-1}\). This means for example that  \(M_a\) converges to the constant map equal to \(a\) a.e.\ but \(|\dd \, M_a|\) tends to \(+\infty\). By using Lemma \ref{lem:conformality_Mobius} we can identify a rate of concentration.

\begin{lemma}\label{lm:unif_mobius}
	There exists $\Lambda>0$ such that 
	\begin{align}\label{eq:uniform_trace_mobius}
		\sup_{a\in\bB^n} \|M_a\|_{W^{1-\frac{1}{n},1}(\mathbb{S}^{n-1})} \leq \Lambda.
	\end{align}
	Furthermore, we have the following convergences:
	\begin{align}\label{eq:uniform_convergence_boundary_mobius}
		\sup_{a\in\mathbb{S}^{n-1}} \| M_{(1-r)a} -a \|_{L^1(\bB^n)} \xrightarrow[r\to 0]{}{0},
	\end{align}
	\begin{align}\label{eq:DL_Mobius}
	\int_{\mathbb{B}^n\cap B\left( a,r^{1/2} \right)} |\dd  M_{(1-r)a}|^n \ust{r\to 0}{=} n^{\frac{n}{2}}\, |\mathbb{B}^n|+
	\cO(r),
	\end{align}
	and the last estimate is uniform in \(a\in \mathbb{S}^{n-1}\).
\end{lemma}

\begin{proof}
	By using the conformal invariance of the \(n\)-energy and the fact that each \(M_a\) is conformal we find that
	\begin{equation}
	\int_{\mathbb{B}^n}| \dd  M_a|^n=\int_{\mathbb{B}^n}|\dd( \id)|^n= n^{\frac{n}{2}} |\mathbb{B}^n|.
	\end{equation}
	Hence we have that \(\sup_{a\in \mathbb{B}^n} \|M_a\|_{W^{1,n}(\mathbb{B}^n)}\leq ((1+n^{\frac{n}{2}})|\mathbb{B}^n|)^{1/n}\). We obtain \eqref{eq:uniform_trace_mobius} by trace theorem.
	
	\medskip
	
	Let $a\in\s^{n-1}$. To prove \eqref{eq:uniform_convergence_boundary_mobius} we use the expression of \(\psi_{(1-r)a}\) in \eqref{eq:def_psia} to write that 
	\begin{align}
	\psi_{(1-r)a}(x)&=(1-r)\, a+(1-(1-r)^2)\, \frac{(1-r)\, a-x}{|(1-r)\, a-x|^2}. \label{eq:explicit_psi_ra}
	\end{align}
	If $|x-a|>r^{\frac{1}{4}}$, then we have, for $r<\frac{1}{10}$,
	\begin{align*}
		|(1-r)a - x|^2 \geq \left( r^{\frac{1}{4}} -r \right)^2 \geq \frac{r^{1/2}}{2}.
	\end{align*}
	Moreover, given any $x\in\bB^n$, we have 
	\begin{align*}
		& \left| a-r\, a+(2r-r^2) \frac{a-x-r\, a}{|a-x-r\, a|^2}\right|^2 \\[2mm]
		& = 1+r^2 + \frac{ (2r-r^2)^2 }{ |a-x-r\, a|^2 } - 2r + 2(2r-r^2) \frac{ a\cdot (a-x-r\, a) }{|a-x-r\, a|^2} \\[2mm]
		&\quad -2r\, (2r-r^2) \frac{ a\cdot (a-x-r\, a) }{ |a-x-r\, a|^2 }.
	\end{align*}
	If $r>0$ is small enough, since $|a-x|\leq r^{1/4}$, we have $|a-x-r\, a|\geq | r-|x-a| | \geq r^{1/4}(1-r^{3/4})$. Therefore, 
	\begin{align*}
		\left| \left| a-ra+(2r-r^2) \frac{a-x-r\, a}{|a-x-r\, a|^2}\right|^2 -1 \right| \leq& r^2  + 8\frac{ r^2 }{r^{1/2}} +2r + 8\frac{r}{r^{1/2}} + 8 \frac{r^2}{r^{1/2}}.
	\end{align*}
	Thus, we obtain
	\begin{align}\label{eq:denom}
		 \left| a-r\, a+(2r-r^2)\, \frac{a-x-r\, a}{|a-x-r\, a|^2}\right|^2 \underset{r\to 0}{=} 1+\cO(r^{1/2}), 
	\end{align}
	where \(\cO(r^{1/2})\) is independent of \(a\in \mathbb{S}^{n-1}\). 
	
	In the case where $x\in \bB^n\setminus B(a,r^{\frac{1}{4}})$, we obtain
	\begin{align*}
		\left| M_{(1-r)a}(x) -a \right| & = \left| \frac{a-ra+(2r-r^2) \frac{a-x-ra}{|a-x-ra|^2}}{\left| a-ra+(2r-r^2) \frac{a-x-ra}{|a-x-ra|^2}\right|^2}-a \right|\leq C\, r^{\frac{1}{2}}.
	\end{align*}
Therefore, it holds
\begin{align*}
\int_{\mathbb{B}^n}|M_{(1-r)a}-a| &\leq \int_{\mathbb{B}^n \cap B(a,r^{\frac14})} |M_{(1-r)a}-a|+C\, r^{n/4} \\[2mm]
& \leq \int_{\mathbb{B}^n \cap B(a,r^{\frac14})} \left| \frac{a-ra+(2r-r^2) \frac{a-x-ra}{|a-x-ra|^2}}{\left| a-ra+(2r-r^2) \frac{a-x-ra}{|a-x-ra|^2}\right|^2}-a \right| +C\, r^{1/2}.
\end{align*}
Using \eqref{eq:denom}, we find that 
\begin{align*}
	& \int_{\mathbb{B}^n}|M_{(1-r)a}-a| \\[2mm]
	&= \int_{\mathbb{B}^n \cap B(a,r^{\frac14})}\left| \left(a-ra+(2r-r^2)\frac{a-x-ra}{|a-x-ra|^2}\right) (1+\cO(r^{1/2}))-a\right| +\cO(r^{n/4}) \\[2mm]
&= \int_{\mathbb{B}^n \cap B(a,r^{\frac14})}\left|-ra+(2r-r^2)\frac{a-x-ra}{|a-x-ra|^2}\right| (1+\cO(r^{1/2}))+\cO(r^{1/2}) .
\end{align*}
We then use that \(|a-x-ra|^2\geq r^{1/2}-4r+r^{2}\) in the region \(\mathbb{B}^n \cap B(a,r^{\frac14})\) and hence we find that for any \(a\in \mathbb{S}^{n-1}\)
\[ \int_{\mathbb{B}^n}|M_{(1-r)a}-a| \leq Cr^{1/2}\]
for some \(C>0\) independent of \(a\in \mathbb{S}^{n-1}\) for \(r\) small enough. This concludes the proof of \eqref{eq:uniform_convergence_boundary_mobius}.

\medskip

To prove \eqref{eq:DL_Mobius}, we use that Möbius maps are conformal to obtain that 
\begin{align*}
\int_{\mathbb{B}^n \cap B\left(a,r^{\frac12}\right) } |\dd M_{(1-r)a}|^n= \left| M_{(1-r)a}\left(\mathbb{B}^n \cap B(a,r^{1/2}) \right) \right|.
\end{align*}
We must then determine the image of \(\mathbb{B}^n \cap B(a,r^{\frac12})\) by the M\"obius map \(M_{(1-r)a}\) for \(r>0\) small. We know that \(M_{(1-r)a}\) is a bijection of the unit ball \(\mathbb{B}^n\) and it sends \((1-r)a\) to 0.  The image of the piece of hypersphere \(\mathbb{B}^n \cap \p B(a,r^{1/2})\) by \(M_{(1-r)a}\) separates the disc into two regions. The image \(M_{(1-r)a}\left( \mathbb{B}^n \cap B(a,r^{1/2})\right) \) is the part containing zero and \(M_{(1-r)a}\left( \mathbb{B}^n \setminus B(a,r^{1/2}) \right)\) is the part containing \(a\). We deduce from the explicit definition of $\psi_{(1-r)a}$ in \eqref{eq:explicit_psi_ra} that there exists a constant $C>0$ independant of $a$ satisfying, for $x\in \partial B(a,\sqrt{r})$ 
\begin{align*}
	|\psi_{(1-r)a}(x) -a|\leq C\sqrt{r}.
\end{align*}
This implies that \(M_{(1-r)a} \left( \mathbb{B}^n \setminus B(a,r^{1/2}) \right) \) is contained in a ball \(B(a,Cr^{1/2})\) for some \(C>0\) independent of \(r\), hence \( \left|M_{(1-r)a} \left(\mathbb{B}^n \setminus B(a,r^{1/2}) \right) \right| \leq Cr^{n/2}\) and by complementarity
we obtain \eqref{eq:DL_Mobius}.
\end{proof}

\subsection{Modifying a diffeomorphism so that it is conformal in one point}

The main object of this section is the following Proposition.

\begin{proposition}\label{continuityscheme1_v2}
	For all \(\e>0\), there exists \(L>0\) small enough  such that if \(\Phi\colon \overline{\Omega}\rightarrow \mathbb{B}^n\) is a \(\C^1\)-diffeomorphism satisfying
	\begin{equation}\label{eq:proche_id}
	\| \Phi-\id \|_{\C^1 (\Omega)} \leq L, 
	\end{equation} 
	then for all \(a\in \overline{\Omega}\), there exists \(\Phi_{a}\colon \overline{\Omega}\rightarrow \mathbb{B}^n\) a \(\C^1\)-diffeomorphism such that 
	\begin{enumerate}
		\item $\Phi_{a}(a) = \Phi(a)$.
	\item If $a\in\pl\Omega$, then
	\begin{itemize}
	\item[i)] there exist \(\beta_a\in \R\) and $R_a \in O_n(\R)$ such that $\dd \Phi_{a}(a) = \beta_a R_a$,
	\item[ii)] \label{item:distortion2} 
		\begin{equation}\label{eq:distortiond}
		\| \Phi_{a} -\id \|_{\C^1(\Omega)}<\e.
		\end{equation}
	\end{itemize}
	\item \label{item:last} \(\Phi_{a'} \xrightarrow[a' \to a]{a.e.} \Phi_a, \quad \Phi_{a'}^{-1} \xrightarrow[a' \to a]{a.e.} \Phi_a^{-1}\) and \( \dd \Phi_{a'} \xrightarrow[a' \to a]{a.e.} \dd \Phi_a, \quad \dd \Phi_{a'}^{-1} \xrightarrow[a' \to a]{a.e.} \dd \Phi_a^{-1}\).
	\end{enumerate}
\end{proposition}

The idea to prove this Proposition is first to observe that the important point is to construct \(\Phi_a\) for \(a\) on \(\p \Omega\). Indeed, once it is done, we can extend it for \(a\) in a tubular neighbourhood of \(\p \Omega\) and then glue it with \(\Phi\) well inside the domain thanks to a cut-off function. For \(a\in \p \Omega\), we can use a polar decomposition to write that \(\dd \Phi(a)= R_a S_a\) with \(R_a\in O_n(\R)\) and \(S_a\in S_n^{++}(\R)\). The idea is to find a neighbourhood of \(a\in \p \Omega\), denoted by \(U_a\), and a diffeomorphism \(\psi_a:U_a \rightarrow U_a\) such that \(\dd (\Phi\circ \psi_a)=\beta_a R_a\) for some \(\beta_a\in \R^+\) well-chosen. It is actually simpler to consider the case where \(\p \Omega\) is flat and we will use a local straightening of the boundary to be in this situation. Then the diffeomorphism \(\psi_a\) is constructed as the flow of a well-chosen vector-field at time \(t=1\). Along the construction we have to ensure a sort of continuity in the parameter \(a\in \Omega\) to ensure that \ref{item:last} holds.

\begin{proof}

\underline{Step 1:} We first construct \(\Phi_{a}\) for \(a\in \p \Omega\).

\begin{claim}\label{claim:straightening}(Straightening of the boundary)

Let \( a\in \p \Omega\), there exist a number \(r_a>0\), a rotation \(O_a\in O_n(\R)\) and \(f_a:\overline{\Omega}\cap B(a,r_a)\rightarrow\R^n\) such that
\begin{enumerate}
\item \label{item_1}\(f_a(a)=0\),
\item \label{item_2}\(\dd f_a(a)=O_a\),
\item \label{item_3}\(f(\Omega\cap B(a,r_a))\subset \R^n_+=\{x\in \R^n; x_n>0\}\),
\item \label{item_4}\(f(\p \Omega\cap B(a,r_a))\subset \{x\in \R^n; x_n=0\}\),
\item \label{item_5}\(f\) is a \(\C^1\)-diffeomorphism onto its image. We denote by \(U_a:=f_a(\Omega\cap B(a,r_a))\) and \(\overline{U}_a:=f_a(\overline{\Omega}\cap B(a,r_a))\).
\item \label{item_6}The maps \(a \in \p \Omega \mapsto r_a\) and \(a\in \p \Omega \mapsto O_a\) are continuous and for all \(a\in \p \Omega\), for all \(\e>0\), there exists \(\delta_a>0\) such that if \( |a-a'|<\delta_a\) then 
\[\|f_a-f_{a'}\|_{L^\infty(B(a,r_a)\cap B(a',r_{a'}))}+\|\dd f_a -\dd f_{a'}\|_{L^\infty(B(a,r_a)\cap B(a',r_{a'}))}\|<\e.\]
\end{enumerate}
\end{claim}

\begin{proof}(proof of Claim \ref{claim:straightening})
Let \(a\in \p \Omega\), since \(\Omega\) is a regular open set of class \(\C^1\), by definition, there exist \(r_a>0\) and \(\rho_a\in \C^1(B(a,r_a))\) such that
\begin{itemize}
\item \(\Omega \cap B(a,r_a)=\{x\in B(a,r_a); \rho_a(x)>0\}, \)
\item \(\p \Omega \cap B(a,r_a)=\{x\in B(a,r_a); \rho_a(x)=0\}\),
\item \(\nabla \rho_a(x)\neq 0\) for all \(x\in B(a,r_a)\).
\end{itemize}
The outward unit normal to \(\Omega\) at \(x \in \p \Omega\cap B(a,r_a) \) is then given by \(\nu(x)=-\nabla \rho_a(x)/|\nabla \rho_a(x)|\) and we can see that the map \(x\in \p \Omega \cap B(a,r_a) \mapsto \nu (x)\) is continuous. For all \(\tilde{a} \in \p \Omega \cap B(a,r_a)\), we can then find an affine rotation \(O_{\tilde{a}}\) such that
\begin{itemize}
\item \(O_{\tilde{a}} \nu(\tilde{a})=-e_n\), 
\item \(O_{\tilde{a}} \tilde{a}=\tilde{a}\),
\item \(\tilde{a} \mapsto O_{\tilde{a}} \in M_n(\R)\) is continuous.
\end{itemize}
We then set 
\begin{equation}
\left\{
\begin{array}{rcll}
f_a: & \overline{\Omega}\cap B(a,r_a) & \rightarrow & \R^n_+ \\
& x & \mapsto & \left( O_ax\cdot e_1, \dots, O_ax \cdot e_{n-1}, \frac{\rho_a(x)}{|\nabla \rho_a(a) |} \right)-\left(a_1,\dots,a_{n-1},0 \right).
\end{array}
\right.
\end{equation}
By definition \(f_a \in \C^1(\overline{\Omega}\cap B(a,r_a),\R^n)\) and items \ref{item_1}, \ref{item_3}, \ref{item_4} are satisfied. Now, for all \(x\in \overline{\Omega}\cap B(a,r_a)\) and all \(h\in \R^n\) we have
\begin{align*}
\dd f_a(x).h & = \left( O_a h\cdot e_1, \dots, O_ah \cdot e_{n-1}, \frac{\nabla \rho_a (x)\cdot h}{| \nabla \rho_a (a)|} \right).
\end{align*}
In particular, for \(x=a\), we find that 
\begin{align*}
\dd f_a(a) &= \left(O_a h\cdot e_1, \dots, O_ah \cdot e_{n-1}, O_a^T e_n \cdot h \right)= O_a h.
\end{align*}
Hence item \ref{item_2} is also satisfied and, by the local inversion theorem, up to reducing \(r_a>0\) we also have that item \ref{item_5} is satisfied. 
To prove item \ref{item_6}, given \(a'\in B(a,r_a)\) we set \(r_{a'}=r_a-|a-a'|>0\) and we define
\begin{equation}
\left\{
\begin{array}{rcll}
f_{a'}: & \overline{\Omega}\cap B(a',r_{a'}) & \rightarrow & \R^n_+ \\
& x & \mapsto & \left( O_{a'}x\cdot e_1, \dots, O_{a'}x \cdot e_{n-1}, \frac{\rho_a(x)}{|\nabla \rho_a(a') |} \right)-\left(a'_1,\dots,a'_{n-1},0) \right).
\end{array}
\right.
\end{equation}
By using the continuity of \( \tilde{a} \in \p \overline{\Omega} \cap B(a,r_a) \mapsto O_{\tilde{a}}\), we can see that 
\begin{equation}\label{eq:continuity_straightening}
\| f_a -f_{a'}\|_{\C^1(B(a',r_{a'}))} \xrightarrow[a' \to a]{} 0.
\end{equation}
Hence, by the openness of the set of \(\C^1\) diffeomorphism \cite[Theorem 1.6]{Hirsch_1976}, for \(a'\) close enough to \(a\), we have that \(f_a\) is also a \(\C^1\)-diffeomorphism from \(\overline{\Omega}\cap B(a',r_{a'})\) onto its image (which is contained in \(\R^n_+\)). Now item \ref{item_6} is verified is a consequence of \eqref{eq:continuity_straightening}.
\end{proof}

Now consider \(\tilde{\Phi}_a:=\Phi\circ f_a^{-1}: \overline{U}_a\rightarrow \overline{\mathbb{B}^n}\), where \(\overline{U}_a\) is defined in Claim \ref{claim:straightening}. Locally near \( \Phi_a\) we can find an orthonormal basis \( (u_1^a,\dots,u_{n-1}^a)\) of \(T_{\Phi_a(a)} \mathbb{S}^{n-1}\) such that \( a \mapsto (u_1^a, \dots, u_{n-1}^a)\) is continuous. In the basis \(\mathcal{B}:=(e_1,\dots,e_{n-1},e_n)\) and \(\mathcal{B}_a':=(u_1^a,\dots,u_{n-1}^a,\Phi(a))\), the matrix of \(\dd \tilde{\Phi}_a(0)\) has the form
\begin{equation}
\dd \tilde{\Phi}_a(0)=\begin{pmatrix}
A_a & B_a \\
0 &\beta_a
\end{pmatrix}
\end{equation}
with \(A_a\in GL_{n-1}(\R)\), \(B_a\in M_{n-1,1}(\R)\) and \(\beta_a \in \R\). We can use the polar decomposition to find \(R_a\in O_{n-1}(\R)\) and \(S_a\in S_{n-1}^{++}(\R)\) such that \( A_a=R_a S_a\). Furthermore, the dependence on \(a\in \p \Omega\) of all the objects \(R_a, S_a, \beta_a\) is continuous, by continuity of the polar decomposition. We can also see that, if \(\Phi\) is close to \(\id\), then \(A_a\) is close to a rotation, meaning that \(S_a\) is close to \(\id\), and \(\beta_a\) is close to \(1\) and \(B_a\) is close to \(0\). Hence we can assume \(\beta_a>0\).

\begin{claim}\label{claim:diffeo_local}(local diffeomorphism making \(\Phi\) conformal in \(a\) in the flat case)
There exists a diffeomorphism \(\psi_a: \overline{U}_a \rightarrow \overline{U}_a\) such that 
\begin{enumerate}
\item \label{item:diffeo2}\(\psi_a\) is a \(\C^1\)-diffeomorphism and \(\psi_a(\{x_n=0\}\cap \overline{U}_a)\subset \{x_n=0\}\),
\item \label{item:02}\(\psi_a(0)=0\),
\item \label{item:id2} \(\psi_a=\id\) in a neighbourhood of \(\p \overline{U}_a \setminus \R^n_+\),
\item \label{item:diff02}\begin{align*}
\dd \psi_a(0)&=\beta_a \dd \tilde{\Phi}_a(0)^{-1}\begin{pmatrix}
R_a &0 \\
0 &1
\end{pmatrix} =\beta_a\begin{pmatrix}
A_a & B_a \\ 
0 & 1
\end{pmatrix}^{-1}\begin{pmatrix}
R_a & 0 \\
0 &1\end{pmatrix}  =\beta_a\begin{pmatrix}
A^{-1}_aR_a & -A_a^{-1} B_a \\
0 & 1
\end{pmatrix}.
\end{align*}
\item \label{item:continuity_2} \( \| \psi_a-\psi_{a'}\|_{\C^1(U_a \cap U_{a'})}\) is small if \( |a-a'|\) is small enough and \( \| \psi_a -\id\|_{\C^1(U_a)}\) is small if \(\Phi\) is close to \(\id\).
\end{enumerate}
\end{claim}

\textit{Proof of Proposition \ref{continuityscheme1_v2} in the case \(a\in \p \Omega\):}  

Once \(\psi_a\) is obtained we can check that \(\Phi_a:= \Phi \circ (f_a^{-1}\circ \psi_a \circ f_a)\) satisfies all the desired properties in Proposition \ref{continuityscheme1_v2}. Indeed,
\begin{itemize}
\item  it holds \(\Phi_a(a)=\Phi( f_a^{-1}(\psi_a(0)))=\Phi(f_a^{-1}(0))=\Phi(a)\),
\item We have 
\begin{align*}
\dd \Phi_a(a)&= \dd \tilde{\Phi}_a(0) \circ \dd \psi_a(0)\circ \dd f_a(a) \\
&=\dd \tilde{\Phi}_a(0) \circ \dd \tilde{\Phi}_a(0)^{-1}\circ  \begin{pmatrix}
\beta_aR_a & 0 \\
0 & \beta_a
\end{pmatrix}\circ O_a \\
&=\beta_a\begin{pmatrix}
R_a & 0 \\
0 & 1
\end{pmatrix} O_a =\beta_a \tilde{R}_a.
\end{align*}
with \(\tilde{R}_a\in O_n(\R)\).

\item Furthermore, we will check from the construction that item \ref{item:last} in Proposition \ref{continuityscheme1_v2} holds and that \(\|\Phi_a-\id\|_{\C^1}\) is small if \(\|\Phi-\id\|_{\C^1(\Omega)}\) is small enough.
\end{itemize}
In order to prove Claim \ref{claim:diffeo_local} we construct a suitable vector field which generates a one-parameter family of diffeomorphism \((\psi_t^a)_t\) from \(\overline{U}_a\) to \(\overline{U}_a\) for \(t\in \R\) and we choose \(\psi_a:=\psi_1^a\). The desired vector-field is constructed in the following claim.
\begin{claim}\label{claim:vector-field}(Construction of a vector field generating the local diffeomorphism \(\psi_a\))

There exists a \(\C^1\) vector-field \(X_a\colon \overline{U}_a \rightarrow \R^n\) such that
\begin{enumerate}
\item \label{item:O} \(X_a(0)=0\), 
\item \label{item:perp}\(X_a(x)\cdot e_n=0\) for all \(x\in \overline{U}_a \cap \{x\in \R^n; x_n=0\}\),
\item \label{item:C1} \(X_a \in \C^1_c(U_a,\R^n)\),
\item \label{item:trueverif}\(\exp (\dd X_a(0))=\beta_a\dd \tilde{\Phi}_a(0)^{-1}\begin{pmatrix}
R_a & 0\\
0 & 1
\end{pmatrix}.\)
\item \label{item:continuity_3} \(\|X_a-X_{a'}\|_{\C^1 (U_a\cap U_{a'})}\) is small if \( |a-a'|\) is small enough and \( \|X_a\|_{\C^1 (U_a,\R^n)}\) is small if \(\Phi\) is close to \(\id\).
\end{enumerate}
\end{claim}

\begin{proof}[Proof of Claim \ref{claim:vector-field}]
In order to obtain \(X_a\) we first construct a matrix \(N_a=\begin{pmatrix}
N_a^1 & N_a^2 \\
0 & 0
\end{pmatrix}\) which satisfies \(\exp (N_a)=\begin{pmatrix}
 A_a^{-1}R_a & -A_a^{-1}B_a \\
 0 & 1
\end{pmatrix}\). Since\footnote{This can be shown, by using that, for \(N\in M_n(\R)\), the function \(t\mapsto \exp (t N)\) is the unique solution to \(\frac{\dd }{\dd t}\exp(tN)= N \exp(tN)\) taking the value $I_n$ at $t=0$.}
\begin{equation}
\exp \begin{pmatrix}
N_a^1 & N_a^2 \\
0 & 0
\end{pmatrix}=\begin{pmatrix}
\exp (N_a^1) & \exp(N_a^1)\int_0^1 \exp (-s N_a^1) \dd s N_a^2 \\
0 & 1
\end{pmatrix},
\end{equation}
we can take \(N_a^1\) and \(N_a^2\) such that 
\begin{align}
\exp (N_a^1)&=A_a^{-1}R_a =S_a^{-1},  \label{cond1} \\
 N_a^2& = -\left( \int_0^1 \exp (-s N_a^1)\dd s \right)^{-1} \exp (-N_a^1)A_a^{-1}B_a. \label{cond2}
\end{align}
 
It is possible to find \(N_a^1\) such that \eqref{cond1} is satisfied because \(\exp\) is a local diffeomorphism from a neighbourhood of \(0\) to a neighbourhood of \(\id\) in \(GL_{n-1}(\R)\) and, if \(\Phi\) is close enough to \(\id\) then \(A_a^{-1}R_a=S_a^{-1}\) is close to the identity.  Hence \(N_a^1\) exists and is close to \(0\). This also shows that \(\int_0^1 \exp (-s N_a^1)\dd s\) is close to the identity and thus invertible, proving that it is possible to find \(N_a^2\) to satisfy \eqref{cond2}. Furthermore, \(N_a^2\) can also be make close to the zero matrix if \(\|\Phi-\id\|_{\C^1(\Omega)}\) is small enough since in that case \(B_a\) is close to zero.
Now we set
\begin{equation}
X_a(x) =\log (\beta_a) \, \chi_a(x) \, N_ax, \text{ for all } x\in \overline{U}_a,
\end{equation}
where \(\chi_a\in \C^\infty_c(\overline{U}_a)\) with \(\chi_a\equiv 1\) in a neighbourhood of \(0\). With this definition one can check that \(X_a\) has all the required properties in Claim \ref{claim:vector-field}. Indeed, Item \ref{item:O} is satisfied because \(N_a0=0\). Item \ref{item:perp} holds since \(N_ax\in \text{Vect}(e_1,\dots,e_{n-1})\) thanks to the shape of \(N_a\). Item \ref{item:C1} is guaranteed by the choice of \(\chi_a\).
Now 
\begin{align*}
\exp (\dd X_a(0))&= \exp ( \log (\beta_a)N_a)=\beta_a \exp (N_a) \\
& =\beta_a \begin{pmatrix}
A_a^{-1}R_a & -A_a^{-1}B_a \\
0 & 1 
\end{pmatrix} \\
& =\beta_a \dd \tilde{\Phi}_a (0)^{-1}\begin{pmatrix}
R_a & 0 \\
0 & 1
\end{pmatrix},
\end{align*}
so that Item \ref{item:trueverif} is true. The last point \ref{item:continuity_3} follows first from the continuity of \(a\in \p \Omega \mapsto N_a\), which itself follows from the continuity in \(a\in \p \Omega\) of \(A_a, R_a\) and \(B_a\), and then from the fact that if \(\Phi\) is close to the identity then \(N_a\) is close to zero.
\end{proof}
With the help of the previous vector-field we obtain the diffeomorphism in Claim \ref{claim:diffeo_local}.

\begin{proof}[Proof of Claim \ref{claim:diffeo_local}]

Let us consider the flow generated by \(X_a\), i.e., the solution of the ODE system
\begin{equation}\label{eq:Caucy}
\left\{
\begin{array}{rcll}
\frac{d}{dt} \psi_t^a (x) & =& X_a (\psi_t^a(x)), \\
\psi_0^a(x)& =& x,
\end{array}
\right.
\end{equation}
for all \(x\in \overline{U}_a\).

By using the Cauchy--Lipschitz theory, we see that \(\psi_t^a\) is defined for all \(t\in \R\) and for all \(t\in \R\) we have that \(\psi_t^a\) is a \(\C^1\)-diffeomorphism of \(\overline{U}_a\) with \(\psi_t^a(\{x_n=0\}\cap \overline{U}_a)\subset \{x_n=0\}\). Besides, for all \(t\in \R\), since \(X_a(0)=0\), the uniqueness in the Cauchy--Lipschitz theorem implies that \(\psi_t^a(0)=0\). Hence \(\psi_t^a\) satisfies Item \ref{item:diffeo2} and Item \ref{item:02} in Claim \ref{claim:diffeo_local}. By differentiating \eqref{eq:Caucy} and evaluating at \(x=0\) we see that 
\begin{equation}\label{eq:Cauchy_2}
\left\{
\begin{array}{rcll}
\frac{d}{dt} \dd \psi_t^a (0)& = &\dd X_a(0)\dd \psi_t^a(0), \\
\dd \psi_0(0)&=&\id.
\end{array}
\right.
\end{equation}
Hence we see that \(\dd \psi_t^a (0)=\exp (t \dd X_a(0))\). From Claim \ref{claim:vector-field}, if we choose \(\psi_a\coloneqq \psi_1^a\) we obtain that 
\(\dd \psi_a(0)=\exp (\dd X_a(0))=\beta_a \dd \tilde{\Phi}_a(0)^{-1} \begin{pmatrix}
R_a & 0 \\
0 & 1
\end{pmatrix}\). Thus Item \ref{item:diff02} is satisfied. Since \(X\in \C^1_c(U_a,\R^n)\) we obtain that Item \ref{item:id2} holds. 
The last point, Item \ref{item:continuity_2}, is a consequence of Item \ref{item:continuity_3} in Claim \ref{claim:vector-field}.
\end{proof}

Thus for every \(a\in \p \Omega\), for all \(\e>0\) we can construct \(\Phi_{a}\) satisfying the properties of Proposition \ref{continuityscheme1_v2}.

\underline{Step 2:} We construct \(\Phi_{a}\) when \(a\in \Omega\).

We first observe that for \(\eta>0\) small enough, the set \(\Omega_\eta\coloneqq \{x\in \overline{\Omega}; \dist(x,\p \Omega)<\eta\}\) defines a tubular neighbourhood of \(\p\Omega\) where there exists a continuous projection \(\Pi\colon  \Omega_\eta \rightarrow \p \Omega\). 
For all \(a\in \Omega_\eta\) we set 
\begin{equation}
\Phi_{a}=\Phi_{\Pi(a)}
\end{equation}
where \(\Phi_{\Pi(a)}\) is defined in the previous step since \(\Pi(a)\in \p \Omega\). 
Now, let \(\lambda:\R^+\rightarrow \R\) be a function such that \(\lambda\in \C^\infty(\R^+)\), such that \(\lambda(t)=1\) if \(t\leq \eta/2\) and \(\lambda(t)=0\) if \(t\geq \eta\). We set
\begin{equation}
\Phi_{a}:= \lambda (\dist(a,\p \Omega)) \Phi_{\Pi(a)}+(1-\lambda(\dist (a,\p \Omega)))\Phi
\end{equation}
for all \(a\in \Omega\). We observe that, if \(\Phi\) is close enough to the identity then \(\Phi_{a}\) is also close to the identity and hence is a \(\C^1\)-diffeomorphism by \cite[Theorem 1.6]{Hirsch_1976}. The map \(a\mapsto \Phi_{a}\) is continuous by construction and all the required properties for \(a\in \p \Omega\) have been verified in the previous steps.
\end{proof}

\subsection{Definition of ``almost M\" obius maps'' and energy estimates}

\begin{definition}\label{def:almost_Mobius_map}
For any $a\in \overline{\Omega}$ and $r\in(0,1)$, we define $\widetilde{M}_{r,a} \colon \overline{\Om} \to \overline{\bB}^n$, for $x\in \overline{\Omega}$ by
\begin{equation*}
	\widetilde{M}_{r,a}(x) \coloneqq M_{(1-r)\Phi(a)}(\Phi_a(x)),
\end{equation*}
where \(\Phi_a\) is defined in Proposition \ref{continuityscheme1_v2}.
\end{definition}

\begin{lemma}\label{lm:energy_almostMobius}
	For every \(r>0\), for every \(\alpha\geq 0\) the map \( a\in \overline{\Omega} \mapsto \widetilde{M}_{r,a} \in \I_{1,\alpha}\) is continuous.
	
	The $(n+\alpha)$-energy of $\tilde{M}_{r,a}$ can be estimated uniformly in $a\in\pl \Omega$
	\begin{equation}\label{eq:energy_boundary}
		 \lim_{r \to 0 } \, \lim_{\alpha \to 0} \int_{\Omega} |\dd \widetilde{M}_{r,a}|^{n+\alpha} = n^\frac{n}{2} \, |\bB^n|.
	\end{equation}
	Moreover, for all \(\e>0\) there exists \(L>0\) such that if \(\Phi:\overline{\Omega} \rightarrow \overline{\mathbb{B}^n}\) satisfies \(\| \Phi-\id \|_{\C^1(\Omega)}<L\) then
	\begin{equation}\label{eq:uniform_Mobius}
		\limsup_{\alpha\to 0}\ \sup_{a\in\Omega} \int_{\Omega} |\dd \widetilde{M}_{r,a}|^{n+\alpha} \leq   n^\frac{n}{2}\, |\bB^n| +\e  .
	\end{equation}
\end{lemma}

\begin{proof}
Let \(r>0\) and \(\alpha\geq 0\), to show the first point we show that, for any \(a\in \overline{\Omega}\),
\begin{equation}\label{eq:continuity_en_a}
\lim_{ a' \to a} \int_{\Omega} | \dd \widetilde{M}_{r,a}- \dd \widetilde{M}_{r,a'}|^{n+\alpha} =0.
\end{equation}
The change of variable \(y=\Phi_a(x)\) shows that 
\begin{align*}
&\int_{\Omega} | \dd \widetilde{M}_{r,a}- \dd \widetilde{M}_{r,a'}|^{n+\alpha} \\
&= \int_{\mathbb{B}^n} \Bigl|\dd M_{(1-r)\Phi(a)}(y) \dd \Phi_a\left(\Phi_a^{-1}(y))\right)\\
& \quad \quad -\dd M_{(1-r)\Phi(a')}\left(\Phi_{a'}(\Phi_a^{-1}(y))\right) \dd \Phi_{a'}\left(\Phi_a^{-1}(y) \right) \Bigr|^{n+\alpha} |\det \dd \Phi_a^{-1}(y)| \dd y .
\end{align*}
By using the dominated convergence theorem along with the fact that
\begin{equation}\label{eq:consequence_constr}
\Phi_a\circ \Phi_{a'}^{-1} \xrightarrow[a.e.]{a' \to a} \id, \quad \dd \Phi_{a'}\left(\Phi_a^{-1}(y) \right) \xrightarrow[a.e.]{a' \to a} \dd \Phi_a(\Phi_a^{-1}(y)),
\end{equation}
and the explicit expression of the M\"obius map we arrive at \eqref{eq:continuity_en_a}. Note that \eqref{eq:consequence_constr} is a consequence of item \ref{item:last} in Proposition \ref{continuityscheme1_v2}.

	For the second point, by dominated convergence theorem, we have that \(\lim_{\alpha \to 0} \int_{\O}|\dd  v|^{n+\alpha}=\int_{\O}|\dd v|^n\) for any fixed map \(v\in \mathcal{I}\cap \C^\infty(\Omega,\R^n)\). 
	We start with \eqref{eq:energy_boundary}. We consider the following decomposition
	\begin{align*}
		E_n(\widetilde{M}_{r,a}) &= E_n \left( M_{(1-r)\Phi(a)}\circ \Phi_a, \Phi_a^{-1}\big(B(\Phi(a),r^\frac{1}{2}) \big) \right)\\[2mm]
		& \qquad +E_n\left( M_{(1-r)\Phi(a)}\circ \Phi_a ,\Phi_a^{-1}\big( \mathbb{B}^n\setminus  B(a,r^{\frac{1}{2}}) \big) \right).
	\end{align*}
	We use a change of variable \(y=\Phi_a (x)\) to estimate that
	\begin{align*}
		& E_n\left( M_{(1-r)\Phi(a)}\circ \Phi_a, \Phi_a^{-1}\big( B(\Phi(a),r^\frac{1}{2})\big) \right) \\[2mm]
		&= \int_{\mathbb{B}^n \cap B\left(\Phi(a),r^{1/2}\right)} \left| \dd M_{(1-r)\Phi(a)}(y)\circ \dd \Phi_a(\Phi_a^{-1}(y)) \right|^n  \  \frac{1}{\det \big( \dd\Phi_a(\Phi_a^{-1}(y)) \big)} \dd y \\[2mm]
		& =\int_{\mathbb{B}^n \cap B\left( \Phi(a),r^{1/2} \right)} \left| \dd M_{(1-r)\Phi(a)}(y)\left(\dd\Phi_a(a)+\cO(r^{\frac12})\right) \right|^n \ \frac{1}{\det \big( \dd\Phi_a(a)+\cO(r^{\frac12}) \big) } \dd y \\[2mm]
		&=\int_{\mathbb{B}^n \cap B\left( \Phi(a),r^{1/2} \right)} \left| \dd M_{(1-r)\Phi(a)}(y)\circ \dd\Phi_a(a)) \right|^n \ \frac{1}{\det \dd\Phi_a(a)} \dd y +\cO(r^{\frac12}) \\[2mm]
		& =\int_{\mathbb{B}^n \cap B\left(\Phi(a),r^{1/2} \right)} \left| \dd M_{(1-r)\Phi(a)}(y) \right|^n \dd y +\cO(r^{\frac12}) \\[2mm]
		&= n^{\frac{n}{2}}\, |\mathbb{B}^n|+\cO(r^{\frac12}),
	\end{align*}
	where we have used that \(\dd \Phi_a(a)=\beta_a O_a\) with \(O_a \in O_n(\R)\) and where we have also used \eqref{eq:DL_Mobius}. We proceed to the same change of variable on the complement, using that \(| \dd \Phi_a(\Phi_a^{-1}(y))|^n/ |\det \dd \Phi_a(\Phi_a^{-1}(y))|$ is uniformly bounded by \eqref{eq:distortiond}, we find
	\begin{equation*}
		E_n\left( M_{(1-r)\Phi(a)}\circ \Phi_a ,\Phi_a^{-1}\big( \mathbb{B}^n\setminus  B(a,r^{\frac{1}{2}}) \big) \right)  \leq C\, \int_{\mathbb{B}^n \setminus  B\left(\Phi(a),r^{1/2} \right)} \left| \dd M_{(1-r)\Phi(a)}(y) \right|^n \dd y .
	\end{equation*}
	Using again \eqref{eq:DL_Mobius}, we obtain that 
	\begin{equation*}
		\lim_{r\to 0}E_n\left( M_{(1-r)\Phi(a)}\circ \Phi_a ,\Phi_a^{-1}\big( \mathbb{B}^n\setminus  B(a,r^{\frac{1}{2}}) \big) \right) =0,
	\end{equation*}
	which completes the proof of \eqref{eq:energy_boundary}.
	
	We now prove \eqref{eq:uniform_Mobius}. For $r>0$ fixed, and any $\alpha\in[0,1]$ and $a\in \overline{\Omega}$, the same change of variable previously used shows that 
	\begin{align*}
		\int_{\O} \left| \dd\widetilde{M}_{r,a} \right|^{n+\alpha} &  = \int_{\bB^n } \left| \dd (M_{(1-r)\Phi(a)})(y)\circ \dd (\Phi_a)(\Phi_a^{-1}(y)) \right|^{n+\alpha} \frac{\dy}{\det \dd(\Phi_a)(\Phi_a^{-1}(y))}.
		\end{align*} 	
\begin{claim}
We have
\begin{multline}\label{eq:sing_value_3}
		 \frac{\left| \dd (M_{(1-r)\Phi(a)})(y)\circ \dd (\Phi_a)(\Phi_a^{-1}(y)) \right|^{n+\alpha}}{|\det \dd(\Phi_a)(\Phi_a^{-1}(y))|} \\
		 \leq  \frac{\max_{|z|=1} \left||\dd \Phi_a \left(\Phi_a^{-1}(y)\right).z|^{n+\alpha} \right|}{\min_{|z|=1} \left|\dd \Phi_a\left(\Phi_a^{-1}(y).z|^{n+\alpha} \right).z \right|^n } \left|\dd (M_{(1-r)\Phi(a)})(y) \right|^{n+\alpha}
	\end{multline}	
\end{claim}

\begin{proof}
It suffices to show that for two matrices \(A\in M_n(\R)\) and \(B\in GL_n(\R)\) we have 
\begin{equation}\label{eq:sing_value_1}
 \frac{| AB|^{n+\alpha}}{\det B} \leq \frac{\max_{|z|=1} |Bz|^{n+\alpha}}{\min_{|z|=1} |Bz|^n} |A|^{n+\alpha}.
\end{equation} 
  To this end, we observe that, by definition,
\begin{align*}
|AB|^2 &=\tr \left( (AB)^T AB \right) = \tr \left( B^TA^T A B\right) =\tr \left( B B^T A^T A\right).
\end{align*}
Now, \(BB^T\) is a positive definite symmetric matrix, from the spectral theorem we can find \(P\in O_n(\R)\) and \(D=\rm{diag} (\mu_1^2,\dots,\mu_n^2)\) a diagonal matrix with positive entries such that \(0<\mu_1\leq \dots \leq \mu_n\) and \( BB^T=PDP^T\). Note that the \(\mu_i\) are the singular values of \(B\). Now
\begin{align*}
|AB|^2 &=\tr \left( PDP^T A^TA \right) =\tr \left(D (AP)^T (AP) \right) \leq \mu_n^2 \tr \left( (AP)^T (AP)\right) =\mu_n^2 \tr (A^T A)
\end{align*}
where we used that \(P^TP=\id\).  In the same way \(|\det B|=\sqrt{\det (BB^T)}=\sqrt{\det D}\geq \mu_1^n.\) 
Hence 
\begin{equation}\label{eq:sing_value_2}
\frac{| AB|^{n+\alpha}}{\det B} \leq \frac{\mu_n^{n+\alpha} |A|^{n+\alpha}}{\mu_1^n}.
\end{equation}
But we also have that \(\mu_n =\max_{|z|=1} |Bz|\) and \(\mu_1=\min_{|z|=1} |Bz|\). Indeed, by using a polar decomposition \(B=OS\) with \(O\in O_n(\R)\) and \(S\in \S_n^{++}(\R)\) we see that \( |Bz|=|Sz|\) and \( BB^T=O S  S^T O^T\). Hence, the singular values of \(B\) are the same as the ones of \(S\). But since \(S\) is a real definite positive symmetric matrix, its singular values are just the squares of its eigenvalues. Thus
\(\max_{|z|=1} |Bz|= \mu_n\) and \(\min_{|z|=1} |Bz|=\mu_1\).
Now we can see that \eqref{eq:sing_value_1} follows from the previous equalities together with \eqref{eq:sing_value_2}.
\end{proof}

Now, coming back to \eqref{eq:sing_value_3}, we see that 
\begin{align}\label{eq:domination_1}
\int_{\O} \left| \dd\widetilde{M}_{r,a} \right|^{n+\alpha} &\leq \int_{\mathbb{B}^n}\frac{\max_{|z|=1} \left||\dd \Phi_a \left(\Phi_a^{-1}(y)\right).z|^{n+\alpha} \right|}{\min_{|z|=1} \left|\dd \Phi_a\left(\Phi_a^{-1}(y).z|^{n+\alpha} \right).z \right|^n } \left|\dd (M_{(1-r)\Phi(a)})(y) \right|^{n+\alpha}.
\end{align}
If \(\Phi\) is close enough to \(\id\) then \(\Phi_a\) is close enough to the identity from Proposition \ref{continuityscheme1_v2} and thus
\begin{equation}\label{eq:petitesse}
\frac{\max_{|z|=1} \left||\dd \Phi_a \left(\Phi_a^{-1}(y)\right).z|^{n+\alpha} \right|}{\min_{|z|=1} \left|\dd \Phi_a\left(\Phi_a^{-1}(y).z|^{n+\alpha} \right).z \right|^n } \leq 1 +\frac{\e}{n^{\frac{n}{2}}|\mathbb{B}^n|}. 
\end{equation}

For each \(\alpha>0\) there exists \(a_\alpha \in \overline{\Omega}\) such that 

\[ \sup_{a\in \Omega} \int_{\mathbb{B}^n} | \dd M_{(1-r)\Phi(a)}|^{n+\alpha} =\int_{\mathbb{B}^n} | \dd M_{(1-r)\Phi(a_{\alpha})}|^{n+\alpha}.\]
Up to extraction, we can assume that \(a_\alpha \xrightarrow[\alpha \to 0]{} a_0\in \overline{\Omega}\). Since \(r>0\), the dominated convergence theorem shows that 
\begin{equation}\label{eq:convergence}
 \int_{\mathbb{B}^n} | \dd M_{(1-r)\Phi(a_{\alpha})}|^{n+\alpha} \xrightarrow[\alpha \to 0]{} \int_{\mathbb{B}^n}| \dd M_{(1-r)\Phi(a_0)}|^n=n^{\frac{n}{2}}|\mathbb{B}^n|.
 \end{equation}
The proof of \eqref{eq:uniform_Mobius} is concluded by gathering \eqref{eq:domination_1}, \eqref{eq:petitesse} and \eqref{eq:convergence}.
\end{proof}

\section{Mountain pass approach for a perturbed problem} \label{sec:MountainPass}

The goal of this section is to prove existence of critical points for the \((n+\alpha)\)-energy in the space of maps with values into \(\mathbb{S}^{n-1}\) on \(\p \Omega\) and with degree one on \(\p \Omega\). In order to produce critical points for the \((n+\alpha)\)-energy we need to verify that the space on which it is defined is a Banach manifold and that this functional is \(\C^1\). We also need to check that it satisfies the Palais--Smale condition. This will be the content of the the first subsection of this section. In the second subsection we prove that we can devise a min-max scheme for which the \((n+\alpha)\)-energy has the mountain pass geometry. We can then apply the mountain-pass theorem of Ambrosetti--Rabinowitz to conclude. More precisely we want to apply the following theorem, see e.g. \cite[Theorem 4.3 and Corollary 4.3]{Mawhin-Willem}:
\begin{theorem}\label{th:Mountain Pass}
Let $K_0 \subset K$ two compact metric spaces. Let $X$ be a Banach space or a $\mathcal{C}^{1,1}$ Finsler manifold\footnote{See \cite[Section 3.7]{Struwe_2008}} and $J \in \mathcal{C}^1(X,\R)$. Let $\chi \in \mathcal{C}^0(K_0,X)$ be a fixed map. We define
\begin{equation}
c\coloneqq\inf\{\underset{K}{\max}\ J \circ F :  F \in \mathcal{P} \}, \ \text{ where } \mathcal{P}\coloneqq\{F \in \mathcal{C}^0(K,X): F= \chi \ \text{on } K_0 \}.
\end{equation}
Assume that
\begin{equation}\label{Mountain pass geometry}
c>c_1\coloneqq\underset{a \in K_0}{\max}\ J \left( \chi(a) \right).
\end{equation}
 Then there exists a sequence $(x_k) \subset X$ such that
\begin{equation}\label{Palais-Smale}
J(x_k) \rightarrow c \ \text{ and } \dd J(x_k) \rightarrow 0 \ \text{ as } k \rightarrow +\infty.
\end{equation}
\end{theorem}

If $J$ satisfies \eqref{Mountain pass geometry} for some $K_0,K$ we say that $J$ has a Mountain Pass geometry.   Under this condition, the theorem produces a Palais--Smale sequence, \textit{i.e.}, a sequence such that \eqref{Palais-Smale} holds. Assume furthermore that the functional $J$ satisfies the following Palais--Smale condition.
\begin{equation}\tag{PS}\label{PScondition}
\text{ each sequence } (x_k) \text{ satisfying }  \eqref{Palais-Smale} \text{ contains a convergent subsequence }.
\end{equation}
Then the sequence $(x_k)$ in \eqref{Palais-Smale} converges to $x$ which is a critical point of $J$, such that $J(x)=c$.

We want to apply Theorem \ref{th:Mountain Pass} with
\begin{align}
 J=E_{n+\alpha}, \quad X=\mathcal{I}_{1,\alpha}, \quad K_0=\p \Omega, \quad K=\overline{\Omega} \text{ and } \label{eq:choice_1}\\
 \chi(a)= M_{(1-r)\Phi(a)}\circ \Phi_a \label{eq:choice_2}
 \end{align}
for any \(a \) in \( \p \Omega\), and for \(r>0\) small enough where \(M_{(1-r)\Phi(a)}\circ \Phi_a\) appears in Definition \ref{def:almost_Mobius_map}.

\subsection{Differentiability structure and Palais-Smale condition for the perturbed problem}

We first explain why \(\I_{1,\alpha}\) is a smooth Finsler manifold and \(E_{n+\alpha}\) is \(\C^1\) on that space. First we can see that, for all $\alpha >0$, the space $\I_\a$ is a smooth reflexive Banach manifold and its tangent space $T_u \I_\a$ at a given point $u \in \I_\a$ is given by
\begin{equation}
T_u \I_\a =\{\varphi \in W^{1,n+\a}(\O,\R^n) : \varphi(x)\in T_{u(x)} \S^{n-1}, \ \forall
 x \in \p \O \} .
 \end{equation}
The proof of this fact is an adaptation of the proof of \cite[Proposition II.1]{Riviere_2017} or \cite[Theorem 1.4.1]{Moore_2017}, the crucial point being the embedding \( W^{1,n+\alpha} \hookrightarrow \C^{0,\beta}\) for some \(0<\beta<1\). The Finsler structure we consider is given by the restriction of the \(W^{1,n+\alpha}\) norm to each fiber of \( T \I_\alpha\). We can check that it is indeed a Finsler structure
Now we can use that \(\I_{1,\alpha}\) is open in \(\I_{\alpha}\). Indeed, if \(u\) is in \(\I_{1,\alpha}\) and if \(v\in \I_{\alpha}\) is such that \(\|u-v\|_{W^{1,n+\alpha}}\) is small enough, then \(\|u-v\|_{C^0}\) is also small and thus \(v\) is homotopic to \(u\) on \( \p \O\). As an open subset of a Banach manifold \(\I_{1,\alpha}\) is also a Banach manifold and it naturally inherits the Finsler structure originally defined on \(\I_{\alpha}\).

Now we can see that $E_{n+\alpha}$ is $\mathcal{C}^1$ in $W^{1,n+\a}(\O,\R^n)$, for $\a \geq 0$ and its differential is defined for any $u,v\in W^{1,n+\a}(\O,\R^n)$ by
\begin{equation}\label{eq:differential}
\dd E_{n+\alpha}(u). v = (n+\alpha)\int_\O |\dd u|^{n+\a-2}\dd u:\dd v.
\end{equation}
Hence $E_{n+\alpha}$ is \(\C^1\) on \(\I_{\alpha}\) and on \(\I_{1,\alpha}\). Furthermore its differential, as a map from \( \I_\alpha\) to \(\R\) is given by \eqref{eq:differential} from every \(u\) in \( \I_\alpha\) and every \(v\) in \(T_u \I_\alpha\).

We now show that \(E_{n+\alpha}\) satisfies the Palais--Smale condition in \(\I_{1,\alpha}\). In order to do that, we follow closely \cite{Struwe_1984}, see also \cite{Fraser_2000}.

We define a projection
\begin{eqnarray}\label{projection}
\pi_\a(\cdot,\cdot) \colon \I_\a\times W^{1,n+\a}(\O,\R^n) &\rightarrow &T\I_\a \nonumber \\
(u,v) &\mapsto &v- \eta_\a\left( [u\langle u,v\rangle ]_{|\p\O} \right),
\end{eqnarray}
\noindent where $\eta_\a \colon W^{1-\frac{1}{n+\a},n+\a}(\O,\R^n) \rightarrow W^{1,n+\a}(\O,\R^n)$
is a continuous linear extension operator. Note that if \(u\in \I_\alpha\) and \(v\in W^{1,n+\alpha}(\Omega,\R^n)\) then  \(\pi_\alpha(u,v)\in T_u\I_{\alpha}\). For \(u\in \I_\alpha\) and \(v\in W^{1,n+\alpha}(\Omega,\R^n)\) we also define
\begin{equation}\label{def:definition_diff}
\dd_{\I_\a} E_{n+\alpha}(u).v= \dd E_{n+\alpha}(u).\pi_\alpha (u,v)
\end{equation}
We have the following lemma 
\begin{lemma}(Lemma 2.1 in \cite{Struwe_1984}) \label{outil}
Let $\a \geq 0$.
\begin{itemize}
\item[i)] If $u,v \in W^{1-\frac{1}{n+\a},n+\a}\cap L^\infty(\p \O,\R^n)$, then $u\cdot v \in W^{1-\frac{1}{n+\a},n+\a} \cap L^\infty(\p \O,\R^n)$ and
\[ \|u\cdot v\|_{W^{1-\frac{1}{n+\a},n+\a}} \leq \|u\|_{\infty}\|v\|_{W^{1-\frac{1}{n+\a},n+\a}} +\|v\|_{\infty}\|u\|_{W^{1-\frac{1}{n+\a},n+\a}}. \]
\item[ii)] If $u \in W^{1-\frac{1}{n+\a},n+\a}(\p \O,\R^n)$ and $\sigma \in \mathcal{C}^1(\R^n)$ then
\[\| \sigma \circ  u \|_{W^{1-\frac{1}{n+\a},n+\a}} \leq \|\sigma \circ u\|_{L^{n+\a}}+\|(d\sigma)\circ u\|_{\infty} \|u\|_{W^{1-\frac{1}{n+\a},n+\a}} . \]
\end{itemize}
\end{lemma}

Thanks to the previous lemma we can show the next result which is the analogue of \cite[Lemma 3.1]{Struwe_1984}
\begin{lemma}\label{projection2}
Let $\alpha >0$ and let $(u_m)_m$ be a sequence in $\I_\a$ such that $u_m \rightharpoonup u$ in $W^{1,n+\a}(\Omega,\R^n)$ and $u_m \rightarrow u$ uniformly in $\O$ as $m\rightarrow +\infty$. Then,
\[ \| (u_m-u)-\pi_\a(u_m,u_m-u)\|_{W^{1,n+\a}} \xrightarrow[m\to +\infty]{}0.\]
\end{lemma}

\begin{proof}
Let $(u_m)_m$ be as in the statement of the lemma. Since $u_m$ converges weakly to $u$ in $W^{1,n+\a}(\O,\R^n)$, we can see, from the continuity of the trace operator, that  \(|u|=1\) a.e.\ and hence $u \in \I_\a$. We now apply item ii) of Lemma \ref{outil} with \(\sigma=\sigma_{u_m} =P_\perp^{\mathbb{S}^{n-1}}(u_m)\) the projection onto the normal space to \(\mathbb{S}^{n-1}\) at \(u_m\) when \(u_m\) belongs to \(\mathbb{S}^{n-1}\), which is the case for \(x\in \p \Omega\). More precisely, for \(y\in \R^n\) and \(x\in \p \Omega\) we have
\begin{equation}
P_\perp^{\mathbb{S}^{n-1}}(u_m(x)) (y)=u_m(x) (u_m(x)\cdot y).
\end{equation}
We note that \(\sigma\) is linear so its differential is itself, besides, \(\sigma_{u_m}\) being an orthogonal projection, we have, \(|\sigma_{u_m}(y)| \leq |y|\). Then, we find

\begin{align}
\| (u_m-u)-\pi_\a(u_m,u_m-u)\|_{W^{1,n+\a}}^{n+\a} &= \| \eta_\a\left( [u_m ( u_m\cdot u_m-u)]_{|\p \O} \right)\|_{W^{1,n+\a}}^{n+\a} \nonumber \\[2mm]
 & \leq  C \|u_m ( u_m\cdot u_m-u) \|_{W^{1-\frac{1}{n+\a},n+\a}}^{n+\a} \nonumber \\[2mm]
 &= C\|\sigma_{u_m} (u_m-u)\|_{W^{1-\frac{1}{n+\a},n+\a}}^{n+\a} \nonumber \\[2mm] 
 & \leq C \Bigl(\| \sigma_{u_m}(u_m-u)\|_{L^{n+\a}(\p \Omega)}^{n+\a} \nonumber \\[2mm]
 & \quad +\| \sigma_{u_m}(u_m-u)\|_{L^{\infty}}^{n+\alpha}\| u_m-u\|_{W^{1-\frac{1}{n+\a}}}^{n+\alpha}\Bigr) \nonumber \\[2mm] 
 &\leq C\|u_m-u\|_{L^\infty}^{n+\alpha} \left(|\p \Omega|+\|u\|_{W^{1-\frac{1}{n+\alpha}}}^{n+\alpha} +\|u\|_{W^{1-\frac{1}{n+\alpha}}}^{n+\alpha}\right) \nonumber \\[2mm]
 & \xrightarrow[m\to +\infty]{}  0. \nonumber
\end{align}
\end{proof}
In the last line we have used that \(u_m\) converges uniformly to \(u\) in \(\Omega\) and that the \(W^{1-\frac{1}{n+\alpha},n+\alpha}(\p \Omega, \R^n)\) norms of \(u_m\) are uniformly bounded.  We will also need the following inequality, see e.g.\ \cite[Chapter 12]{Lindqvist_2019}.
\begin{lemma}\label{inequality}
There exists $C>0$ depending only on $n$ such that for any $\a\in[0,1]$ and $u,v\in \R^n$, it holds
\[ \left( |v|^{n+\a-2}v-|u|^{n+\a-2}u\right)\cdot (v-u) \geq C|v-u|^{n+\a}. \]
\end{lemma}
We now prove that $E_{n+\alpha}$ verifies the Palais--Smale condition.
\begin{proposition}\label{pertPalais}
For all $\alpha>0$, the functional \(E_{n+\alpha}\) satisfies the Palais--Smale condition in $\I_\a$: if $(u_m)_{m\in\N}\subset \I_\a$ is such that $E_{n+\alpha}(u_m) \leq C$ and $\dd_S E_{n+\alpha}(u_m) \rightarrow 0$ in $(W^{1,n+\alpha}(\Omega,\R^n))^*$, where \(d_SE_{n+\alpha}\) is defined in Definition \ref{def:definition_diff}, then, up to a subsequence, $u_m \rightarrow u$ strongly in $W^{1,n+\a}(\O,\R^n)$.
\end{proposition}
\begin{proof}
We have that $(u_m)$ is uniformly bounded in $W^{1,n+\a}$. This comes from the fact that $\|\dd u_m\|_{n+\a}^{n+\a}  = E_{\kappa}^\a$ and $\|{u_m}_{|\p \O}\|_\infty \leq 1$  (these two facts and the Poincar\'e inequality imply that $\|u_m\|_{n+\a}$ is bounded). Thus, up to a subsequence, we can assume that $u_m \rightharpoonup u$ in $W^{1,n+\a}(\O,\R^n)$ and $u_m \rightarrow u$ uniformly in $\O$, for some $u \in W^{1,n+\a}(\O,\R^n)$. Using Lemma \ref{inequality} it comes
\begin{eqnarray}
& & C\int_\O |\dd u_m-\dd u|^{n+\a} \nonumber\\[1mm]
 &\leq & \int_\O  |\dd u_m|^{n+\a-2}du_m-|\dd u|^{n+\a-2}\dd u :\dd u_m-\dd u  \nonumber \\[1mm]
& \ust{m\to +\infty}{\leq } & \int_\O   |\dd u_m|^{n+\a-2}\dd u_m: \dd u_m-\dd u  +\smallO(1) \nonumber \\[1mm]
& \ust{m\to +\infty}{\leq } & \int_\O   |\dd u_m|^{n+\a-2}\dd u_m : \dd \left(\pi_\a(u_m,u_m-u) +u_m-u - \pi_\a(u_m,u_m-u)\right) +\smallO (1) \nonumber \\[2mm]
&\ust{m\to +\infty}{=} & \dd_S E_{n+\alpha}(u_m).u(_m-u) \\
& & \quad \quad + \dd E_{n+\alpha} (u_m).\left[(u_m-u)-\pi_\a(u_m,u_m-u) \right] +\smallO(1) \nonumber \\[2mm]
&\ust{m\to +\infty}{\leq } & C' \|(u_m-u)-\pi_\a(u_m,u_m-u)\|_{1,n+\a}+\smallO(1) \rightarrow 0, \nonumber
\end{eqnarray}
by Lemma \ref{projection2}.
\end{proof}

\subsection{Mountain pass geometry}

Now that we know that \(E_{n+\alpha}\) is \(\C^1\) and that \(\I_{1,\alpha}\) is a \(C^2\) Banach manifold endowed with a Finsler structure we prove that we the main assumption \eqref{Mountain pass geometry} is satisfied with our choices \eqref{eq:choice_1}-\eqref{eq:choice_2} in  Theorem \ref{th:Mountain Pass}. These choices lead to define 
\[
	\pra = \left\{F\in C^0(\Omega,\I_{1,\alpha})\colon F(x) = \widetilde{M}_{r,x} \text{ if } x\in \partial \Omega \right\},
\]
and for $\alpha=0$ we have
\begin{equation}\label{eq:Mr0}
	\pr = \left\{F\in C^0(\Omega,\I_{1})\colon F(x) = \widetilde{M}_{r,x} \text{ if } x\in \partial \Omega \right\}.
\end{equation}
With this notation, for \(\alpha\geq 0\) we set
\begin{equation}\label{eq:ourc}
 c(\alpha,r)\coloneqq \inf_{F\in \pra} \max_{a\in \oline\Omega} E_{n+\alpha} \left(F(a)\right),
\end{equation}
and
\begin{equation}\label{eq:ourc1}
 c_1(\alpha,r) \coloneqq \max_{a\in \partial\Omega}\int_{\Omega} E_{n+\alpha} \left( \widetilde{M}_{r,a}\right) .
\end{equation}
Lemma \ref{lm:energy_almostMobius} shows directly the following result.
\begin{lemma}\label{lm:c1lemma}
	For $c_1(\alpha,r)$ defined in \eqref{eq:ourc1} we have
	\begin{equation}\label{eq:c1asymptotics}
		 \lim_{r\to 0}\, \lim_{\alpha\to 0}\, c_1(\alpha,r) = n^{\frac n2}|\bB^n|.
	\end{equation}
\end{lemma}

In order to estimate \(c(\alpha,r)\) we need the following result which is the main topological ingredient responsible for the presence of the mountain pass geometry in our situation. We define the \(n\)-harmonic extension operator \(T_n\colon W^{1-\frac{1}{n},n}(\p \Omega, \R^n) \rightarrow W^{1,n}(\Omega,\R^n)\) such that for all \(g\in W^{1-\frac{1}{n},n}(\p \Omega, \R^n)\) it holds
\begin{equation}\label{eq:def_nharmonic_extension}
	\left\{
	\begin{array}{rcll}
		\lap_{n} (T_{n}(g))& =& 0  \quad  &\mbox{in $\Omega$},\\
		T_{n}(g) &=& g \quad &\mbox{on $\partial \Omega$}.\\
	\end{array}
	\right.
\end{equation}

\begin{lemma}\label{lm:topology}
	There exists $r_0>0$ sufficiently small so that for all $r<r_0$ and for all $F\in \pr = \{G\in C^0(\overline \Omega, \I_1)\colon G=M_r \text{ on } \partial \Omega\}$ there exists a point $a_F\in \overline{\Omega}$ such that
	\[
	\frac{1}{|\Omega|} \int_{\Omega} T_n \left( \tr_{| \p \Omega} F(a_F)\right) =0.
	\] 
\end{lemma}

\begin{proof}
	Given $F \in \pr$, we define the following map:
	\begin{align*}
		G_r \colon \left| \begin{array}{c l l }
			\oline{\Omega} & \to & \oline{\bB^n}\\
			a & \mapsto & \frac{1}{|\Omega|} \int_\Omega T_n\left( \tr_{| \pl \Omega} F(a) \right).
		\end{array}
		\right.
	\end{align*}
	If $a\in\pl\Omega$, by definition of \(\pr\) we have 
	\begin{align*}
		G_r(a) = \frac{1}{|\Omega|} \int_\Omega T_n \left[ \tr_{| \pl\Omega} \left(\widetilde{M}_{(1-r)\Phi(a)}\circ \Phi_a \right) \right].
	\end{align*}
	Thus, $\tr_{| \pl\Omega} G_r$ is independent of the choice of $F\in \pr$. Combining \Cref{lm:unif_mobius} and \Cref{lm:uniform_nharm_extension}, we obtain
	\begin{align*}
		\left\| \tr_{| \pl\Omega} G_r - \tr_{| \pl\Omega} \Phi \right\|_{L^\infty(\Omega)} & \leq \frac{1}{|\Omega|}\, \left\| T_n \left[ \tr_{| \pl\Omega} \left(\widetilde{M}_{(1-r)\Phi(a)}\circ \Phi_a\right)\right]-\Phi(a) \right\|_{L^1(\Omega)} \\
		& \xrightarrow[r\to 0]{}{0}.
	\end{align*}
	Consequently, there exists $r_0>0$ small enough such that for any $r\in(0,r_0)$ the map \( G_r\) does not vanish on \(\p \Omega\) and 
	\begin{align*}
		\deg\left( \frac{G_r}{|G_r|},\pl\Omega \right) = \deg(\Phi,\pl\Omega) = 1.
	\end{align*}
	For such \(0<r\leq r_0\) it means that \( G_r\) has to vanish inside \(\Omega\). Indeed, otherwise \(\tr_{|\p \Omega}  G_r\) would be homotopic to a constant since \( \Omega\) is contractible. Hence there exists $a_F\in\Omega$ such that $G_r(a_F) = 0$.
\end{proof}

Using Lemma \ref{lm:topology}, we estimate from below $c(0,r)\eqqcolon c(r)$.
\begin{proposition}\label{pr:lowerc}
	Let
	\[
	c(r)\coloneqq \inf_{F\in \pr} \max_{a\in \overline{\Omega}} E_n( F(a))
	\]
	where the class $\pr$ is defined in \eqref{eq:Mr0}.	There exist $c_2 > n^{\frac n2}|\bB^n|$ and $r_0>0$ such that for all $r<r_0$ we have $c(r)\ge c_2$.
\end{proposition}

Before proceeding to the proof of this proposition we point out that the constant \(c_2\) depends on the domain \(\Omega\) and the proof crucially  uses that \(\Omega\) is not a round ball.
\begin{proof}
	By contradiction assume that for any $k\in\N^*$, there exists $r_k\in(0,\frac{1}{k})$  such that
	\[
	\inf \left\{ \max_{a\in\overline{\Omega}} E_{n}[F(a)]\colon F\in \mathcal{P}_{r_k,0}\right\} \le n^{\frac n2} |\bB^n| + \frac{1}{k}.
	\]
	By Lemma \ref{lm:topology}, there exist $F_k\in \mathcal{P}_{r_k,0}$ and $(a_k)_{k\in\N^*}\in\Omega$ such that
	\begin{enumerate}
		\item we have
		\begin{align}
			n^{\frac n2} |\bB^n| \leq E_n \big( T_n[\tr_{| \pl\Omega}F_k(a_k)] \big)\leq E_{n}[F_k(a_k)]\xrightarrow[k\to \infty]{} n^{\frac n2} |\bB^n|.
		\end{align}
		\item Each $T_n[\tr_{| \pl\Omega}F_k(a_k)]$ has zero average, i.e.,
		\begin{align}\label{eq:average}
			\frac{1}{|\Omega|} \int_{\Omega} T_n[\tr_{| \pl\Omega}F_k(a_k)] =0.
		\end{align}
	\end{enumerate}
	
	Let now $u_k \coloneqq T_n[\tr_{| \pl\Omega}F_k(a_k)]\subset \I_1$, it is a minimizing sequence for $E_n$ in \(\I_1\) and up to a subsequence
	\begin{align}
		u_k &\rightharpoonup u_\infty \quad \text { weakly in } W^{1,n}(\Omega,\bB^n), \label{eq:weakW1n}\\
		\dd u_k &\to \dd u_\infty \quad \text { a.e. } \label{eq:a.e. convergence}
	\end{align}
	The a.e.\ convergence of \(\dd u_k\) comes from the fact that we considered \((n+\alpha)\)-harmonic extensions and then elliptic estimates provide \(\C^1\) uniform bounds on the gradient in the interior of \(\Omega\), cf.\ e.g.\ \cite{Uhlenbeck_1977}. Thanks to \Cref{prop:Price_lemma}, we find
	\begin{align}\label{eq:contradiction_price}
		n^\frac{n}{2}\, |\bB^n| \geq \int_{\Omega} |\dd u_\infty|^n + n^\frac{n}{2}\, |\bB^n|\, |\deg(u_\infty,\pl\Omega)-1|  \geq n^\frac{n}{2}\, |\bB^n|\, |\deg(u_\infty,\pl\Omega)-1|.
	\end{align}
	We first deduce from \eqref{eq:contradiction_price} that  $\deg(u_\infty,\pl\Omega)\in\{0,1,2\}$. 
	
	If $\deg(u_{\infty},\pl\Omega)=2$, \eqref{eq:contradiction_price} implies that \(\int_{\Omega}|\dd u_\infty|^n=0\) and thus \(u_\infty\) is constant in \(\Omega\). But constant maps have degree zero and hence we obtain a contradiction.
	If $\deg(u_{\infty},\pl\Omega)=1$, then \eqref{eq:contradiction_price} reads
	\begin{align*}
		n^\frac{n}{2}\, |\bB^n| & \geq \int_{\Omega} |\dd u_\infty|^n.
	\end{align*}
	Hence, from Theorem \ref{th:Liouville} we deduce that  $u_{\infty}\colon \Omega\to \bB^n$ is a conformal map. By Liouville theorem, $\Omega$ must be a ball, which is excluded. 
	
	If $\deg(u_{\infty},\pl\Omega)=0$, we obtain again from \eqref{eq:contradiction_price} that $u_{\infty}$ is constant. By \eqref{eq:weakW1n}, we have that $\tr_{| \pl\Omega}u_k$ converges weakly to $\tr_{|\pl\Omega}u_{\infty}$ in $W^{1-\frac{1}{n},n}(\pl\Omega;\s^{n-1})$. Hence, we must have $|\tr_{|\pl\Omega}u_{\infty}|=1$ a.e.\ on $\pl\Omega$. Thus we find that \(u_\infty\) is a constant map and is in \(\mathbb{S}^{n-1}\). However, by using \eqref{eq:average} and \eqref{eq:weakW1n} along with the compact injection from \(W^{1,n}(\Omega,\R^n)\) into \(L^1(\Omega,\R^n)\) provided by Rellich--Kondrachov's Theorem, we also find that $u_{\infty}$ has average zero. This is again a  contradiction and this concludes the proof of Proposition \ref{pr:lowerc}.
\end{proof}

We now have  all the ingredients to show that we can apply Theorem \ref{th:Mountain Pass}.

\begin{proposition}\label{pr:validity_minmax}
	There exists $r_0>0$ such that for all $0<r<r_0$ there exists a sequence of positive numbers \( (\alpha_k)_k\) with \(\alpha_k\xrightarrow[]{k\to +\infty} 0\) such that
	\begin{equation}
		c_1(\alpha_k,r) < c(\alpha_k,r).
	\end{equation}
\end{proposition}

	\begin{proof}
		We first prove that for any $r>0$, it holds
		\begin{align}\label{eq:comparison_c}
			c(r) \leq \limsup_{\alpha\to 0}  c(\alpha,r).
		\end{align}
		Indeed,  by density of $W^{1,n+\alpha}$ in $W^{1,n}$, we have that
\begin{align*}
c(r) = \inf_{F\in \pr} \max_{a\in\bar{\Omega}} E_n \left(F(a) \right) = \inf_{F\in \pra} \max_{a\in\bar{\Omega}} E_n \left(F(a) \right).
\end{align*}
Using Hölder's and Young's inequalities, we obtain
\begin{align*}
c(r) & \leq \inf_{F\in \pra} \max_{a\in\bar{\Omega}} \brac{\int_{\Omega} |\dd F(a)|^{n+\alpha}}^{\frac{n}{n+\alpha}} |\Omega|^{1-\frac{n}{n+\alpha}}\\[2mm]
& \leq  \inf_{F\in \pra} \max_{a\in\bar{\Omega}} \int_{\Omega} |\dd F(a)|^{n+\alpha} \brac{1-\frac{\alpha}{n+\alpha}} +  \frac{\alpha}{n+\alpha}|\Omega|.
\end{align*}
We thus find that 
\begin{align}
	c(r) \leq c(\alpha,r)\brac{1-\frac{\alpha}{n+\alpha}} + \frac{\alpha}{n+\alpha}|\Omega|,
\end{align}
	and  inequality \eqref{eq:comparison_c} follows.
	
 Combining  \eqref{eq:comparison_c}, Proposition \ref{pr:lowerc} and Lemma \ref{lm:c1lemma}, we deduce that for any $r>0$ small enough, 
 \begin{equation}\label{eq:comparison_allc}
 	\lim_{r\to 0}\, \lim_{\alpha\to 0} \, c_1(\alpha,r) = n^{\frac{n}{2}}\, |\bB^n| < c_2 \leq c(r) \leq \limsup_{\alpha\to 0} c(\alpha,r).
 \end{equation}
Hence, there exists \(r_0>0\) such that, for all \(0<r\leq r_0\), there exists a sequence \( (\alpha_k)_k\subset \R^+\) with \( \alpha_k \xrightarrow[k \to +\infty]{} 0\) such that \( c_1(\alpha_k,r)<c(\alpha_k,r)\).
\end{proof}

Combining Proposition \ref{pr:validity_minmax}, Theorem \ref{th:Mountain Pass} and Proposition \ref{pertPalais}, we obtain the following conclusion.

\begin{theorem}\label{th:existence_nplusalpha}
There exists \(r_0>0\) such that for all \(0<r<r_0\)  we can find a sequence of positive numbers \(\alpha_k \to 0\) and \(r_k \to 0\) and \( (u_k)_{k\geq 1}\subset \I_{1,\alpha_k}\) such that 
\begin{equation}\label{eq:freealphakharmonic}
\int_{\Omega} |\dd u_k|^{n-2} \dd u_k : \dd \varphi =0
\end{equation}
for all \(\varphi \in W^{1,n+\alpha_k}(\Omega,\R^n)\) such that \(\tr_{| \p \Omega}\varphi(x) \in T_{u(x)}\mathbb{S}^{n-1}\) for every \(x\in \p \Omega\). Furthermore
\begin{equation}
E_{n+\alpha_k} (u_k) =c(\alpha_k,r),
\end{equation}
where \(c(\alpha_k,r)\) is defined in \eqref{eq:ourc}.
\end{theorem}

\begin{proposition}\label{prop:prochegroundstate}
Let \(\e>0\), there exist \(L>0\) and \(r_0>0\) such that if \(\Phi:\overline{\Omega}\rightarrow \overline{\mathbb{B}^n}\) is a \(\C^1\)-diffeomorphism satisfying \( \|\Phi-\id\|_{\C^1(\Omega)}<L\) and if \( 0<r<r_0\) then for \(k\) large enough we have
\begin{equation}
c(\alpha_k,r)<n^{\frac{n}{2}}|\mathbb{B}^n|+\e
\end{equation}
where \((\alpha_k)_k\) is the sequence defined in Proposition \ref{pr:validity_minmax}.
\end{proposition}

\begin{proof}
By definition 
\[
 c(\alpha_k,r)=\inf_{F\in \mathcal{P}_{r,\alpha_k}} \max_{a\in \overline{\Omega}}E_{n+\alpha_k}\left(F(a) \right).\]
 In particular, since \(a\in \overline{\Omega} \mapsto \widetilde{M}_{r,a}=M_{(1-r)\Phi(a)}\circ \Phi_a\) belongs to \(\mathcal{P}_{r,\alpha_k}\) for any \(k\in \mathbb{N}\), we find that
 \begin{align*}
 c(\alpha_k,r) &\leq  \max_{a\in \overline{\Omega}}E_{n+\alpha_k}\left(M_{(1-r)\Phi(a)}\circ \Phi_a \right).
 \end{align*}
 Now the conclusion follow thanks to \eqref{eq:uniform_Mobius} in Lemma \ref{lm:energy_almostMobius}.
\end{proof}

\section{Energy gap for the Möbius transformations}\label{sec:gapMobius}

The goal of this section is to prove the energy gap \Cref{th:Energy_gap_Mobius} for Möbius maps. To do so, we proceed by contradiction and obtain a family of degree 1 free boundary $n$-harmonic maps $(u_{\eps})_{\eps>0}$ with energy converging to $n^{\frac{n}{2}}\, |\bB^n| = E_n(\id)$. We will first prove in \Cref{lm:renormalized_FBnharm} that if we normalize the maps $u_{\eps}$ by the condition $u_{\eps}(0)=0$ (this is a balancing condition), then $u_{\eps}$ converges to a rotation in the $C^{\infty}(\bB^n)$ topology, which we can assume to be the identity map. In particular, there is no bubbles appearing. Then we prove in \Cref{lm:Poincare} an optimal version of the trace inequality $\|f\|_{L^2(\s^{n-1})}\aleq \|\dd f\|_{L^2(\bB^n)}$ for functions $f\in W^{1,2}(\bB^n,\R^n)$ such that $\int_{\s^{n-1}} f=0$. To prove \Cref{th:Energy_gap_Mobius}, we first slightly change the balancing condition in order to keep the convergence of $u_{\eps}$ to the identity, but now imposing instead $\int_{\s^{n-1}}u_{\eps}=0$, see \Cref{cl:Balancing}. Then we use the fact that the identity map has no critical points to show that $E_n$ is "strictly coercive" in a neighbourhood of the identity. As a consequence, we show that if $\|u_{\eps}-\id\|_{L^2(\s^{n-1})}>0$, then $\frac{ u_{\eps}-\id }{ \|u_{\eps}-\id\|_{L^2(\s^{n-1})} }$ has to converge to a linear map given by the multiplication by a antisymmetric matrix. However, this is the first term in an asymptotic expansion that can be cancelled by choosing an appropriate rotation, see the condition \eqref{eq:minimality}. This constitutes our contradiction. These are the ideas of our proof to the following result.

\begin{theorem}\label{th:Energy_gap_Mobius}
	There exists $\eps_M>0$ depending only on $n$ such that the following holds. Let $u \in W^{1,n}(\bB^n ,\bB^n)$ be a solution to the following system:
	\begin{align}\label{eq:system_FBnharm_deg1}
		\left\{
		\begin{array}{l l}
			\lap_n u = 0 & \text{ in }\bB^n,\\[2mm]
			|u| = 1 & \text{ on }\s^{n-1},\\[2mm]
			|\dd u|^{n-2}\p_\nu u \wedge u = 0 & \text{ on }\s^{n-1}, \\[2mm] 
			\deg(u,\s^{n-1}) = 1.
		\end{array}
		\right.
	\end{align}
	Assume that 
	\begin{align*}
		\int_{\bB^n}|\dd u|^n \leq n^{\frac{n}{2}} |\bB^n| + \eps_M.
	\end{align*}
	Then $u$ is a Möbius map or more precisely, there exists \(a\in \mathbb{B}^n\) and \(R\in O_n(\R)\) such that \( u=R\circ M_a\), where \(M_a\) is defined in Definition \ref{def:Mobius}.
\end{theorem}

We proceed by contradiction. First we prove the following result.

\begin{lemma}\label{lm:renormalized_FBnharm}
	Consider a sequence $(u_\eps)_{\eps>0}$ of solutions to 
	\begin{align}\label{eq:normalization}
		\begin{cases} 
			\lap_n u_\eps = 0 & \text{ in }\bB^n,\\[2mm]
			|u_\eps| = 1 & \text{ on }\s^{n-1},\\[2mm]
			|\dd u_\eps|^{n-2}\p_\nu u_\eps \wedge u_\eps = 0 & \text{ on }\s^{n-1},\\[2mm]
			\deg(u_\eps,\s^{n-1}) = 1, &\\[2mm]
			u_\eps(0) =0.
			\end{cases} 
	\end{align}
	Assume that 
	\begin{align}\label{eq:low_energy}
		\int_{\mathbb{B}^n} |\dd u_\eps|^n \xrightarrow[\eps\to 0]{} n^{\frac{n}{2}} |\bB^n|.
	\end{align}
	Then, up to a subsequence, each $u_\eps$ is in $\C^\infty(\overline{\mathbb{B}^n},\overline{\mathbb{B}^n})$ and $(u_\eps)_{\eps>0}$ converges to a rotation in $\C^\infty\left(\overline{\bB^n},\R^n\right)$. 
\end{lemma}

\begin{proof}
	
	By \cite[Theorem 1.3]{MartinoMazowieckaRodiac2025}, there exist \(K\in \mathbb{N}\), points \(a_1,\dots,a_K\in \p \Omega\) and maps \(\omega_0,\omega_1,\dots,\omega_K\colon\overline{\mathbb{B}^n}\rightarrow \overline{\mathbb{B}^n}\) which are critical points of \(E_n\) in \(\I\) such that, up to a subsequence, $u_\eps$ converges in $C^1_{\text{loc}}(\overline{\bB^n}\setminus \{a_1,\dots,a_k\},\R^n)$ and weakly in $W^{1,n}(\bB^n,\R^n)$ to $\omega_0$ and such that
	\begin{align}
		& \int_{\bB^n} |\dd u_\eps|^n \xrightarrow[\eps\to 0]{} n^{\frac{n}{2}} |\bB^n| = \int_{\bB^n} \left( |\dd \omega_0|^n + \sum_{k=1}^K |\dd \omega_k|^n \right), \label{eq:quant_energy}\\
		& 1 = \deg(u_\eps) = \deg(\omega_0) + \sum_{k=1}^K \deg(\omega_k). \label{eq:quant_degree}
	\end{align}
	Let $d_k \coloneqq \deg(\omega_k)$. From \eqref{eq:quant_degree}, it holds
	\begin{align*}
		1 \geq |\deg(\omega_0)| + \sum_{k=1}^K |\deg(\omega_k)|.
	\end{align*}
	Hence, there exists $k_0\in\{0,\ldots,K\}$ such that $d_{k_0} = 1$ and $d_k = 0$ for all $k\neq k_0$. From \eqref{eq:quant_degree} and Lemma \ref{lowbound}, we also have
	\begin{align*}
		0 \geq n^{\frac{n}{2}} |\bB^n| - \int_{\bB^n} |\dd \omega_{k_0}|^n \geq \sum_{k\neq k_0} \int_{\bB^n} |\dd \omega_k|^n \geq 0.
	\end{align*}
	Hence, each $\omega_k$, for $k\neq k_0$, is constant. However, since $u_\eps(0) = 0$ and $(u_\eps)_{\eps>0}$ converges strongly in the $\C^1$-topology on $B(0,\frac{1}{2})$, we have $\omega_0(0) = 0$. Hence, $\omega_0$ is not constant. Thus, $k_0 = 0$ and there is no other bubble. Furthermore, \eqref{eq:quant_energy} reduces to
	\begin{align*}
		\int_{\bB^n} |\dd \omega_0|^n = n^{\frac{n}{2}} |\bB^n|.
	\end{align*}
	From Lemma \ref{lowbound}, we deduce that   $\omega_0$ is conformal. Thus, from Theorem \ref{th:Liouville}, $\omega_0$ is a Möbius map. Since $\omega_0(0) = 0$, $\omega_0$ is a rotation, cf.\ e.g.\ \cite[Theorem 3.4.1]{Beardon_1983}. Since there is no bubble, $(u_\eps)_{\eps>0}$ converges to $\omega_0$ in $\C^1\left( \overline{\bB^n},\R^n \right)$. In particular, $|\dd u_\eps|\geq\frac{1}{2}$ everywhere in $\bB^n$ for $\eps$ small enough. Hence, from elliptic regularity theory, $u_\eps$ is in $\C^\infty(\overline{\mathbb{B}^n},\overline{\mathbb{B}^n})$ with uniform bounds and the convergence holds in $\C^\infty$, not only in $\C^1$.
\end{proof}

In order to prove \Cref{th:Energy_gap_Mobius}, we will need to have sharp constants in the trace inequality $W^{1,2}(\bB^n)\hookrightarrow L^2(\s^{n-1})$. In order to study this inequality we use spherical harmonics.
Let \(\Delta_{\mathbb{S}^{n-1}}\) be the Laplace-Beltrami operator on the sphere \(\mathbb{S}^{n-1}\). For each $k\in \N$, there exists \(L(k)\in \mathbb{N}\) and $(Y_{k,\ell})_{1\leq \ell\leq L(k)}$ functions satisfying
	\begin{align*}
		\begin{cases}
			-\lap_{\s^{n-1}} Y_{k,\ell} = k(k+n-2)\, Y_{k,\ell}, \\[2mm]
			\|Y_{k,\ell}\|_{L^2(\s^{n-1})} = 1.
		\end{cases}
	\end{align*}
	Furthermore, the system \((Y_{k,\ell})_{k\in \mathbb{N}, 1 \leq \ell\leq L(k)}\) is an orthonormal Hilbert basis of \(L^2(\mathbb{S}^{n-1})\).

\begin{lemma}\label{lm:Poincare}
	 Let $u=(u^1,\dots,u^n) \in \C^{\infty}(\overline{\bB^n};\R^n)$. For each \(1\leq i\leq n\), we decompose \(u^i\) using spherical coordinates \( (r,\theta)\in \R^+\times \mathbb{S}^{n-1}\) and spherical harmonics:
	\begin{align*}
		u^i(r,\theta) = f_0(r)+\sum_{k=1}^{\infty} \sum_{\ell=1}^{L(k)} f_{k,\ell}^i(r)\, Y_{k,\ell}(\theta).
	\end{align*}
	Then, the  Dirichlet energy of $u$ is given by
	\begin{align*}
		\int_{\bB^n} |\dd u|^2&  =|\mathbb{S}^{n-1}|\int_0^1 |f_0'(r)|^2r^{n-1} \, \dd r \\
		& \quad +  \sum_{i=1}^n \sum_{k=1}^{\infty} \sum_{\ell=1}^{L(k)} \int_0^1\left[ (\p_r f_{k,\ell}^i)^2 + \frac{k(k+n-2)}{r^2}\, (f_{k,\ell}^i)^2 \right]\, r^{n-1}\, \dd r.
	\end{align*}
	If we assume furthermore that \(\int_{\mathbb{S}^{n-1}} u =0\) then \(f_0(1)=0\) and
	\begin{align*}
	\int_{0}^1 |f_0(r)|^2 r^{n-1} \dd r \leq C(n) \int_0^1 |f_0'(r)|^2 r^{n-1} \dd r,
	\end{align*}
	\begin{align*}
		\int_{\s^{n-1}} |u|^2 \leq \sum_{i=1}^n \sum_{k=1}^{\infty} \left( \frac{1}{k} \sum_{\ell=1}^{L(k)}  \int_0^1\left[ (\p_r f_{k,\ell}^i)^2 + \frac{k(k+n-2)}{r^2}\, (f_{k,\ell}^i)^2 \right]\, r^{n-1}\, \dd r\right).
	\end{align*}
\end{lemma}

\begin{proof}
	The Dirichlet energy of $u$ satisfies
	\begin{align*}
		\int_{\bB^n} |\dd u|^2 & = \sum_{i=1}^n \int_{\bB^n} |\dd u^i|^2 = \sum_{i=1}^n \int_0^1 \int_{\s^{n-1}} \left[ (\p_r u^i)^2 + \left| \frac{\g_{\theta} u^i}{r}\right|^2 \right]\, r^{n-1}\, \dd \theta\, \dd r.
	\end{align*}
	Since the $Y_{k,\ell}$ are orthonormal in $W^{1,2}(\s^{n-1})$, we obtain
	\begin{multline*} 
		\int_{\bB^n} |\dd u|^2  = \sum_{i=1}^n \Big( |\mathbb{S}^{n-1}|\int_0^1 |(f_0^i)'(r)|^2r^{n-1} \dd r  \\
		+ \sum_{k=1}^{\infty} \sum_{\ell=1}^{L(k)} \int_0^1  \left[ (\p_r f^i_{k,\ell})^2 + (f^i_{k,\ell})^2\, \int_{\s^{n-1}}  \left| \frac{\g_{\theta} Y_{k,\ell}}{r}\right|^2\, d\theta \right]\, r^{n-1}\, \dd r \Big).
	\end{multline*}
	
	The Dirichlet energy of $Y_{k,\ell}$ is given by 
	\begin{align*}
		\int_{\s^{n-1}}  \left| \frac{\g_{\theta} Y_{k,\ell}}{r}\right|^2\, \dd\theta & = -\frac{1}{r^2} \int_{\s^{n-1}} (\lap_{\s^{n-1}} Y_{k,\ell})\, Y_{k,\ell}\, \dd \theta  \\[3mm]
		& = \frac{ k(k+n-2) }{r^2}  \int_{\s^{n-1}} |Y_{k,\ell}|^2 \dd \theta \\[3mm]
		& = \frac{ k(k+n-2) }{r^2} .
	\end{align*}
	We obtain 
	\begin{align*} 
		\int_{\bB^n} |\dd u|^2 & = \sum_{i=1}^n \Big( |\mathbb{S}^{n-1}|\int_0^1 |(f_0^i)'(r)|^2r^{n-1} \dd r\\
		& \quad + \sum_{k=1}^{\infty} \sum_{\ell=1}^{L(k)} \int_0^1  \left[ (\p_r f^i_{k,\ell})^2 + \frac{ k(k+n-2) }{r^2} \, (f^i_{k,\ell})^2\, \dd \theta \right]\, r^{n-1}\, \dd r\Big).
	\end{align*}
	Now, since \(\int_{\mathbb{S}^{n-1}} Y_{k,\ell}(\theta) \dd \theta =0\), we find that \(0=\int_{\mathbb{S}^{n-1}} u= |\mathbb{S}^{n-1}| f_0(1)\) and hence \( f_0(1)=0\). 
	Thus we obtain that 
	\begin{align*}
		\int_{\s^{n-1}} |u|^2 = \sum_{i=1}^n \sum_{k=1}^{\infty} \sum_{\ell=1}^{L(k)} f^i_{k,\ell}(1)^2.
	\end{align*}
	
	Notice that since \(u\) is smooth in \(\mathbb{B}^n\) and since \( (Y_{k,\ell})_{k,L(k)}\) is an orthonormal family in \(L^2(\mathbb{S}^{n-1})\), we can write that, for \( r\in \R^+_*\)  and \(k\in \mathbb{N}^*\) and \(1\leq \ell\leq L(k)\) we have
	\begin{align*}
	f_{k,\ell}(r)=\int_{\mathbb{S}^{n-1}}u(r,\theta)Y_{k,\ell}(\theta) \dd \theta.
	\end{align*}
	Passing to the limit \(r\to 0\) thanks to the dominated convergence theorem we realize that 
	\begin{align*}
	f_{k,\ell}(0)= u(0) \int_{\mathbb{S}^{n-1}}Y_{k,\ell}(\theta) \dd \theta =0.
	\end{align*}
	Given a function $f\colon [0,1]\to \R$, we define
	\begin{align*}
		E_{n,k}(f) \coloneqq \int_0^1  \left[ (\p_r f)^2 + \frac{ k(k+n-2) }{r^2} \, f^2 \right]\, r^{n-1}\, \dd r.
	\end{align*}
	We want to show that for functions $f\colon [0,1]\to \R$ such that $f(0)=0$, it holds
	\begin{align*}
		f(1)^2 \leq \frac{1}{k} E_{n,k}(f).
	\end{align*}
	By homogeneity, it suffices to prove it for functions satisfying $f(1)=1$. In dimension 1, we have $W^{1,2}(0,1)\hookrightarrow C^{0,\frac{1}{2}}[0,1]$. Hence, the following variational problem has a solution for all $n\geq 2$ and $k\geq 1$:
	\begin{align}\label{eq:var_problem}
		\inf \big\{ E_{n,k}(f) \mid f\colon [0,1]\to \R,\ f(1)=1 \big\}.
	\end{align}
	The Euler--Lagrange equation associated to the minimization problem \eqref{eq:var_problem} is given by
	\begin{align*}
		0 & = -\frac{\dd}{\dd r}\left(r^{n-1}\, \frac{\dd f}{\dd r}\right) + k(k+n-2)\, r^{n-3}\, f(r) \\[3mm]
		& = -r^{n-1}\, f''(r) -(n-1)\, r^{n-2}\, f'(r) + k(k+n-2)\, r^{n-3}\, f(r).
	\end{align*}
	Dividing by $-r^{n-1}$, we obtain a solution to 
	\begin{align}\label{eq:EL}
		\begin{cases} 
			\displaystyle f''(r) +\frac{n-1}{r}\, f'(r) - \frac{k(k+n-2)}{r^2}\, f(r) =0. \\[2mm]
			f(1) =1.
		\end{cases} 
	\end{align}
	The solutions of the above problems are of the form $\left( r\mapsto (1-\lambda)r^k + \lambda r^{-k}\right)$ for some $\lambda\in\R$. However, the function $r^{-k}$ does not have finite energy for $k\geq \frac{n-2}{2}$. Hence, the function $(r\mapsto r^k)$ is the unique solution to \eqref{eq:var_problem} for $k\geq \frac{n-2}{2}$. For this function, it holds
	\begin{align*}
		E_{n,k}(r^k) & = \int_0^1  \left[ k^2\, r^{2k-2} +  k(k+n-2) \, r^{2k-2}\, \dd \theta \right]\, r^{n-1}\, \dd r \\[2mm]
		& = k(2k+n-2) \int_0^1 r^{2k+n-3}\, \dd r \\[2mm]
		& = \frac{k(2k+n-2)}{2k+n-2} =k.
	\end{align*}
	If now $k<\frac{n-2}{2}$, it holds
	\begin{align*}
		E_{n,k}(r^{-k}) & = \int_0^1  \left[ k^2\, r^{-2k-2} +  k(k+n-2) \, r^{-2k-2}\, \dd \theta \right]\, r^{n-1}\, \dd r \\[2mm]
		& = k(2k+n-2) \int_0^1 r^{-2k+n-3}\, \dd r \\[2mm]
		& = \frac{k(2k+n-2)}{-2k+n-2} >k.
	\end{align*}
	We obtain for any $\lambda\in\R$, and $0\leq k<\frac{n-2}{2}$,
	\begin{align*}
		& \inf \left\{ E_{n,k}(f) \mid f\colon [0,1]\to \R,\ f(1)=1 \right\} \\[2mm]
		& \geq E_{n,k}\left( (1-\lambda)r^k + \lambda r^{-k} \right) \\[2mm]
		& \geq (1-\lambda)^2\, E_{n,k}(r^k) + \lambda^2\, E_{n,k}(r^{-k}) \\[2mm]
		&\qquad  + 2\lambda(1-\lambda)\int_0^1 \left[ (k\, r^{k-1})(-k\, r^{-k-1})+\frac{k(k+n-2)}{r^2}\, r^k\, r^{-k} \right]\, r^{n-1}\, \dd r \\[2mm]
		& \geq (1-\lambda)^2\, k + \lambda^2\, k + 2\lambda(1-\lambda) \int_0^1 k(n-2)\, r^{n-3}\, \dd r \\[2mm]
		& \geq (1-\lambda)^2\, k + \lambda^2\, k + 2\lambda(1-\lambda)\, k \\[2mm]
		& \geq k\left(1-\lambda+\lambda\right)^2 =k.
	\end{align*}
\end{proof}

We now prove \Cref{th:Energy_gap_Mobius}.

\begin{proof}
	By contradiction, assume that there exists a sequence $(u_\eps)_{\eps>0}$ solutions to \eqref{eq:system_FBnharm_deg1} such that \( (u_\e)_{\e>0}\) is not a M\"obius transformation of \(\mathbb{B}^n\) in the sense of Definition \ref{def:Mobius} and such that
	\begin{align*}
		\int_{\bB^n} |\dd u_\eps|^n \xrightarrow[\eps\to 0]{} n^{\frac{n}{2}} |\bB^n|.
	\end{align*}
	Since each $u_\eps$ has degree 1, there exists $x_\eps\in\mathbb{B}^n$ such that $u_\eps(x_\eps) = 0$. Let $M_\eps$ be a Möbius transform satisfying $M_\eps(0) = x_\eps$, then \(\tilde{u}_\e=u_\e \circ M_\e\) is also a solution to \eqref{eq:system_FBnharm_deg1} and it satisfies that \( \tilde{u}_\e(0)=0\). From Lemma \ref{lm:renormalized_FBnharm}, up to a subsequence, $(\tilde{u}_\eps)_{\eps>0}$ converges in \(\C^\infty(\overline{\mathbb{B}}^n,\R^n)\) to a rotation denoted by \(R_0\).
	
	\begin{claim}\label{cl:Balancing}
	For \(\e>0\) small enough, there exists another M\"obius transform \(\hat{M}_\e\) such that \(\hat{u}_\e:=\tilde{u}_\e \circ \hat{M}_\e\) satisfies \eqref{eq:system_FBnharm_deg1}, \(\int_{\mathbb{S}^{n-1}} \hat{u}_\e =0\) and \(\hat{u}_\e \xrightarrow[\e \to 0]{} R_0\in O_n(\R)\) in \(\C^\infty(\overline{\mathbb{B}}^n,\R^n)\).
	\end{claim}
	
	\begin{proof}
	For \(\e>0\) and \(a\in \mathbb{B}^n\) we define
	\begin{align*}
	F_\e(a)=\int_{\mathbb{S}^{n-1}} \tilde{u}_\e \circ M_a(x) \dd \mathcal{H}^{n-1}(x),
	\end{align*}
	where \(M_a\) is a M\"obius transform defined in Definition \ref{def:Mobius}. Note that since for each \(x\in \mathbb{B}^n\) the map \(a \in \mathbb{B}^n \mapsto M_a (x)\) is \(\C^\infty\) we can see that \(F_a\) is also \(\C^\infty\). Since \(\tilde{u}_\e \xrightarrow[\e \to 0]{} \id\) in \(\C^\infty(\overline{\mathbb{B}^n}, \R^n)\) we can show that 
	\begin{align*}
	F_\e \xrightarrow[\e \to 0]{} a \mapsto \int_{\mathbb{S}^{n-1}}M_a(x) \dd \mathcal{H}^{n-1}(x)\coloneqq F_0(a)  \text{ in } \C^2_{\text{loc}}(\mathbb{B}^n,\R^n).
	\end{align*}
	The map \(F_0:\mathbb{B}^n \rightarrow \mathbb{B}^n\) satisfies that \(F_0(0)=\int_{\mathbb{S}^{n-1}} x \dd \mathcal{H}^{n-1}(x)=0\).
	Furthermore, we can compute that for \(v\in \R^n\)
	\begin{align*}
	\dd F_0(0).v & =\int_{\mathbb{S}^{n-1}} \frac{d}{dt} \Bigl\vert_{t=0} M_{tv}(x) \dd \mathcal{H}^{n-1}(x) \\
	& =\int_{\mathbb{S}^{n-1}} \frac{d}{dt} \Bigl\vert_{t=0} \frac{|tv-x|^2tv+(1-|tv|^2)(tv-x)}{|tv|^2|tv-x|^2+(1-|tv|^2)^2+2 tv\cdot (tv-x)(1-|tv|^2)} \dd \mathcal{H}^{n-1}(x) \\
	&= \int_{\mathbb{S}^{n-1}} \Big( |x|^2v+v\\
	& \quad -x\frac{d}{dt}\Bigl\vert_{t=0} \Big[ |tv|^2|tv-x|^2+(1-|tv|^2)^2 +2tv\cdot (tv-x) (1-|tv|^2) \Bigr] \Big)\, \dd \mathcal{H}^{n-1}(x) \\
	&= 2|\mathbb{S}^{n-1}| v-\int_{\mathbb{S}^{n-1}} 2(v\cdot x) x \, \dd \mathcal{H}^{n-1}(x).
	\end{align*}
	But, for \(1\leq i \leq n\), we have
	\begin{align*}
	\sum_{j=1}^n v_j \int_{\mathbb{S}^{n-1}} x_ix_j \dd \mathcal{H}^{n-1}(x) =v_i\int_{\mathbb{S}^{n-1}} |x_i|^2 \, \dd \mathcal{H}^{n-1}(x).
	\end{align*}
	Since \[\sum_{i=1}^n \int_{\mathbb{S}^{n-1}} |x_i|^2 \dd \mathcal{H}^{n-1}(x) =\int_{\mathbb{S}^{n-1}} 1 \, \dd \mathcal{H}^{n-1}(x)=|\mathbb{S}^{n-1}|\]
	we find that for \(1\leq i \leq n\)
	\begin{align*}
	\int_{\mathbb{S}^{n-1}}|x_i|^2 \, \dd \mathcal{H}^{n-1}(x)=\frac{|\mathbb{S}^{n-1}|}{n}.
	\end{align*}
	We thus obtain that
	\begin{align*}
	\dd F_0(0).v=2|\mathbb{S}^{n-1}|-2\frac{|\mathbb{S}^{n-1}|}{n}v =\frac{2(n-1)}{n}|\mathbb{S}^{n-1}|.
	\end{align*}
	Hence \(\dd F_0(0)=c_n \id\) for some \(c_n \in \R\) and \(\dd F_0(0)\) is invertible. We thus find that for \(r>0\) small enough \(\dd F(x)\) is invertible if \(|x|<r\). By \(\C^1_{\text{loc}}(\mathbb{B}^n)\)-convergence of \(F_\e\) to \(F_0\), we obtain that for some \(r>0\) small enough  and for \(\e_0>0\) small enough, it holds that for all \(0<\e<\e_0\), the map \(\dd F_\e(x)\) is invertible if \(|x|<r\) and
	\begin{align*}
	\sup_{0<\e<\e_0} \| \left[ \dd F_\e(x)\right]^{-1}\|_{L^\infty(B(0,r))}.
	\end{align*}
	Since \(F_\e\) is \(\C^2\) in \(\mathbb{B}^n\), for \(a\in \mathbb{B}^n\) close enough to \(0\), from Taylor--Young's theorem we can write that
	\begin{align*}
	F_\e(a) =F_\e(0)+\dd F_\e(0).a+\mathcal{R}_\e(a),
	\end{align*}
	with 
	\begin{equation}\label{eq:cst_ind_eps}
	|\mathcal{R}_\e(a)| \leq C |a|^2 \text{ and } C \text{ is independent of } \e
	\end{equation}
	 for \(\e>0\) small enough. Indeed, since \(F_0\) is \(\C^2(\mathbb{B}^n)\), we have that \( \dd^2 F_0\) is uniformly bounded in \( B(0,1/2)\). From the \(\C^2_{\text{loc}}(\mathbb{B}^n)\)-convergence of \(F_\e\) towards \(F_0\) we find that \(\dd^2 F_\e\) is also uniformly bounded in \( B(0,1/2)\) for \( \e>0\) small enough.
	We look for \(a\in \mathbb{B}^n\) such that \( F_\e(a)=0\) and then we are lead to solve 
	\begin{align}\label{eq:sol_zero_average}
	a =\left[ \dd F_\e(0)\right]^{-1} \left(-F_\e(0)-\mathcal{R}_\e(a) \right).
	\end{align}
	By using \eqref{eq:cst_ind_eps}, we can apply a fixed point argument to find that there exists a unique solution \(a_\e\in B(0,r)\) to \eqref{eq:sol_zero_average}, if \(r>0\) is small enough (independently of \(\e>0\)).
	Furthermore, up to a subsequence, we can assume that \(a_\e \to a\in \overline{B(0,r)}\), from the uniform convergence of \(F_\e\) to \(F_0\) we find that \( F_0(a)=0\) but then, choosing \(r>0\) small enough, by the local inverse theorem, we necessarily find that \(a=0\). This shows that, for the whole sequence \(a_\e \to 0\) and thus \(\tilde{u}_\e \circ M_{a_\e} \xrightarrow[\e \to 0]{} R_0\) in \(\C^\infty(\mathbb{B}^n,\R^n)\).
	\end{proof}
	 We now consider the following map
	\begin{align*}
		\forall R\in SO(n),\qquad E(R) \coloneqq \int_{\s^{n-1}} |R\, \hat{u}_{\eps} -R_0|^2=\int_{\s^{n-1}} |R_0^{-1}R\, \hat{u}_{\eps} -\id|^2 .
	\end{align*}
	 $E\colon SO(n)\to [0,+\infty)$ is a continuous map on a compact set. Hence it has a minimum $R_{\eps}\in SO(n)$. Up to changing $u_{\eps}$ in $R_0^{-1}R_{\eps}\, \hat{u}_{\eps}$, we can assume that \(u_\e\) is a solution to \eqref{eq:system_FBnharm_deg1} such that 
	\begin{itemize}
	\item[1)] \( \int_{\mathbb{S}^{n-1}}u_\e=0\),
	\item[2)] \( u_\e \xrightarrow[]{\e \to 0} \id\) in \(\C^\infty(\mathbb{B}^n,\R^n)\) and
	\item[3)] \begin{align}\label{eq:minimality}
		\forall R\in SO(n),\qquad 0<\int_{\mathbb{S}^{n-1}} |u_\e-\id|^2 \leq \int_{\mathbb{S}^{n-1}} |R u_\e-\id|^2.
	\end{align}
	\end{itemize}

	Let $\eta \colon L^1(\mathbb{S}^{n-1},\R^n) \rightarrow L^1(\mathbb{B}^n,\R^n)\) be a continuous extension operator such that \(\eta\) is also continuous from \(L^p(\mathbb{S}^{n-1},\R^n)\) to \(L^p(\mathbb{B}^n,\R^n)\) and from \(W^{1-\frac{1}{p},p}(\mathbb{S}^{n-1},\R^n)\) to \(W^{1,p} (\bB^n,\R^n)$ for every \( 1\leq p<+\infty\). Such an extension operator can be given by an extension by convolution in the neighbourhood of \(\mathbb{S}^{n-1}\) and then using a cut-off function, see e.g.\ \cite[Lemma 15.47]{Brezis_Mironescu_2021}.  Given a map $u\in W^{1,n}(\bB^n,\bB^n)$ such that $u \colon \s^{n-1}\to \s^{n-1}$, given $v\in \C^\infty(\overline{\bB^n}, \overline{\bB^n})$ we define
	\begin{equation}\label{eq:def_proj}
		\pi_u(v)  \coloneqq  v-\eta\left( \left[ \scal{u}{v} u\right]_{|\s^{n-1}} \right), \quad \pi_u^\perp(v) \coloneqq \eta\left( \left[ \scal{u}{v} u\right]_{|\s^{n-1}} \right).
	\end{equation}

With this notation, $v = \pi_u(v) + \pi_u^\perp(v)$ and, since \( (u_\e)_{\e>0}\) is a solution to \eqref{eq:system_FBnharm_deg1}, we have
	\begin{align*}
		0 =&\ \dd E_n(u_\eps). \pi_{u_\eps}(u_\eps-\id) \\
		= &\ \int_{\bB^n} |\dd u_{\eps}|^{n-2} \dd u_{\eps} : \dd \pi_{u_{\eps}}(u_{\eps}-\id)\\
		\ust{\eps\to 0}{=} & \int_{\bB^n} |\dd (\id)|^{n-2} \dd (\id) : \dd \pi_{u_{\eps}}(u_{\eps}-\id)\\
		& +\int_{\bB^n} \Big((n-2) |\dd (\id)|^{n-4} \left[\dd (\id):\dd(u_{\eps} - \id )\right] \dd u_\e \\
		& \quad \quad + |\dd (\id)|^{n-2} \dd(u_{\eps} - \id)\Big): \dd \pi_{u_{\eps}}(u_{\eps}-\id)\\
		& + \smallO \left(\int_{\bB^n} |\dd(u_{\eps}-\id)|\, |\dd \pi_{u_{\eps}}(u_{\eps}-\id)| \right).\\
	\end{align*}
	We observe that 
\begin{align*}
\int_{\mathbb{B}^n}&= 2\int_{\mathbb{B}^n}| \dd (u_\e-\id)|^2 +2 \int_{\mathbb{B}^n} | \dd \left[\eta \left( u_\e\cdot (u_\e-\id) u_\e \right)_{| \mathbb{S}^{n-1}} \right] |^2 \\
& \leq 2\int_{\mathbb{B}^n}| \dd (u_\e-\id)|^2 + C\| u_\e\cdot (u_\e-\id) u_\e\|_{W^{1-\frac12,2}(\mathbb{S}^{n-1}}^2 \\
& \leq 2\int_{\mathbb{B}^n}| \dd (u_\e-\id)|^2 + C\| u_\e\cdot (u_\e-\id) u_\e\|_{W^{1,2}(\mathbb{B}^{n}}^2 \\
& \leq C \| u_\e-\id\|_{W^{1,2}(\mathbb{B}^n)}.
\end{align*}	
	Thus, we obtain
	\begin{align*}
		0\ust{\eps\to 0}{=}&\ \dd E_n(\id).  \pi_{u_\eps}(u_\eps-\id)  \\
		& + \int_{\bB^n}  (n-2)|\dd (\id)|^{n-4} \left[ \dd (\id):\dd(u_\eps-\id)\right] \left[\dd (\id):\dd \pi_{u_\eps}(u_\eps - \id)\right] \\
		& + \int_{\bB^n} |\dd (\id)|^{n-2} \left[\dd(u_\eps - \id):\dd \pi_{u_\eps}(u_\eps - \id)\right] + \smallO \left( \|u_\eps - \id \|_{W^{1,2}(\mathbb{B}^n)}^2 \right)\\
		\ust{\eps\to 0}{=} & \underbrace{\dd E_n(\id). \pi_{u_\eps}(u_\eps-\id)  }_{\eqqcolon I} \\
		& + (n-2)n^{\frac{n-4}{2}} \underbrace{\int_{\bB^n}  \divv(u_\eps-\id) \divv\left( \pi_{u_\eps}(u_\eps - \id) \right)}_{\eqqcolon II} \\
		& + n^{\frac{n-2}{2}} \underbrace{\int_{\bB^n} \left[ \dd(u_\eps - \id):\dd \pi_{u_\eps}(u_\eps - \id)\right]}_{\eqqcolon III} + \smallO \left( \|u_\eps - \id \|_{W^{1,2}(\bB^n)}^2 \right).
	\end{align*}
	
	\textbf{Asymptotic expansion of $I$.} The explicit expression of $I$ is the following:
	\begin{align*}
		I  = \int_{\bB^n} |\dd(\id)|^{n-2} (\dd \id):  \dd \left[ \pi_{u_\eps}(u_\eps-\id)  \right]  = n^{\frac{n-2}{2}} \int_{\bB^n} \divv\left( \pi_{u_\eps}(u_\eps-\id) \right).
	\end{align*}
	We integrate by parts and obtain
	\begin{align*} 
		I &= n^{\frac{n-2}{2}} \int_{\s^{n-1}} \id\cdot \left( \pi_{u_\eps}(u_\eps-\id) \right) \\
		&= n^{\frac{n-2}{2}} \int_{\s^{n-1}} \id\cdot \Big( u_{\eps}-\id -[u_{\eps}\cdot(u_{\eps}-\id)]u_{\eps} \Big)\\
		& = n^{\frac{n-2}{2}} \int_{\s^{n-1}} \Big( \id\cdot u_{\eps} -1 - (\id\cdot u_{\eps})(1-u_{\eps}\cdot\id)\Big)\\
		&= n^{\frac{n-2}{2}} \int_{\s^{n-1}}  \Big( (\id\cdot u_\eps)^2 -1  \Big) .
	\end{align*}
	The following equality holds on $\s^{n-1}$
	\begin{align*}
		 1-\frac{|u_{\eps}-\id|^2}{2}=u_{\eps}\cdot \id.
	\end{align*}
	Thus, we obtain
	\begin{align*}
		(\id\cdot u_\eps)^2 -1 = -|u_{\eps}-\id|^2 + \frac{|u_{\eps}-\id|^4}{4},
	\end{align*}
	and
	\begin{align*}
		I \ust{\eps\to 0}{=} -\left( n^{\frac{n-2}{2}} +\smallO(1) \right) \int_{\s^{n-1}} |u_{\eps}-\id|^2.
	\end{align*}
	
	\textbf{Asymptotic expansion of $II$.} We express the term $II$ as follows:
	\begin{align*}
		II = \int_{\bB^n} |\divv(u_{\eps}-\id)|^2 - \divv(u_{\eps}-\id)\, \divv(\pi_{u_{\eps}}^{\perp}(u_{\eps}-\id)).
	\end{align*}
	We integrate by parts the second term:
	\begin{align*}
		II & = \int_{\bB^n} |\divv(u_{\eps}-\id)|^2 +\int_{\bB^n} \dd\left[\divv(u_{\eps}-\id)\right]\cdot \pi_{u_{\eps}}^{\perp}(u_{\eps}-\id) \\
		& \qquad + \int_{\s^{n-1}} \divv(u_{\eps}-\id)\, [(u_{\eps}-\id)\cdot u_{\eps}]\, (u_{\eps}\cdot\id).
	\end{align*}
	
We notice that, using the definition of \(\pi_{u_\e}^\perp\) in  \eqref{eq:def_proj} and the properties of the extension operator \(\eta\),

\begin{align*}
\left| \int_{\bB^n} \dd\left[\divv(u_{\eps}-\id)\right]\cdot \pi_{u_{\eps}}^{\perp}(u_{\eps}-\id) \right| & \leq \| \dd^2 (u_\e-\id)\|_{L^\infty(\mathbb{B}^n)} \int_{\mathbb{B}^n} |\pi_{u_\e}^\perp (u_\e-\id)| \\
& \leq C\| \dd^2 (u_\e-\id)\|_{L^\infty(\mathbb{B}^n)}\|u_\e\cdot (u_\e-\id)\|_{L^1(\mathbb{S}^{n-1})} \\
& \leq  C\| \dd^2 (u_\e-\id)\|_{L^\infty(\mathbb{B}^n)}\| u_\e-\id\|_{L^2(\mathbb{S}^{n-1})}^2, 
\end{align*}
where we have used that the following equality  on $\s^{n-1}$:
	\begin{align*}
		(u_{\eps}-\id)\cdot u_{\eps} = 1-u_{\eps}\cdot \id = \frac{|u_{\eps}-\id|^2}{2}.
	\end{align*}
	Using again this equality in the third term of \(II\) we obtain
	\begin{align*}
		II &\ust{\eps\to 0}{=} \int_{\bB^n} |\divv(u_\eps - \id)|^2 + \cO \left( \|u_\eps-\id\|_{L^2(\s^{n-1})}^2 \|u_\eps-\id\|_{\C^2(\mathbb{B}^n)}\right) \\
		& \ust{\eps\to 0}{=} \int_{\bB^n} |\divv(u_\eps - \id)|^2 + \smallO \left( \|u_\eps-\id\|_{L^2(\s^{n-1})}^2 \right).
	\end{align*}
	
	\textbf{Asymptotic expansion of $III$.} Using the same arguments as for $II$, we obtain
	\begin{align*}
		III & = -\int_{\mathbb{B}^n}\Delta (u_\e-\id) \cdot \pi_{u_\e}^{\perp} (u_\e-\id) +\int_{\mathbb{S}^{n-1}} \p_\nu (u_\e-\id) \cdot \pi_{u_\e}^\perp (u_\e-\id) \\
		&\ust{\eps\to 0}{=} \int_{\bB^n} |\dd(u_\eps - \id)|^2 + \smallO \left( \|u_\eps-\id\|_{L^2(\s^{n-1})}^2\right).
	\end{align*}
	
	\textbf{Conclusion.}
	We have found that
	\begin{align}\label{eq:conc1}
		\begin{aligned} 
			\int_{\bB^n} \left(  |\dd(u_\eps - \id)|^2 + \frac{n-2}{n}|\divv(u_\eps - \id)|^2 \right) \ust{\eps\to 0}{=} (1+\smallO(1))\int_{\s^{n-1}} |u_{\eps}-\id|^2 .
		\end{aligned} 
	\end{align}
	We now decompose $u_{\eps}-\id$ as in \Cref{lm:Poincare}:
	\begin{align}\label{eq:decompo_ueps}
		u_{\eps}(r,\theta)-\id = \sum_{k=1}^{\infty} \sum_{\ell=1}^{L(k)} f_{k,\ell,\eps}(r)\, Y_{k,\ell}(\theta).
	\end{align}
	By using \eqref{eq:conc1} and applying \Cref{lm:Poincare}, we obtain 
	\begin{multline}\label{eq:identificcation}
	|\mathbb{S}^{n-1}| \int_0^1 |f_0'(r)|^2 r^{n-1} \dd r  \\
	 \quad +\sum_{k\geq 1, 1\leq \ell \leq L(k)} \int_0^1 \left[ (\p_r f_{k,\ell})^2 +\frac{k(k+n-2)}{r^2}f_{k,\ell}^2\right] r^{n-1} \dd r \\
	\quad \quad +\frac{n-2}{n}\int_{\mathbb{B}^n}|\divv (u_\e-\id)|^2 \\
	 \leq \sum_{k\geq 1, 1\leq \ell \leq L(k)} \int_0^1 \frac{1}{k}\left[ (\p_r f_{k,\ell})^2 +\frac{k(k+n-2)}{r^2}f_{k,\ell}^2\right] r^{n-1}\,  \dd r \\ 
	 \quad  +\smallO \left(\|u_\e-\id\|_{L^2(\mathbb{S}^{n-1}}^2\right)
	\end{multline}
	Hence, using the Poincaré inequality in \Cref{lm:Poincare} we find that 
	\begin{align}
		\left\| f_0 + \sum_{k=2}^{\infty} \sum_{\ell=1}^{L(k)} f_{k,\ell,\eps}^i(r)\, Y_{k,\ell}(\theta) \right\|_{W^{1,2}(\bB^n)} & \ust{\eps\to 0}{=}  \smallO\left( \|u_\eps - \id \|_{L^2(\s^{n-1})} \right) \label{eq:strong_High_frequency}\\
		 & \ust{\eps\to 0}{=}  \smallO\left( \|u_\eps - \id \|_{W^{1,2}(\mathbb{B}^n)} \right)
	\end{align}
	where we have used the trace inequality. We also infer from \eqref{eq:identificcation} that 
	\begin{equation}\label{eq:div_negligeable}
	\int_{\mathbb{B}^n}|\divv(u_\eps - \id)|^2  \ust{\eps\to 0}{=} \smallO\left( \|u_\eps - \id \|_{W^{1,2}(\mathbb{B}^n)} \right).
	\end{equation}
	For $k=1$, we have $L(1) = n$ and $Y_{1,\ell} = x^{\ell}$ for $\ell\in\{1,\ldots,n\}$. Thus, we obtain functions $f_{\ell}\colon [0,1]\to \R^n$ such that the following asymptotic expansion holds\footnote{Thanks to the embedding $W^{1,2}(\mathbb{B}^n)\hookrightarrow W^{1-\frac{1}{2},2}(\s^{n-1})$, the embedding $W^{1,2}(\bB^n)\subset L^2(\s^{n-1})$ is compact.}: 
	\begin{align}\label{eq:limit1}
		\begin{aligned} 
			& \frac{u_{\eps} - \id}{\|u_{\eps}-\id\|_{W^{1,2}(\mathbb{B}^n)}} \xrightarrow[\eps\to 0]{} \sum_{\ell=1}^n f_{\ell}(|x|)\, \frac{x^{\ell}}{|x|} \\
			& \qquad \text{ strongly in } W^{1,2}(\mathbb{B}^n) \text{ and strongly in }L^2(\s^{n-1}).
		\end{aligned}
	\end{align}
	Passing to the limit in \eqref{eq:div_negligeable} we find that 
	\begin{align}\label{eq:Int}
		\divv\left( \sum_{\ell=1}^n f_{\ell}(|x|)\, \frac{x^{\ell}}{|x|}  \right) =0 \qquad \text{in }\bB^n.
	\end{align}
Using the divergence theorem we find that for \(\mathcal{H}^{n-1}\)-a.e.\ \(x\in \mathbb{S}^{n-1}\) we have
	\begin{align}\label{eq:BC}
		\sum_{1\leq i,\ell\leq n} f_{\ell}^i(1)\, x^{\ell}\, x_i = 0.
	\end{align}
	We now define the matrix $M^i_j(r) =r^{-1} f^i_j(r)$. From \eqref{eq:BC} and decomposing $M(1)$ into symmetric and antisymetric part $M(1) = M_S + M_{AS}$, we obtain $M_S=0$ since for $x\in \s^{n-1}$, we have
	\begin{align*}
		0 = M(1)x\cdot x = \begin{cases}
			M_S x\cdot x + M_{AS}x\cdot x, \\[2mm]
			x\cdot M_S x - x\cdot M_{AS} x.
		\end{cases} 
	\end{align*}
	Thus, we obtain $x\cdot M_S x=0$ for all $x\in\s^{n-1}$, so that $M_S=0$ and $M(1)\in so(n)$. We now pass to the limit in the weak formulation of the system $\lap_n u_{\eps}=0$. For any $\vp\in \C^{\infty}_c(\bB^n)$, since \(x\mapsto x\) is \(n\)-harmonic, we have
	\begin{align*}
		0 & = \int_{\bB^n} \dd \vp: \left[|\dd u_{\eps}|^{n-2}\, \dd u_{\eps} - |\dd \id |^{n-2} \dd \id \right]  \\
		&= \int_{\mathbb{B}^n} \dd \varphi : \left[  | \dd (\id +(u_\e-\id) |^{n-2} \dd (\id +(u_\e-\id) ) - |\dd \id |^{n-2} \dd \id \right] \\
		&= \int_{\mathbb{B}^n }\dd \varphi : \Big[  (n-2)|\dd \id |^{n-4} \left( \dd \id:\dd (u_\e-\id)\right) \dd \id  \\
		& \quad +|\dd \id|^{n-2} \dd (u_\e-\id) \Big] +\smallO (\|u_\e-\id\|_{W^{1,2}(\mathbb{B}^n)}).
	\end{align*}
	 Dividing by $\|u_{\eps}-\id\|_{W^{1,2}(\bB^n)}>0$ and letting \(\e \to 0\), we obtain
	\begin{align*}
		0 & = \int_{\bB^n} \dd \vp: \left[ n^{\frac{n-2}{2}} \dd(M(r)x) + (n-2)\, n^{\frac{n-4}{2}}  \dd x : \dd (M(r)x)  \dd x \right] \ \dd x \\[2mm]
		& = \int_{\bB^n} \dd \vp : \left[  n^{\frac{n-2}{2}} \dd(M(r)x) + (n-2)\, n^{\frac{n-4}{2}} : \divv (M(r)x)\, \dd x \right] \dd x.
	\end{align*}
	By \eqref{eq:Int}, the second term vanishes. Hence, the map $\left( x\mapsto M(|x|) x \right) \in W^{1,2}(\bB^n)$ verifies the system 
	\begin{align*}
		\begin{cases}
			\lap (M(r) x)=0 & \text{ in }\bB^n, \\[2mm]
			M(r)x = M(1)x & \text{ on }\s^{n-1}.
		\end{cases}
	\end{align*}
	Since the map $x\mapsto M(1)x$ is also harmonic (this is a linear map), we deduce from the uniqueness of the solution of the above system that the constant matrix $M\coloneqq M(1)$ satisfies
	\begin{align*}
		\begin{cases} 
			\forall x\in \bB^n,\qquad M(|x|) x = Mx, \\[2mm]
			M\in so(n).
		\end{cases} 
	\end{align*}
	We now define the map 
	\begin{align*}
		\forall x\in \bB^n,\qquad v_{\eps}(x) \coloneqq e^{-\|u_{\eps}-\id\|_{L^2(\s^{n-1})}\, M } u_{\eps}(x).
	\end{align*}
	Since $M\in so(n)$, the exponential term is a rotation converging to the identity as $\eps\to 0$. Hence, $v_{\eps}$ is also a sequences of free-boundary $n$-harmonic maps verifying \eqref{eq:normalization} and \eqref{eq:low_energy}. By construction of $v_{\eps}$, we have the following expansion in $W^{1,2}(\mathbb{B}^n)$
	\begin{align*}
		v_{\eps} \ust{\eps\to 0}{=} &\ \left(I_n - \|u_{\eps}-\id\|_{L^2(\s^{n-1})}\, M + \cO \left(\|u_{\eps}-\id\|_{L^2(\s^{n-1})}^2\right)\right)  \\[2mm]
		& \qquad \times \left( \id +\|u_{\eps}-\id\|_{W^{1,2}(\mathbb{B}^n)}\, M\, \id + \smallO\left(\|u_{\eps}-\id\|_{W^{1,2}(\mathbb{B}^n)}\right) \right) \\[3mm]
		\ust{\eps\to 0}{=} &\ \id + \smallO\left(\|u_{\eps}-\id\|_{L^2(\s^{n-1})}\right).
	\end{align*}
	Therefore, we have for $\eps>0$ small enough 
	\begin{align*}
		\int_{\s^{n-1}} |v_{\eps}-\id|^2 < \int_{\s^{n-1}} |u_{\eps}-\id|^2.
	\end{align*}
	This contradicts \eqref{eq:minimality}. In other words, we could not assume $\|u_{\eps}-\id\|_{L^2(\s^{n-1})}>0$ in \eqref{eq:limit1}, i.e., $u_{\eps}=\id$ for $\eps>0$ small enough.
\end{proof}

\section{Existence of a free-boundary \(n\)-harmonic map of degree 1: Proof of Theorems \ref{th:main_2}}\label{sec:ProofMain}

In this last section, we prove our main result Theorem \ref{th:main_2}. The idea is to show that, up to a subsequence, the sequence \( (u_k)_{k\geq 1}\) obtained in Theorem\ref{th:existence_nplusalpha} converges to a limit map \(u\) which belongs to \(\I_1\) and which is a critical point of \(E_n\) in \(\I\) in the sense of Definition \ref{def:critical}. The convergence and the possible apparition of bubbles for a sequence of \((n+\alpha_k)\)-harmonic maps with free-boundary in \(\mathbb{S}^{n-1}\) is the object of \cite{MartinoMazowieckaRodiac2025}. 
Thanks to the bubbling theorem \cite[Theorem 1.3]{MartinoMazowieckaRodiac2025} we only need to show that no bubbles of non-zero degree appear in the limiting process. This is obtained by a contradiction argument relying on the following estimate on the energetic level of our \( (n+\alpha_k)\)-critical points.

\begin{proposition}\label{pr:calpharsmallerthansomething}
	
Let $\eps_0=\eps_0(n) = \min\{\eps_M, \eps_c\}$, where $\eps_M$ is from \Cref{th:Energy_gap_Mobius} and $\eps_c$ is from \Cref{le:3.1inMMR} with $\delta=1000\, \|\id\|_{\C^1(\bB^n)}$. There exists \(L>0\) and \(r_0>0\) such that if $\O$ is a \(\C^1\) bounded set of $\R^n$ such that there exists  $\Phi\colon \overline{\O}\to \overline{\bB^n}$ with \(\|\Phi-\id\|_{\C^1}<L\) and if \(0<r<r_0\) then the following holds.
\begin{itemize}
\item[1)] There exist a sequence \( (\alpha_k)_k \subset \R^+\) with \(\alpha_k\to 0\) and a sequence \( (u_k)_k\subset \I_{1,\alpha_k}\) satisfying \eqref{eq:freealphakharmonic}. 

\item[2)]  \begin{equation}\label{eq:calphaestimate}
  E_{n+\alpha_k}(u_k) =c(\alpha_k,r) \le n^\frac{n}{2}|\bB^n| + \eps_0.
 \end{equation}
\end{itemize}
\end{proposition}

We recall that $c(\alpha,r)$ is defined in \eqref{eq:ourc} and we point out that Proposition \ref{pr:calpharsmallerthansomething} is a direct consequence of Theorem \ref{th:existence_nplusalpha} and Proposition \ref{prop:prochegroundstate}. With the estimate \eqref{eq:calphaestimate}, and with the gaps results Lemma \ref{le:3.1inMMR} and Theorem \ref{th:Energy_gap_Mobius} we are able to conclude. However, in order to do that we need to improve the result of of \cite[Theorem 1.3]{MartinoMazowieckaRodiac2025} into a true energy identity and not only a weak energy identity with coefficients \(\lambda_i^*\) in front of the energies of the bubbles being possibly \(>1\). We achieve that in our specific situation by using the Pohozaev identity as in \cite[Lemma 4.1]{Li-Zhu2019}. Since the Pohozaev identity for \( (n+\alpha_k)\)-harmonic maps is not direct because these maps are not in \(\C^2(\overline{\Omega},\bB^n)\), we first discuss the validity of this identity. Then we conclude the proof of Theorem \ref{th:main_2}.

\subsection{Pohozaev identity}
In this section we prove the  Pohozahev identity for free-boundary \(p\)-harmonic maps for \(p\geq 2\). A similar identity, valid for \(p\)-harmonic maps with values into a manifold was proved in \cite[Lemma 2.5]{DuzaarFuchs}.

\begin{proposition}[Pohozaev identity]\label{prop:Pohozaev}
	Let $(\Sigma^n,g)$ be an $n$-dimensional manifold with boundary.	Let \(d\geq 1\) and  $u\colon (\Sigma^n,g)\to \bB^d$ be a free-boundary $p$-harmonic map of class $\C^1$ with $p\geq 2$, in particular \(u(\p \Sigma)\subset \mathbb{S}^{d-1}\) . Consider a boundary point $x_0\in\pl\Sigma$ and boundary coordinates near $x_0$ on $B(0,r)^+$. Then,
	\begin{equation*}
		\int_{\partial B(0,r)^+\cap \R^n_+} |\dd u|^{p-2}\brac{|\partial_\nu u|^2 - \frac 1p |\dd u|^2}\, \dd\mathrm{vol}_g = \frac{p-n}{p\, r}\int_{B(0,r)^+} |\dd u|^p\, \mathrm{dvol}_g.
	\end{equation*}
\end{proposition}

\begin{proof}
First, since in the interior of \(\Sigma\) there is no constraint on the image of \(u\), by convexity of the energy \(E_p\) we see that a \(p\)-harmonic map is a local minimizer of the energy \(E_p\). Hence it is stationary for inner variations. More precisely 
\begin{equation*}
\frac{\dd}{\dd t} \Big\vert_{t=0} \int_{\Sigma} \left| \dd (u\circ\vp_t^{-1})_x \right|^p \dx  =0,
\end{equation*}
for any  smooth family of diffeomorphism \(\vp_t\colon \overline{\Sigma}\rightarrow \overline{\Sigma}\) such that \(\vp_t(x)=x\) on \(\p \Sigma\) and such that 
\(\frac{d}{dt} \Big\vert_{t=0} \varphi_t(x)=\xi(x)\) for any \(x\in \Sigma \), with \(\xi\in \C^\infty_c(\mathring{\Sigma},\mathring{\Sigma})\). This implies that 
\begin{equation}\label{eq:divergence_free_int}
\textrm{div} \left( |\dd u|^p\delta_{ij} -p|\dd u|^{p-2} \p_iu \cdot \p_j u \right) =0 \text{ in } \mathring{\Sigma} \quad \text{ for every } 1\leq i,j\leq n.
\end{equation}
Now we want to prove that, for every \(1\leq i,j\leq n\),
\begin{equation}\label{eq:stationarity_cond_full}
\int_{\Sigma} \left( |\dd u|^p\delta_{ij} -p| \dd u|^{p-2} \p_iu \cdot \p_j u \right) : \dd \xi =0
\end{equation}
for any \(\xi \in W^{1,p}(\Sigma,\R^n)\) such that \(\xi(x) \in T_x \p \Sigma\) for \(\mathcal{H}^{n-1}\)-a.e.\ \(x\in \p \Sigma\). 
By using that \(u\in \C^1(\overline{\Sigma},\mathbb{B}^d)\), combining \eqref{eq:divergence_free_int} and the divergence theorem shows that \eqref{eq:stationarity_cond_full} is equivalent to 
\begin{equation}\label{eq:boundary_1}
\int_{\p \Sigma} \left( |\dd u|^p\id -p| \dd u|^{p-2} \p_iu \cdot \p_j u \right)\xi\cdot \nu =0,
\end{equation}
for all \(\xi \in W^{1-\frac{1}{p},p}(\p \Sigma,\R^d)\) such that \(\xi(x) \in T_x \p \Sigma\) for \(\mathcal{H}^{n-1}\)-a.e.\ \(x\in \p \Sigma \). Since in that case \(\xi\cdot \nu=0\) on \(\p \Sigma\), we have that \eqref{eq:boundary_1} is itself equivalent to 
\begin{equation}\label{eq:boundary_2}
\int_{\p \Sigma} | \dd u|^{p-2} (\p_iu \cdot \p_j u) \xi\cdot \nu = \int_{\p \Sigma } | \dd u|^{p-2} \p_\nu u \cdot \p_{\xi} u= 0,
\end{equation}
where we have denoted \(\p_\xi u(x)=\dd u (x). \, \xi(x)\) for \(x\in \p \Sigma\). We note that since \(u(\p \Sigma)\subset \mathbb{S}^{d-1}\) we have that \(\p_{\xi} u (x) \in T_{u(x)} \mathbb{S}^{d-1}\) for \(\mathcal{H}^{n-1}\)-a.e.\ \(x\in \p \Sigma\). Now, the free boundary condition expresses that 
\begin{equation}
\int_{\Sigma} |\dd u|^{p-2} \dd u : \dd \varphi =0
\end{equation}
for all \(\varphi \in W^{1,p}(\Sigma,\R^n)\) such that \(\varphi(x)\in T_{u(x)}\mathbb{S}^{d-1}\) for a.e.\ \(x\in \p \Sigma\). An integration by parts shows that it implies that 
\begin{equation}
\int_{\p \Sigma} | \dd u|^{p-2} \p_\nu u\cdot \varphi=0
\end{equation}
for such \(\varphi\). It suffices to observe that the choice of \(\varphi=\dd u.\, \xi\) is admissible and implies that \eqref{eq:boundary_2} holds and thus \eqref{eq:stationarity_cond_full} also.

To conclude the proof of Proposition \ref{prop:Pohozaev}, since near \(x_0\in \p \Sigma\) we work in coordinates such that \(\p \Sigma\) is flat and equal to \( \{x_n=0\}\cap B(0,r)\) we can choose \(\xi(x)=(x_1,\dots,x_{n-1},x_n)\) for all \(x\in B(0,r)^+\). On one side we find that 
\begin{align*}
\int_{B(0,r)^+} \left( |\dd u|^p\delta_{ij}-p | \dd u|^{p-2} \p_i u \cdot \p_j u \right) : \dd \xi  = \int_{B(0,r)^+} (n-p)\int_{\Sigma}| \dd u|^p.
\end{align*}
And on the other side, the divergence theorem and \eqref{eq:boundary_2} show that 
\begin{align*}
\int_{B(0,r)^+} \left( |\dd u|^p\delta_{ij} -p| \dd u|^{p-2} \p_i u \cdot \p_j u \right) : \dd \xi  = r\int_{\p B(0,r)^+\cap \R^n_+} (|\dd u|^p-p|\dd u|^{p-2} |\p_r u|^2 ).
\end{align*}
\end{proof}

\subsection{Proof of Theorem \ref{th:main_2}}

We are now ready to conclude the proof of our main Theorem \ref{th:main_2}.

\begin{proof}[Proof of \Cref{th:main_2}]
Consider $\O$ as in \Cref{pr:calpharsmallerthansomething}.

Proposition \ref{pr:calpharsmallerthansomething} provides a sequence  $(u_{k})_{k}$ of free boundary $(n+\alpha_k)$-harmonic maps from $\Omega\to \bB^n$ with degree 1 with $E_{n+\alpha_k}(u_{k}) = c(\alpha_k,r)$.

 By \cite[Theorem 1.3]{MartinoMazowieckaRodiac2025}, there exist free boundary $n$-harmonic maps $u_0\colon \Omega\to \bB^n$ and $\omega_1,\ldots,\omega_K\colon \bB^n\to \bB^n$ and numbers $\lambda_1^*,\ldots,\lambda_K^*\in[1,+\infty)$ such that
\begin{align}\label{eq:energy_identity}
	\lim_{k\to +\infty} c(\alpha_k,r) = \lim_{k \to +\infty}\int_{\Omega}|\dd u_{\alpha_k}|^{n+\alpha_k} = \int_{\Omega}|\dd u_0|^n +\sum_{j=1}^K \lambda_j^* \int_{\bB^n} |\dd \omega_j|^n .
\end{align}
Moreover, by \Cref{pr:validity_minmax} and \Cref{pr:calpharsmallerthansomething} we have, for \(k\) large enough
\[
	c_1(\alpha_k,r)\leq \frac{1}{2}\left( n^{\frac{n}{2}}\, |\bB^n| + c_2 \right) <c_2  \leq c(\alpha_k,r) \le n^{\frac n2} |\bB^n| + \eps_0.
\]
Letting \(k\to +\infty\) we obtain by \Cref{lm:c1lemma}, \Cref{pr:calpharsmallerthansomething},  and \eqref{eq:energy_identity} that
\begin{equation}\label{eq:energyinequalities}
 	n^{\frac{n}{2}}\, |\bB^n|  < c_2 \leq  \lim_{k \to +\infty }\int_{\Omega}|\dd u_{k}|^{n+\alpha_k} = \int_{\Omega}|\dd u_0|^n +\sum_{j=1}^K \lambda_k^* \int_{\bB^n} |\dd \omega_j|^n \le n^{\frac n2} |\bB^n| + \eps_0.
\end{equation}
By \cite[Proposition 1.4]{MartinoMazowieckaRodiac2025}, we also have
\begin{equation}\label{eq:quant_degrees}
	1 = \deg(u_0,\pl\Omega) + \sum_{j=1}^K \deg(\omega_j,\s^{n-1}).
\end{equation}
Thanks to \eqref{lowerbound2}, we deduce from \eqref{eq:energy_identity} and \Cref{pr:calpharsmallerthansomething} that 
\begin{equation}\label{eq:upper_degrees}
	n^{\frac{n}{2}}\, |\bB^n|\left( \left| \deg(u_0,\pl\Omega) \right| + \sum_{j=1}^K \left| \deg(\omega_j,\s^{n-1}) \right| \right) \leq n^{\frac{n}{2}}\, |\bB^n| + \eps_0.
\end{equation} 
Combining \eqref{eq:quant_degrees} and \eqref{eq:upper_degrees}, we deduce that all the degrees of $u$ and the $\omega_j$ are either $0$ or $1$.

If $\deg(u_0,\pl\Omega)=1$, then Theorem \ref{th:main_2} is proved.

If $\deg(u_0,\pl\Omega)=0$, then one of the bubbles has degree $1$, we assume that it is \(\omega_{j_0}\) for some fixed \(j_0\in \{1,\dots,K\}\). Thanks to the inequality $\lambda_{j_0}^*\geq 1$ and the choice of $\eps_0$, we deduce from \eqref{eq:energyinequalities} that 
\[
	\int_{\bB^n} |\dd \omega_{j_0}|^n \leq n^{\frac{n}{2}}\, |\bB^n| + \eps_0 \leq n^{\frac{n}{2}}\, |\bB^n| + \eps_M.
\]
By Theorem \ref{th:Energy_gap_Mobius}, we obtain that $\omega_{j_0}$ is a Möbius map. Hence, we have $E_n(\omega_{j_0})=n^{n/2}\, |\bB^n|$ and \eqref{eq:energyinequalities} implies that 
\[
	n^{\frac{n}{2}}\, |\bB^n| + E_n(u_0) + \sum_{\ell\neq j_0} E_n(\omega_{\ell}) \leq n^{\frac{n}{2}}\, |\bB^n| + \eps_c. 
\]
This implies that 
\[
	E_n(u_0) + \sum_{\ell\neq j_0} E_n(\omega_{\ell}) \leq \eps_c.
\]
Thanks to \Cref{le:3.1inMMR} (see also Lemma 3.1 \cite{MartinoMazowieckaRodiac2025}), we obtain $K=j_0=1$ and $u_0=cst$.

The number $\lambda^*_1$ is given by $\lambda^*_1 = \liminf_{k \to +\infty}\lambda_{1,\alpha_k}^{-\alpha_k}$, with $\lambda_{1,\alpha_k}$ the speed of concentration of the bubble $\omega_1$, cf.\ \cite[Claim 5, p. 22]{MartinoMazowieckaRodiac2025}. By choosing an appropriate chart sending the boundary of $\O$ into $\R^{n-1}\times \{0\}$, we can assume that the point of concentration for $\omega_1$ is $0$. Then, by using the Pohozaev identity in Proposition \ref{prop:Pohozaev}, for any given $\delta>0$ and \(k\) large enough,
\begin{align*}
	 & \int_{\lambda_{1,\alpha_k}}^{\delta} \left( \frac{\alpha_k}{(n+\alpha_k)\, r} \int_{B(0,r)^+} |\dd u_{k}|^{n+\alpha_k}\right)\, \dd r \\
	 & = \int_{\B(0,\delta)^+\setminus B(0,\lambda_{1,\alpha_k})^+} |\dd u_{\alpha_k}|^{n+\alpha_k-2}\left( |\pl_r u_{\alpha_k}|^2 - \frac{|\dd u_{\alpha_k}|^2}{n+\alpha_k} \right).
\end{align*}
By no-neck energy property (see Theorem 1.3 in \cite{MartinoMazowieckaRodiac2025}), it holds
\begin{align}\label{eq:noneck}
	\limsup_{\delta\to 0}\ \limsup_{k \to +\infty} \int_{\B(0,\delta)^+\setminus B(0,\lambda_{1,\alpha_k})^+} |\dd u_k|^{n+\alpha_k-2}\left( |\pl_r u_{k}|^2 - \frac{|\dd u_{k}|^2}{n+\alpha_k} \right)= 0.
\end{align}
On the other hand, we have
\begin{align*}
	 &\int_{\lambda_{1,\alpha_k}}^{\delta} \left( \frac{\alpha_k}{(n+\alpha_k)\, r} \int_{B(0,r)^+} |\dd u_{\alpha_k}|^{n+\alpha_k}\right)\, \dd r \\
	 & \geq \int_{\lambda_{1,\alpha_k}}^{\delta} \left( \frac{\alpha_k}{(n+\alpha_k)\, r} \int_{B(0,\lambda_{1,\alpha_k})^+} |\dd u_{\alpha_k}|^{n+\alpha_k}\right)\, \dd r \\[2mm]
	& \geq \int_{\lambda_{1,\alpha_k}}^{\delta} \left( \frac{\alpha_k}{(n+\alpha_k)\, r} \, \frac{1}{\lambda_{1,\alpha_k}^{\alpha_k}}\int_{B(0,1)^+} |\dd u_k(\lambda_{1,k}\cdot)|^{n+\alpha_k}\right)\, \dd r \\[2mm]
	& \geq \frac12 \int_{\lambda_{1,\alpha_k}}^{\delta} \left( \frac{\alpha_k}{(n+\alpha_k)\, r} \, \int_{B(0,1)^+} |\dd u_k(\lambda_{1,k}\cdot)|^{n+\alpha_k}\right)\, \dd r \\[2mm]
	& \geq \frac12 \left(\int_{B(0,1)^+}|\dd u_k(\lambda_{1,k}\cdot)|^{n+\alpha_k}\right)\, \frac{\alpha_k}{n+\alpha_k}\, \log\left(\frac{\delta}{\lambda_{1,\alpha_k}}\right).
\end{align*}
In the second last inequality we have used that \(\lambda_{1,k} \leq M\) for some \(M>0\) (cf.\ \cite[Claim 5, p. 22]{MartinoMazowieckaRodiac2025}) and hence \(\lambda_{1,k}^{-\alpha_k}\geq 1/2\) for \(k\) large enough. Up to a subsequence, we can assume that 
\begin{equation}
\lim_{k \to +\infty}\int_{B(0,1)^+}|\dd u_k(\lambda_{1,k}\cdot)|^{n+\alpha_k}=\int_{B(0,1)^+}| \dd \omega_1|^n=n^{\frac{n}{2}}|\mathbb{B}^n|.
\end{equation}
Combing back to \eqref{eq:noneck}, we obtain
\begin{align*}
	0 & = \limsup_{\delta\to 0}\ \limsup_{k \to +\infty} \frac{\alpha_k}{n+\alpha_k}\, \log\left(\frac{\delta}{\lambda_{1,k}}\right) \\[2mm]
	& =\limsup_{\delta\to 0}\ \limsup_{k \to +\infty}  \left[ \log \left( \delta^{\alpha_k/(n+\alpha_k)}\right)-\frac{\alpha_k}{n+\alpha_k}\log\lambda_{1,k} \right] \\[2mm]
	& = \frac{1}{n}\ \limsup_{k \to +\infty}\ \log\left(\frac{1}{\lambda_{1,k}^{\alpha_k}}\right) \\[2mm]
	& = \frac{1}{n}\, \log\left(\lambda_1^*\right).
\end{align*}
Consequently, we have $\lambda_1^*=1$. We obtain a contradiction by combining \eqref{eq:energyinequalities} and the equality $E_n(\omega_1) = n^{n/2}\, |\bB^n|$. Hence, we must have $\deg(u_0,\pl\Omega)=1$.
\end{proof}

\bibliographystyle{abbrv}%
\bibliography{bib}%

\end{document}